\documentclass[11pt, oneside]{amsart}

\title[Small data non-linear wave equation numerology]{\large Small data non-linear wave equation numerology \\\small  The role of asymptotics}

\author{Istvan Kadar${}^{{*}}$}
\thanks{$^{*}$\texttt{ik338@cam.ac.uk} \\
	\phantom{1 } Centre for Mathematical Sciences, Wilberforce Road, Cambridge, CB3 0WA, UK}

\usepackage[T1]{fontenc}
\usepackage[utf8]{inputenc}
\usepackage[english]{babel}
\usepackage{lmodern}
\usepackage[margin=3cm]{geometry}

\usepackage{hyperref}
\usepackage{enumitem}
\usepackage{tikz-cd} 	
\usepackage{mathtools}
\usepackage{alltt}
\usepackage{amsfonts}
\usepackage{amsmath}
\usepackage{amssymb}
\usepackage{amsthm}
\usepackage{enumitem}
\usepackage{tikz-cd} 	
\usepackage{graphicx}
\usepackage{caption}
\usepackage{subcaption}
\usepackage{aliascnt}
\usepackage{hyperref}
\usepackage{cases}
\usepackage{comment}	
\usepackage[capitalize,nameinlink]{cleveref}
\usepackage{tikz}
\usetikzlibrary{tikzmark}
\usepackage{slashed}
\usepackage[colorinlistoftodos]{todonotes}

\crefformat{equation}{(#2#1#3)}
\numberwithin{equation}{subsection}

\theoremstyle{definition}

\newtheorem*{theorem*}{Theorem}
\newtheorem*{idea}{Idea}

\newtheorem*{conjecture*}{Conjecture}

\theoremstyle{plain}
\newtheorem{theorem}{Theorem}[section]
\newtheorem{corollary}{Corollary}[theorem]
\newtheorem{lemma}[theorem]{Lemma}
\newtheorem{prop}[theorem]{Proposition}

\newtheorem{cor}{Corollary}[theorem]

\theoremstyle{definition}
\newtheorem{definition}[theorem]{Definition}

\theoremstyle{remark}
\newtheorem{remark}[theorem]{Remark}

\crefname{lemma}{Lemma}{Lemmas}
\crefname{prop}{Proposition}{Proposition}
\crefname{conjecture}{Conjecture}{Conjecture}
\crefname{cor}{Corollary}{Corollary}
\crefname{remark}{Remark}{Remark}
\crefname{defi}{Definition}{Definition}

\newcommand{\R}{\mathbb{R}}

\newcommand{\scri}{\mathcal{I}}

\newcommand{\abs}[1]{\left\lvert #1\right\rvert}

\newcommand{\jpns}[1]{\langle #1 \rangle}

\newcommand{\norm}[1]{\left\lVert #1\right\rVert}

\renewcommand{\d}{\mathrm{d}}

\DeclareMathOperator{\supp}{supp}

\makeatletter
\newenvironment{taggedsubequations}[1]
{%
\addtocounter{equation}{-1}%
\begin{subequations}%
	\def\@currentlabel{#1}%
	\renewcommand{\theequation}{#1.\arabic{equation}}%
}
{\end{subequations}}
\makeatother
\usepackage{diagbox}
\usepackage[autostyle]{csquotes}

\begin{document}
	\maketitle
	\setcounter{tocdepth}{1}
	\begin{abstract}
		Systems of wave equations may fail to be globally well posed, even for small initial data. Attempts to classify systems into well and ill-posed categories work by identifying structural properties of the equations that can work as indicators of well-posedness. The most famous of these are the null and weak null conditions. As noted by Keir,  related formulations may fail to properly capture the effect of undifferentiated terms in systems of wave equations. We show that this is because null conditions are good for categorising behaviour close to null infinity, but not at timelike infinity. In this paper, we propose an alternative condition for semilinear equations that work for undifferentiated non-linearities as well. We illustrate the strength of this new condition by proving global well and ill-posedness statements for some systems of equation that are not \textit{critical} according to the our classification. Furthermore, we given two examples of systems satisfying the weak null condition with global ill-posedness due to undifferentiated terms, thereby disproving the weak null conjecture as stated in \cite{deng_global_2018}.
		
	\end{abstract}
	\tableofcontents
	\section{Introduction}
	\subsection{Overview}
	In this paper, we consider the small data global existence for \textit{semi-linear} system of wave ($m=0$) and Klein-Gordon ($m\neq0$) equations
	\begin{equation}\label{generic model}
		\begin{gathered}
			(\Box-m_i)w_i:=(\partial_t^2-\Delta+m_i)w_i=F_i(w,\partial w)\\
			w:\R^{n+1}\to\R^m,
		\end{gathered}
	\end{equation}
	where $\Delta=\sum_j\partial_j\partial_j$ is the Laplacian in $n$ dimensions. We will recall some motivational results in the general case, then change focus to the $m=0$, where after analysing previous results, we present a unified approach to such problems.
	
	The main point will be to understand how the global ill/well-posedness depends on the asymptotics of the linear inhomogeneous problem:
	\begin{equation}\label{linear equation}
		\begin{gathered}
			\Box w=F(x,t).
		\end{gathered}
	\end{equation}
	Among other non-linearities, we find an optimal condition (see \cref{n variable Strauss condition}) on the powers $c_{ij}$ for the system
	\begin{equation*}
		\begin{gathered}
			\Box w_i=\sum_j a_{ij}\abs{w_j}^{c_{ij}},
		\end{gathered}
	\end{equation*}
	to be globally well-posed.
	
	In the past 15 years, there's been an explosion of works on system of wave Klein-Gordon equations. Such systems are of great physical interest, here we only mention the culmination of these efforts that produced the proof of stability of Minkowski space with a massive scalar field \cite{lefloch_global_2017} \cite{ionescu_einstein-klein-gordon_2022}. For more references see these two. The plethora of conditions on the non-linearities and corresponding results can appear daunting and it is certainly hard to judge when should one expect stability of the zero solution. An \textit{algebraic} or a \textit{simplified} condition on the non-linearities -- similar to the weak null condition -- that suggests global well posedness is a goal of this line of research. This appears as a rather difficult task, here we restrict ourselves to the purely wave part of the problem, that is $m_i=0$ for the rest of the paper.
	
	One motivating line of research for the author starts with \cite{lindblad_weak_2003}. In special wave systems -- with dominant behaviour happening near null infinity ($\scri$) -- a \textit{general} requirement on the non-linear structure for global existence was first formulated under the name of weak null condition by analysing a reduced system.  In particular, one only needs to understand the solution of a 2 dimensional PDE or an ODE, that corresponds to the behaviour of the system towards $\scri$. One of the main legacies of this discovery was to find the \textit{structural} difference between the $(\partial_t\phi)^2$ and $\partial\phi\cdot\partial\phi$ nonlinearities,  which allows to quickly determine if one should expect global existence. However, we note here, that this condition was never explicitly used in the works of the above mentioned authors, its purpose was purely motivational. For further developments on the weak null condition see \cite{alinhac_semilinear_2006},\cite{keir_weak_2018}, \cite{deng_global_2018} and \cite{keir_global_2019} and references therein. In this work, we discuss a similar structural property for equations with undifferentiated components that could also lead to new insight for anisotropic equations (see \cref{anisotropic section}).
	
	\begin{idea}\label{idea}
		Roughly speaking the method is as follows (for more details see \cref{role of asymptotics}-\cref{finding critical exponents}). Instead of solving the nonlinear problem  \eqref{generic model} with $m_i=0$, consider the iterative linear problems
		\begin{equation}\label{generic iterative model}
			\begin{gathered}
				\Box w^{(n)}_i=F_i(w^{(n-1)},\partial w^{(n-1)}).
			\end{gathered}
		\end{equation}
		If there exists $N$ such that for all $n>N$ the leading order behaviour of the solutions to \eqref{generic iterative model} are the same at infinity, then the corresponding nonlinear problem is globally well-posed. The notion of infinity is crucial, and we will see that it has two components\footnote{in a multispeed, ie. anisotropic case, infinity has even more components \cref{anisotropic section}}.
	\end{idea}

	\begin{definition}\label{asymptotic decay condition}
		A system of wave equation \cref{generic model} ($m_i=0$) is said to satisfy the \textit{asymptotic decay condition} if there exists $N$ such that for all $m>N$ the solutions $w^{(m)}_i$ have the same decay rate toward future timelike ($I^+$) and null infinity ($\scri$) as $w^{(N)}_i$.
	\end{definition}
	
	As it was done originally for the weak null condition, we do not prove a general  theorem for a system satisfying our definition, we merely use it to predict what the exponents for a specific example should be and prove it using unrelated methods. Furthermore, we also analyse in \cref{finding critical exponents} many examples of \cref{generic model} from existing literature and show that global well posedness completely overlaps with the \textit{asymptotic decay condition}.
	
	To compare our method with the weak null condition, consider $\Box\phi=\abs{\partial_t\phi}^q$ (see \cref{Glassey}). The reduced system yields the ODE $\partial_s\Psi=s^{1-q}\abs{\Psi}^q$, which has no non-trivial global solutions for $q\leq2$, showing that $q=2$ is a borderline case. Alternatively, one can say, that the iterative solutions $\partial_s\Psi^{(n)}=s^{1-q}\abs{\Psi^{(n-1)}}^q$ have the same asymptotic behaviour as $s\to\infty$ (they stay bounded) only for $q>2$. Indeed, this second approach also gives the critical exponent for the equation $\Box\phi=\abs{\phi}^q$ (\cref{Strauss}) which is harder to understand via the reduced system, for more details see \cref{comparison to weak null}.
	
	\begin{remark}
		Many different numerologies arise in the study of PDEs. In case of Sobolev embeddings, checking the exponents is easy via testing on functions of the form $x^\alpha$. Next, there are the Strichartz estimates, which also have an intuitive explanation in the homogeneous case via interpolation (at least for the Schrodinger and wave equations). Already here, the strict inhomogeneous case was proved later  (see \cite{harmse_lebesgue_1990}, \cite{foschi_inhomogeneous_2005}), with surprising numerology arising. Similarly, optimal decay for geometric equations, such as the wave equation, can have intuitive reasons (see \cite{anderson_global_2021}). Moving away from linear problems, there are the critical Sobolev spaces for scaling invariant PDEs. These exponents provide a very important distinction for potentially very different behaviour (see \cite{tao_low-regularity_2002} for a great intuitive exposition, or \cite{tao_nonlinear_2006} for more detail). Finally, let us mention the different speeds of blow-up for non-linear evolution equations. In this case, no universal heuristic is available at the moment, but understanding the possible speeds of blow-ups and their relation to the parameters in the equation is an exciting area a research. The possible explanations include usage of matching conditions (see eg. \cite{berg_formal_2003}), spectral problems (see \cite{herrero_explosion_1994}) or shooting problem for ODEs (see \cite{merle_blow_2019}).  These motivational ideas don't always work (see \cite{lindblad_sharp_1993} for failure of local well-posedness down to scale invariance), but they provide an important measuring stick against which to measure one's progress and provide a guide for the otherwise scattered results.
	\end{remark}

	Before outlining the rest of the introduction, we emphasis that the aim of this study is to make a connection between an (easy) a priori algebraic calculation for the non-linear system  (without understanding the possible \textit{fine} structure of the equation) and global well-posedness. An argument for the connection is the main idea in this paper and is contained in \cref{role of asymptotics} and \cref{finding critical exponents}, however we provide no definite proof of this correspondence. This is in part due to regularity issues (see \cref{regularity vs decay}). Therefore we content ourselves with some special forms of $F$ that still highlight the use of such a heuristic framework. This is contrasted to the work of Keir \cite{keir_weak_2018}\footnote{see \cref{null condition} for a comparison with his hierarchical condition}, where he treats very general non-linearities in one fell swoop. Finally, we note that we find a novel part of the critical curve for a system considered in \cite{hidano_combined_2016} (see \cref{strauss glassey main theorem}).
	
	The introduction is structured as follows. In \cref{comparison to weak null}, we review the weak null condition, than in \cref{fine structures in non-linearities}, we explain what we mean by fine structure of the non-linearity, and compare the nonlinear wave equation to nonlinear Klein-Gordon and Schroedinger equations. In \cref{zoology}, we present state of the art understanding of critical exponents via examples found in the literature. These will naturally lead to a digression about regularity of the nonlinearity which we address in \cref{regularity vs decay}. Finally, we present the main results in \cref{main results}.

	\subsection{Comparison to weak null condition}\label{comparison to weak null}
	The discussion here closely follows the introduction of \cite{deng_global_2018} and chapter 5 of \cite{keir_weak_2018}.
	
	For the system \cref{generic model} ($m_i=0$) one can associate a reduced system, by using the ansatz
	\begin{equation*}
		\begin{gathered}
			w_i(t,x)=\frac{\epsilon}{r}W_i(q,s,\omega),\qquad q=t-\abs{x},\, s=\epsilon\log t,\, \omega=\frac{x}{\abs{x}}
		\end{gathered}
	\end{equation*}
	and throwing away terms of order $\epsilon^3$. A system is said to satisfy the weak null condition (WNC) if this reduced system has global solution from small initial data. Note, that $\partial_v$ derivatives introduce extra factors of $\epsilon$, thus all non-linear terms with $\partial_v$ derivatives are disregarded. This is motivated by the fact that such terms decay faster toward $\scri$. As discussed under conjecture 5.3.1 \cite{keir_weak_2018}, it is naive to suppose that this condition is sufficient to guarantee global well-posedness. Indeed, one can show (Proposition 5.4.1 \cite{keir_weak_2018}) that the system
	\begin{equation*}
		\begin{gathered}
			\Box\phi_1=0\\
			\Box\phi_2=\partial_t\phi_2\partial_t\phi_1\\
			\Box\phi_3=\partial_t\phi_2\partial_t\phi_3
		\end{gathered}
	\end{equation*} 
	admit solutions that grow exponentially towards $\scri$, even from small initial data. This exponential growth implies linear instability, in particular, additional nonlinear terms cannot be treated perturbatively. It is also discussed in the above work, that nonlinear perturbations of such systems are unlikely to admit global solutions. This motivates the less naive version
	\begin{conjecture*}(Weak null conjecture, 5.3.2 \cite{keir_weak_2018})
		
		If the asymptotic system corresponding to a system of wave
		equations containing only differentiated semilinear terms admits global solutions for sufficiently small initial data, and if those solutions grow no faster than $r\partial_v\phi\sim r^{c\epsilon}$, then the system of wave equations also admits global solutions for sufficiently small initial data.
	\end{conjecture*}
	
	A similar conjecture also appears in \cite{deng_global_2018}, without the restriction for derivative semilinear non-linearities, but the authors restrict the system to only include quadratic terms, even though cubics are usually assumed to be of less relevance (see section 1.2 of \cite{keir_weak_2018}).
	
	In his exhaustive treatment of hierarchical weak null systems Keir shows that under an additional condition to the WNC the conjecture is true, moreover each field decays at least like $\abs{w_i}\leq t^{-1+\epsilon_i}$, for $\epsilon_i$ small depending on the place of $w_i$ in a predetermined hierarchy. Once, he establishes semi-global existence, ie. existence in bounded retarded time, the solution is propagated easily towards timelike infinity ($i^+$), as any partial derivative yields extra decay towards $i^+$. It's important to contrast this to the presence of undifferentiated nonlinear terms.
	
	Consider the following two systems of equations
	\begin{equation*}
		\begin{aligned}
			&\Box\phi_1=0&&&&\Box\eta_1=0\\
			&\Box\phi_2=(\partial_t\phi_1)^2&&&&\Box\eta_2=\eta_1^2\\
			&\Box\phi_3=(\partial_t\phi_2)^2&&&&\Box\eta_3=\eta_2^2.
		\end{aligned}
	\end{equation*}
	It is shown in Keir's work, that $\phi_i$ have $L^\infty$ estimates of the form $t^{-1+\epsilon_i}$, with each additional derivative yielding extra decay towards $i^+$, ie. extra factor of $\frac{1}{u}$. Indeed, this assumption can be used iteratively to show that the leading order decay towards $i^+$ is generated by the behaviour of the forcing near $\scri$. This is unlike the situation for $\eta_i$. $\eta_2$ behaves similarly as $\phi_2$, but due to the lack of derivatives, $\eta_2^2\sim u^2 (\partial_t\phi_2)^2$. This additional $u^2$ term leads to a growth for $\eta_3$ towards $i^+$, and indeed if the sequence is iterated further, the growth can be arbitrarily fast polynomial rate. The same way as exponential growth towards $\scri$ causes serious problems for nonlinear perturbations, the same holds here, see \cref{main results}. Coupling this growth nonlinearly via $\partial_v$ derivatives may result in global ill-posedness, while still satisfying WNC. Indeed, we show that the form of the weak null conjecture as stated in \cite{deng_global_2018} is false, see \cref{main results}.

	\subsection{Critical exponents and fine structure}\label{fine structures in non-linearities}	
	 There is a heuristic analysis that one can perform for the Klein-Gordon equation which suggests that $(\Box-1)\phi=\abs{\phi}^q$ is globally well-posed for small data if $q>q_c=1+2/n$ (\cite{keel_small_1999}). This is understood as follows: the usual energy estimate is roughly
	\begin{equation*}
		\begin{gathered}
			\norm{\phi(t)}_{H^{n}}\leq\norm{\phi(0)}_{H^n}+\int_0^t\d s \norm{\abs{\phi}^q}_{H^{n-1}}\leq\norm{\phi(0)}_{H^n}+\int_0^t\d s\, \norm{\phi}_{L^\infty}^{q-1}\norm{\phi}_{H^{n-1}}.
		\end{gathered}
	\end{equation*}
	One could close a bootstrap if $\norm{\phi}_{L^\infty}^{q-1}\lesssim t^{-1-\epsilon}$. Assuming the same decay as for the linear equation, we arrive at $q_c=1+2/n$\footnote{Note, that the same analysis would yield $1+2/(n-1)$ for the same non-linearity for the wave equation, which is \text{false}.}. Indeed the global existence part of the heuristic has been proved in many cases (see references in \cite{keel_small_1999}), and is very robust to the structure of the non-linearity, it only uses its leading order behaviour. Global ill-posedness is more delicate, as one may have obstructions, such as a coercive energy or simply a Hamiltonian structure. We illustrate this subtlety with the following examples.
	\begin{itemize}
		\item The KG equation $(\Box-1)\phi=\phi^2$ is globally well-posed \cite{ozawa_global_1996}, with the proof using absence of space-time resonances via a normal form transformation. This is a structure in the nonlinearity beyond its leading order behaviour as is shown by the blow up solutions of \cite{keel_small_1999}, which also have quadratic growing non-linear terms.
		\item For the nonlinear Schrodinger equation, the critical exponent analysis is the same as for KG above. Due to the conserved quantities, one cannot have solutions that blow up for power nonlinearities in the predicted regime ($p<1+2/n$), in particular, for pure power non-linearities with $U(1)$ symmetry all solutions are global in the mass subcritical range. Breaking the $U(1)$ symmetry, eg. replacing $\abs{u}^{p-1}u$ nonlinearity with $\abs{u}^p$  can lead to finite time blow up from small data \cite{ikeda_small_2012}.
		\item In the simplest case for the wave equation, this fine structure occurs for focusing/defocusing classification in case of power non-linearity ${\Box w=\pm w\abs{w}^{p-1}}$. In the defocusing case, one  can use low regularity existence with the conserved quantity to get global solutions, while in the focusing case, the energy is not coercive and indeed small data solutions may blow up for $p$ sufficiently small. Indeed, the explanation for such powers is the main point of this work.
		\item 	A more involved example for waves, might be the system
		\begin{equation*}
			\begin{gathered}
				\Box\phi_1=(\partial_t\phi_2)^3\\
				\Box\phi_2=\alpha(\partial_t\phi_1)^3
			\end{gathered}
		\end{equation*}
		in $\R^{2+1}$. The asymptotic system of the weak null condition is 
		\begin{equation*}
			\begin{gathered}
				\partial_v\Phi_1=\frac{1}{v}\Phi_2^3\\
				\partial_v\Phi_2=\frac{1}{v}\alpha\Phi_1^3,
			\end{gathered}
		\end{equation*}
		which has bounded solution only for $\epsilon<0$. Therefore the heuristic of the weak null condition predicts that only $\epsilon\leq0$ system has global solutions. A similar system of 3 waves in $\R^{3+1}$, was studied in detail by Keir \cite{keir_global_2019}.
	\end{itemize}
		
	Finally, let us mention, that singularity formation from small initial data requires much structure from the equation. All ill-posedness results quoted in this paper use the Kato mechanism \cite{keel_small_1999}, which uses positivity for integrated quantities, closely related to the Virial type argument. Another two quasilinear mechanisms are: shock formation for fluid equations and the short pulse method in general relativity. Indeed, this is the reason why the quadratic counter example to the weak null conjecture is essentially a linear problem, as we found no non-linear mechanism consistent with WNC to produce blow-up.
		
	\subsection{The zoology}\label{zoology}
	In this section, we recall some of the works concerning global well-posedness of nonlinear wave equations. We will work with non-linearities ($F$) that are not smooth to observe a large range of possible equations, but there is a price to pay, since this introduces a novel problem regarding regularity, see \cref{regularity vs decay}. Indeed, some of the subtle transitions are missed if one works with only smooth nonlinearities.
	
	In his pioneering work \cite{john_blow-up_1979}, Fritz John showed that the equation
	\begin{equation}\label{Strauss}
		\begin{gathered}
			\Box\phi=\phi\abs{\phi}^{q-1}
		\end{gathered}
	\end{equation}
	for $\phi:\R^{3+1}\to\R$ the $\phi=0$ solution is not stable\footnote{A stronger result holds: the equation has no nontrivial global solutions} for $q<q_c=1+\sqrt{2}$, but stable for $q>q_c$. Here, we are not concerned with the behaviour at $q=q_c$, but see \cite{wang_recent_2012} for references in that case. Subsequently it was conjectured by Strauss (\cite{strauss_nonlinear_1981}) that the same result holds in general dimensions with critical power $q_c(n)$ the positive root of $(n-1)q_c^2-(n+1)q_c-2=0$. The conjecture was fully resolved for $n=2$ by Glassey
	\cite{glassey_existence_1981} \cite{glassey_finite-time_1981}, and mostly resolved for $n\geq4$ \cite{georgiev_weighted_1997}, \cite{lindblad_long-time_1996} with many works preceding \footnote{For a comprehensive literature review see \cite{wang_recent_2012}}. 
	
	The exact value of $q_c$ is somewhat puzzling and the only work known to the author that gives an intuitive reason for it is \cite{tao_johns_2007}, where the $n=3$ flat background case is analysed. The purpose of this paper is to initiate a systematic understanding of the critical powers in equations of the form \eqref{generic model}. 
	
	There are natural generalisations to the above problem. One may introduce obstacles in the interior of spacetime, modify the geometry to slightly relax Minkowski space to almost flat near infinity (with many possible definitions \cite{baskin_asymptotics_2018}, \cite{sogge_concerning_2010}) or introduce drastic change such as a black hole \cite{lindblad_strauss_2014}. Importantly, these changes do not change the value of $q_c$, for partial resolutions of the conjecture in these settings see references in \cite{wang_recent_2012}.
	
	\begin{remark}[Low frequency nature]
		All these generalisations depend on two crucial, but separate properties of the equations (see a similar discussion in \cite{lindblad_strauss_2014}). First, a decay estimate for the solution that depends on the far away region being asymptotically flat (with various notions available in literature). Second, an integrated decay (Morawetz) estimate in the interior. Indeed, such statement can be proven for low frequency part of the solution under very mild restrictions on the spacetime under consideration (see \cite{vasy_morawetz_2013}, \cite{moschidis_logarithmic_2016}). Indeed, this part of the solution provides the decay that will be important in our work. To contrast this, the high frequency part sees the finer structure of the geometry and involves more care, \footnote{moreover, this sensitivity applies also if one modifies the equation with short range potential that fall-off sufficiently fast in $r$, for details see \cite{moschidis_rp_2016}} indeed quantitative mode stability for Kerr black holes \cite{shlapentokh-rothman_quantitative_2015} or the lack of such statement for black strings \cite{benomio_stable_2021} are highly non-trivial statements. Our results would be only applicable to different geometric settings where the decay of the low frequency part is the slowest.
	\end{remark}

	An alternative way to generalise the problem, is to introduce system of wave equations, such as 
	\begin{equation}\label{2 variable Strauss}
		\begin{gathered}
			\Box\phi=\psi\abs{\psi}^{q_1-1}\\
			\Box\psi=\phi\abs{\phi}^{q_2-2}.
		\end{gathered}
	\end{equation}
	In this case, there is no longer a critical value, but a critical curve (\cite{del_santo_global_1997})
	\begin{equation}\label{2 variable Strauss curve}
		\begin{gathered}
			\Sigma_{crit}=\bigg\{\max\Big(\frac{q_1+2+q_2^{-1}}{q_1q_2-1},\frac{q_2+2+q_1^{-1}}{q_1q_2-1}\Big)=\frac{n-1}{2}\bigg\}
		\end{gathered}
	\end{equation}
	with stability on one side and instability on the other. In fact, there is a second restriction concerning local well-posedness related to the fact that the non-linearity is not smooth (see \cref{regularity vs decay}). The first results concerning this system draw intuition from the elliptic realm, and compare the hyperbolic case to the Lane-Emden equation ($\Box\to-\Delta$ in \eqref{2 variable Strauss}) where the critical curve is
	\begin{equation*}
		\begin{gathered}
			\Sigma_{crit}=\bigg\{\max\Big(\frac{q_1+1}{q_1q_2-1},\frac{q_2+1}{q_1q_2-1}\Big)=\frac{n-2}{2}\bigg\}.
		\end{gathered}
	\end{equation*} 
	Still, knowing this result, there is no obvious way to arrive to the critical curve for the hyperbolic case without going through the weighted Strichartz estimates and bootstrap (or contraction mapping) argument. A natural continuation of this program would be to consider the system
	\begin{equation}\label{n variable Strauss}
		\begin{gathered}
			\Box \phi_i=\sum_{j}a_{ij}\phi_j\abs{\phi_j}^{c_{ij}-1}\qquad i\in\{1,2,...,d\},
		\end{gathered}
	\end{equation}
	where knowing the critical set in advance would potentially help to understand the conditions on $c_{ij}$ for global well posedness. For results concerning this system, see \cref{main results}.
	
	One can alternatively introduce derivative non-linearities (we restrict our-selves to semi-linear problems) and consider the equation
	\begin{equation}\label{Glassey}
		\begin{gathered}
			\Box\phi=\abs{\partial_t\phi}^q.
		\end{gathered}
	\end{equation}
	In \cite{john_blow-up_1981}, the author showed that the vacuum ($\phi=0$) is unstable for the above equation with $q\leq2$. Similar to \eqref{Strauss}, there is a power $q_c=1+\frac{2}{n-1}$ that is conjectured (due to Glassey \cite{glassey_finite-time_1981}) to separate stable and unstable regimes. There have been many works on this equation as well, see details in \cite{hidano_glassey_2012}. In particular, the instability for $q<q_c$ is known in all dimension while stability holds for $n=2,3$ without symmetry and for $n\geq4$ under spherical symmetry.
	
	Combining \eqref{Strauss} and \eqref{Glassey}, one may instead consider
	\begin{equation}\label{Strauss Glassey}
		\begin{gathered}
			\Box\phi=\abs{\partial_t\phi}^{q_1}+\abs{\phi}^{q_2}.
		\end{gathered}
	\end{equation}
	Part of the critical curve (see \cite{hidano_combined_2016} for details) is given by
	\begin{equation*}
		\begin{gathered}
			\Sigma_{crit}=\partial\Big\{q_1>q_{\text{Glassey}},q_2>q_{\text{Strauss}},(q_2-1)((n-1)q_1-2)>4\Big\}.
		\end{gathered}
	\end{equation*}
	The results in this case are restricted to $n=2,3$, without symmetry assumptions.

	Alternatively, one may combine \eqref{Strauss} and \eqref{Glassey} into a system
	\begin{equation}\label{Strauss Glassey system}
		\begin{gathered}
			\Box\phi=\abs{\psi}^{q_1}\\
			\Box\psi=\abs{\partial_t\phi}^{q_2}.
		\end{gathered}
	\end{equation}
	In this case, part of the critical curve (\cite{hidano_life_2016}) is 
	\begin{equation*}
		\begin{gathered}
			\Sigma_{crit}=\bigg\{\Big(\frac{n-1}{2}q_2-1\Big)(q_2q_1-1)=q_2+2\bigg\}.
		\end{gathered}
	\end{equation*}
	
	Similarly, adding a term motivated by the null condition (see below), we may consider the system
	\begin{equation}\label{Strauss null}
		\begin{gathered}
			\Box\phi=\abs{\psi}^{q_1}\\
			\Box\psi=\abs{\partial_t\phi}^{1+2/(n-1)}+\abs{\phi}^{q_2}.
		\end{gathered}
	\end{equation}
	In this case, part of the critical curve in $n=3$ ( \cite{hidano_global_2022}) is 
	\begin{equation*}
		\begin{gathered}
			\Sigma_{crit}=\big\{q_1q_2 = 2q_2 + 3\big\}
		\end{gathered}
	\end{equation*}
	which differs from \eqref{2 variable Strauss}. Importantly, note that the point $q_1=q_2=3$ is on the critical curve, which is also physically relevant see \cite{hidano_global_2022}.  This observation is crucial for the ill-posedness result related to the weak null condition presented in \cref{main results}.
	
	Considering analytic non-linearities, we must mention the celebrated null condition of Klainerman \cite{klainerman_long-time_1982}, which in particular gives stability for \footnote{this nonlinear equation is integrable, but is the simplest to exhibit the condition}
	\begin{equation}\label{integrable null}
		\begin{gathered}
			\Box\phi=\partial\phi\cdot\partial\phi=(\partial_t\phi)^2-\nabla_x\phi\cdot\nabla_x\phi.
		\end{gathered}
	\end{equation}
	Importantly, this results shows that imposing structure on the differentiated part of the non-linearity can separate globally well-posed and ill-posed problems. This was extended to the weak null condition \cite{lindblad_weak_2003} to get global existence for
	\begin{equation}\label{weak null}
		\begin{gathered}
			\Box\phi=(\partial_t\psi)^2\\
			\Box\psi=\partial\phi\cdot\partial\phi.
		\end{gathered}
	\end{equation}
	This null condition was further extended using the space-time resonance method by \cite{Pusateri2013}.
	
	Let's return momentarily to the original goal, to understand wave-Klein Gordon systems. The Klein Gordon field has a frequency dependent speed of propagation.\footnote{This statement can be made quantitative with stationary phase method.} Therefore, as a first proxy, one may study multispeed wave systems to gain some intuition to the problem.
	
	Moving away from constant speed wave equations allows one to consider a much larger family of equation with smooth non-linearity. Here, we do not wish to give as a detailed overview as for the fixed speed case, we just mention some surprising works:
	\begin{itemize}
		\item In \cite{yokoyama_global_2000} it was shown that in $\R^{3+1}$
		\begin{equation*}
			\begin{gathered}
				\Box_1\phi=(\partial_t\psi)^2\\
				\Box_2\psi=(\partial_t\phi)^2
			\end{gathered}
		\end{equation*}
		has global solutions where $\Box_c=-\partial_t^2+c\Delta$. This is in strong contrast with single speed system.
		\item In \cite{ohta_counterexample_2003}, Ohta proved that the vacuum in $\R^{3+1}$ is unstable for the system
		\begin{equation*}
			\begin{gathered}
				\Box_1\phi=\psi\partial_t\psi\\
				\Box_2\psi=(\partial_t\phi)^2.
			\end{gathered}
		\end{equation*}
		\item In \cite{kubo_small_2000}, it was shown that for the 2 component Strauss problem \eqref{2 variable Strauss} with unequal speeds the vacuum is unstable under the same critical curve as for the same speed system.
	\end{itemize}
	 There are many more results concerning such systems, see \cite{katayama_global_2006},  \cite{anderson_global_2021}, \cite{hidano_global_2022} references therein. For physical motivation behind such systems, see \cite{anderson_global_2021}.
	 
	\subsection{The problem of regularity}\label{regularity vs decay}
	In this section, we argue, that the critical exponents discussed above form a separate problem from the issues of regularity of the equation and are worth a \textit{separate} study.
	
	As mentioned earlier, there are restrictions beyond the critical exponents (curves) stated above that influence where stability results are available. These, we claim, are related to regularity and in this paper we wish to separate these from what we will call decay. To be more specific, consider the Glassey conjecture \eqref{Glassey} where the critical power is $q_c(2)=3,q_c(3)=2,q_c(n)=1+2/(n-1)=1+2(1/n+1/n^2)+\mathcal{O}(1/n^3)$, in particular $q_c(n)<2\, \forall n>3$. The scale invariant Sobolev exponent (for it's importance see \cite{tao_nonlinear_2006}) for the problem is $s_c(q)=n/2+\frac{q-2}{q-1}$, so $s_c(q_c)=1/2$. Therefore, one could expect local well-posedness in $H^1$ for $q$ slightly above $q_c$. However, since the important ill-posedness results of Lindblad \cite{lindblad_sharp_1993}, we know that well posedness does not hold down to the critical exponent, indeed one requires $s>\max(s_c(q),\frac{n+5}{4})$. When $q$ is not an integer, the right hand side of the equation has at most $H^{q+1/2}$ regularity, so by standard energy estimate the solution can have at most $H^{q+3/2}$. For $n>6$, we have $q_c+3/2<\frac{n+5}{4}$, so it seems extremely hard task to prove the conjecture in any non-symmetric function space. In particular, it's not even clear if the equation near the critical power is locally well-posed in \textit{any} non-symmetric function space. Indeed, the authors of \cite{hidano_glassey_2012} proved the conjecture for $n\geq4$ under spherical symmetry, which effectively reduced the system to a problem in $\R^{1+1}$, for which much less regularity is sufficient.
	
	For another example, note that the critical curve for \eqref{2 variable Strauss} given in \cite{del_santo_global_1997} is not quite the one stated above. In particular, they show the following theorem
	\begin{theorem*}[\cite{del_santo_global_1997}]
		
		\eqref{2 variable Strauss} has stable vacuum above the \eqref{2 variable Strauss curve} curve if
		\begin{itemize}
			\item $n=2,3$ and $q_1,q_2\in(5-n,6-n]$.
			\item $n\geq4$ and $\frac{\min(q_1,q_2)-1}{q_1q_2-1}>\frac{n-1}{2(n+1)}$
		\end{itemize}
	 	
	 	In \eqref{2 variable Strauss} the vacuum is not stable in the set of $\mathcal{C}^2$ solutions below the curve \eqref{2 variable Strauss curve}.
	\end{theorem*}
	
	Note, that the stability theorems require that $q_i$ are bounded below irrespective of the other. The upper bounds are expected to be of technical nature as the improved result (for $n=3$) in \cite{agemi_critical_2000} show.
	
	As we see that regularity is a somewhat separate issue from the decay of the non-linear terms, one may also test our heuristic using additional $t$ and $u$ weight on the non-linearities. One natural question might be the maximum value of $\alpha,\beta$ such that the equations
	\begin{equation*}
		\begin{gathered}
			\Box\phi=t^\alpha u^\beta \partial\phi\cdot\partial\phi\\
			\Box\phi=t^\alpha u^\beta (\partial_t\phi)^2
		\end{gathered}
	\end{equation*}
	have global solutions provided small enough initial data $\norm{\phi_0,\phi_1}_{H^{N+1}\times H^N}\leq\epsilon$. The second equation was analysed in $\R^{1+1}$ in the recent work \cite{kitamura_semilinear_2022}.
	
	\subsection{Main results}\label{main results}
	The main theorems of the paper concern various nonlinear systems that we selected to best represent the usage of heuristics described in \cref{role of asymptotics} and \cref{finding critical exponents}.
		
	\begin{theorem*}[Stability results]
		The vacuum is stable in the following scenarios:
		\begin{enumerate}
			\item 	For the equation
			\begin{equation}\label{partial v}
				\begin{gathered}
					\Box\phi=\abs{\partial_v\phi}^q\\
					\phi(0)=\phi_0,\partial_t\phi_(0)=\phi_1\\
					\phi:\R^{n+1}\to\R
				\end{gathered}
			\end{equation}
			under spherical symmetry with $n\in\{3\}\cup\{n\geq5\}$ and $\frac{4}{n-2}+1>q>q_{c-v}$, where the critical exponent is the one given by the analysis in \cref{role of asymptotics}, ie. the positive root of $q(q-1)\frac{n+1}{2}=1$. (\cref{partival v main theorem})
			\item  For the system of equations \cref{n variable Strauss}
			\begin{equation*}
				\begin{gathered}
					\Box\phi_i=\sum_{j=1}^na_{ij}\abs{\phi_j}^{c_{ij}}\\
					\phi_i:\R^{3+1}\to\R
				\end{gathered}
			\end{equation*}
			with $a_{ij}\in\R$ and $c_{ij}>2$ the $0$ solution is stable if $c_{ij}$ satisfy \cref{n variable Strauss condition}, that is:
			\begin{equation*}
				s_i=\min_{j|a_{ij}\neq0}(c_{ij}-2+\min(0,s_jc_{ij}-1))
			\end{equation*}
			has a solution for $s_i$. (\cref{n strauss main theorem})
			\item  For the system \eqref{Strauss Glassey system}
			\begin{equation*}
				\begin{gathered}
					\Box\phi=\abs{\psi}^{q_1}\\
					\Box\psi=\abs{\partial_t\phi}^{q_2}.
				\end{gathered}
			\end{equation*}
			under spherical symmetry in $\R^{3+1}$ ($n=3$) with $1<q_2q_1\big((q_2-2)+q_2(q_1-2)\big)$ and $q_1<2$.  (\cref{strauss glassey main theorem})
		\end{enumerate}
	\end{theorem*}

	\begin{theorem*}[Instability results]
		In $\R^{3+1}$ the vacuum is unstable for the following equations
		\begin{itemize}
			\item \eqref{n variable Strauss} if $a_{ij}\geq0$ and the condition \eqref{n variable Strauss condition} does not hold. ( \cref{blow up for n strauss theorem})
			\item  \eqref{Strauss Glassey system} for $q_1<2$ and $q_2q_1((q_2-2)+q_2(q_1-2))<1$. (\cref{blow up for Strauss Glassey system})
			\end{itemize}
	
		The system 
	\begin{equation}\label{counter example to WNC}
		\begin{gathered}
			\Box\phi_1=0\\
			\Box\phi_2=\phi_1^2\\
			\Box\phi_3=\phi_2^2\\
			\Box\phi_4=\phi_3^2\\
			(1\pm\partial_v\phi_4)\Box\eta=(\partial_v\phi_4)^2
		\end{gathered}
	\end{equation}
	satisfies the weak null condition, but there exists arbitrarily small initial data such that \cref{counter example to WNC} admits no smooth global solutions. (\cref{WNC cor})
	\end{theorem*}

	\begin{theorem*}[Instability on the critical curve, \cref{failure of weak null condition}]
		In $\R^{3+1}$ the vacuum is unstable for \cref{critical problem without abs}, that is, for
		\begin{equation*}
			\begin{gathered}
				\Box\phi_1=\phi_3^3\\
				\Box\phi_2=(\partial_t\phi_1)^2\\
				\Box\phi_3=(\partial_t\phi_2)^2+\phi_1^3
			\end{gathered}
		\end{equation*}	
	\end{theorem*}

	A few remarks are in order.
	\begin{remark}
		All instability results are proved in spherical symmetry. We find conditions on smooth spherically symmetric initial data such that there cannot be global $\mathcal{C}^2$ solution. We think, that the instability is not related to spherical symmetry, moreover the generic behaviour (even if not for all data) should be finite time blow-up in all instability cases considered. For further details, see \cref{blow up results}.
	\end{remark}

	\begin{remark}
		The upper bound $1+\frac{4n}{n-2}$ in the first stability result is the same as in \cite{hidano_glassey_2012} and it appears because both they and us work at $H^2$ regularity. This is suboptimal, since the nonlinearity has more than 1 order of differentiability, therefore, we expect that the upper bound in both results can be improved using fractional Sobolev spaces or other low regularity methods.
	\end{remark}

	\begin{remark}(Weak null conjecture)
		The global ill-posedness of \cref{counter example to WNC} disproves Conjecture 1.1 of \cite{deng_global_2018}. However, we find the instability for \cref{critical problem without abs} to be much more striking. Indeed, ignoring cubic terms satisfies every possible notion of weak null condition as it is a system of inhomogeneous wave equations. In that sense, it is a new type of singular behaviour coming from the undifferentiated wave components and shows -despite the common belief- that cubic terms can be important for stability question in $\R^{3+1}$. Note also, that removing any one of the nonlinear terms gives global solutions. Without $\psi_3^3$ or $(\partial_t\psi_1)^2$, the system decouples to linear waves. Removing $(\partial_t\phi_2)^2$ it decouples to a linear equation and a system of type \cref{strauss system}. Removing $\phi_1^3$ gives a system similar to \cref{strauss glassey system}, for which the exponents are away from criticality, thus we expect global solutions.
	\end{remark}
	
	\begin{remark}[Failure of heuristics]
	To the author's knowledge, there is no known result in the literature that falsifies the heuristics given in \cref{role of asymptotics}. The only similar statement is the one from the work \cite{keel_small_1999} studying the critical exponent $p=1+2/n$ for nonlinear KG equation. The authors say \enquote{it is likely that one has blow-up for $q=1+2/n+\epsilon$ for sufficiently high $n$}. Note, that by blow-up, one probably means that there is a function space $X$ where the equation is well-posed, but there is no $\delta>0$ such that $\norm{\phi_0}_X<\delta\implies \text{global solution}$. Therefore, in light of \cite{lindblad_sharp_1993}, a suitable setting to test the stability for this problem is under spherical symmetry (as in \cite{hidano_glassey_2012}).
	\end{remark}

	\begin{remark}[Quasi linear equations]
		The advantage of using the weak null condition is that it is also applicable to quasi linear equations. In particular, one may use it to understand stability of $\Box\phi=\phi\Delta\phi$ from modification to characteristics of the flow and instability of $\Box \phi=\partial_t\Delta\phi$ resulting from a Burger's type behaviour. In both cases, the non-linearities change the characteristics of the flow, and create non-perturbative effects in the sense that solutions will not scatter linearly. Such effects cannot be captured by the analysis presented in this paper, however, note that even in the significant work \cite{keir_weak_2018}, Keir separates the problem of quasi-linear behaviour from the semilinear part. 
	\end{remark}

	\begin{remark}
		Throughout this paper, we work with compactly supported initial data, but the heuristics from \cref{role of asymptotics} and \cref{finding critical exponents} suggest optimal conditions on the fall-off for the initial data. For more details see \cref{initial data critical exponent}.
	\end{remark}
	\subsubsection*{Notation and coordinates}\label{notation}
	Throughout the paper we will use $A\lesssim_{a,b,...,c} B$ to denote the existence of a constant $C$ depending on $a,b,...,c$ such that $A\leq CB$. Furthermore, the dependence on parameters that are fixed in the statement to be proven, such as dimension or coefficients in an equation, will be treated implicitly. We will also use $A\sim B$ notation to mean $A\lesssim B$ and $B\lesssim A$.
	
	We will work exclusively on Minkowski spacetime $\R^{n+1}$ with coordinates $t,x_1,...x_n$. We introduce the (usual) radial, advanced/retarded time functions 
	\begin{equation*}
		\begin{gathered}
			r^2=x_1^2+...+x_n^2,\,v=\frac{t+r}{2},\, u=\frac{t-r}{2}.
		\end{gathered}
	\end{equation*}
	\subsubsection*{Outlook}
	The rest of the paper contains 3 main parts. In section \cref{role of asymptotics} we are going to describe the asymptotics of linear wave in Minkowski spacetime, understand how \textit{tails} are generated and give a more precise definition for outgoing and incoming radiation. Afterwards, in \cref{finding critical exponents} we will show how one can recover all the critical exponents discussed in \cref{zoology} using the idea of tail generation. Finally, in the following chapters, we are going to develop optimal decay techniques and prove the theorems in \cref{main results}.
	
	\paragraph{\textbf{Acknowledgement:}} The author would like to thank Claude Warnick, Leonhard Kehrberger and Jason Joykutty for  many helpful discussions on both the technical and conceptual parts of the work. This project was founded by EPSRC.

	\section{The role of asymptotics}\label{role of asymptotics}
	The presentation in this section draws heavily from the linear wave part of \cite{hintz_stability_2020}, presenting the material in a much simplified context.

	In this section, we will give an overview of linear wave propagation in settings of increasing difficulty, with a particular focus on the leading order asymptotic behaviour of the solution. Afterward we present some of the most used estimates in the literature and discuss their optimality with respect to the expected behaviour. We are mainly interested in the case of generic decay of the inhomogeneity, but we make extra remarks how the situation can differ in certain special case, eg. generate logarithmic difference. We wish to stress the following points:
	\begin{itemize}
		\item the role and relationship between \textit{incoming} and \textit{outgoing} radiation, see \cref{intuition from 1+1}. This is a related observation to the way the Newman-Penrose constant is responsible for the creation of tails, see \cite{angelopoulos_late-time_2018}, \cite{gajic_relation_2022}.
		\item the generation of above mentioned radiation by inhomogeneity (=non-linearity). This effect is the one described in \cite{luk_tale_2021}.
		\item the similarity across \textit{all} dimensions for global existence problems. In particular, the tail creation from resonances (\cref{even odd dimensions}) and exceptional cancellations (\cref{exceptional cancellation}) that are relevant in the work of \cite{luk_tale_2021} do not show up in most problems considered here. However, note that in one of the equations studied in this paper, we do observe the importance of this effect \cref{importance of dimension in even case}. To the author's knowledge, this is the first equation where this phenomena is relevant for global existence.
	\end{itemize}
	
	\subsection{Intuition from $\R^{1+1}$}\label{intuition from 1+1}

	Consider the wave equation in the simplest possible setting \footnote{We impose "radial" symmetry to have only 1 asymptotic end, as in higher dimensions.}
	\begin{equation}\label{1+1 linear wave}
		\begin{gathered}
			\Box\phi=(\partial_t^2-\partial_r^2)=F\\
			\phi(0)=\phi_0,\partial_t\phi(0)=\phi_1:\R\to\R.\\
			\phi(t,r)=\phi(t,-r),\, F(t,r)=F(t,-r.)
		\end{gathered}
	\end{equation}
	
	People often represent the above spacetime on Penrose diagram as on \cref{fig: relativist}, for details see eg. \cite{Dafermos2008}, \cite{wald_general_1984}. This corresponds to using coordinate $2T=\tan^{-1}(t+r)+\tan^{-1}(t-r)\,,$ $2R=\tan^{-1}(t+r)-\tan^{-1}(t-r)$ and is a conformal compactification of Minkowski spacetime with conformal factor $\frac{1}{\cos^2(T+R)\cos^2(T-R)}$. The main advantage of the conformal point of view is that it preserves the angle of null geodesics, thus clearly showing the causal relations. This in turn helps understand the propagation of singularities \cite{baskin_asymptotics_2018}, obstruction to decay via unstable \cite{Dafermos2008} or stable trapping \cite{benomio_stable_2021} and many more phenomena related to the high frequency behaviour of waves.  It also highlights the asymptotic region of waves propagating in Minkowski space (see below).
	
	An alternative compact visualisation of Minkowski space is via the approach of \cite{baskin_asymptotics_2018} shown on \cref{fig: analyst}. In this setting, one first radially compactifies Minkowski (attaching a ball at $\abs{t}+\abs{x}=\infty$) and than blows up (zooms in) on the geometrically important part, which is the light cone in case of Minkowski. 
	
	This alternative view highlights all important parts of \textit{infinity}, in the sense, that solutions to the equation we are interested in will be  \textit{smooth} with respect to a geometric choice of vector fields, in particular their behaviour close to the boundary (late time in case of future boundary) is given by an asymptotic expansion. This notion of smoothness is called polyhomogeneity. We give two results from the literature to highlight the generality of this approach. Let's use  coordinates $\rho=x/t$ on $I^+$ (as shown on \cref{fig: analyst}) and $t$ we work in the region $\abs{x/t}\in[1/4,3/4],t\to\infty$. 
	\begin{itemize}
		\item The spherically symmetric linear wave equation on a $3+1$ dimensional black hole backgrounds have asymptotics 
		\begin{equation*}
			\begin{gathered}
				\bigg(1+\sum_{j=1}^k\big(\frac{1-\rho}{1+\rho}\big)^j\bigg)\frac{1}{(1-\rho)^{k+1}(1+\rho)t^{k+1}}
			\end{gathered}
		\end{equation*}
		where $k$ depends on the number of vanishing Newman-Penrose constants. Furthermore, derivatives normal to $I^+$ yield additional decay. Similar results hold for higher $l$ modes. (\cite{angelopoulos_late-time_2018}).
		\item Solutions to the wave equation on spacetimes that are asymptotically Minkowski have conormal/polyhomogeneous solutions if the initial data is of respective type (\cite{baskin_asymptotics_2018}). For definitions of conormal and polyhomogeneous, see \cite{baskin_asymptotics_2018}. The first refers to functions that decay faster when differentiated towards the conformal boundary, while the latter says that it has a certain expansion near the boundary. 
	\end{itemize}
	
	Finally, let us mention that the Klein Gordon equation also has a \textit{nice} expansion, but this is only apparent up to an oscillatory phase, ie. it is an oscillating factor times something polyhomogeneous. In particular, the homogeneous linear Klein-Gordon equation has leading order asymptotics 
	\begin{equation*}
		\begin{gathered}
			e^{it\sqrt{1-\rho^2}}t^{-n/2}(1-\rho^2)^{-(n+2)/2}\hat{\phi}(-\frac{\rho}{\sqrt{1-\rho^2}}),
		\end{gathered}
	\end{equation*}
	where $\hat{\phi}$ is related to the Fourier transform of the initial data, see \cite{hormander_lectures_1997}). Importantly, $\partial_t\phi$ will have the same rate in $\{r\leq 1\}$ of as $\phi$, thus we expect no difference between the critical exponent for $(\Box-1)\phi=\abs{\phi}^q$ and $(\Box-1)\phi=\abs{\partial_t\phi}^q$.
	
	\begin{figure}[h]
		\centering
		\subfloat[\centering Relativist's Minkowski diagram\label{fig: relativist}]{\resizebox{.3\textwidth}{!}{
\begingroup%
  \makeatletter%
  \providecommand\color[2][]{%
    \errmessage{(Inkscape) Color is used for the text in Inkscape, but the package 'color.sty' is not loaded}%
    \renewcommand\color[2][]{}%
  }%
  \providecommand\transparent[1]{%
    \errmessage{(Inkscape) Transparency is used (non-zero) for the text in Inkscape, but the package 'transparent.sty' is not loaded}%
    \renewcommand\transparent[1]{}%
  }%
  \providecommand\rotatebox[2]{#2}%
  \newcommand*\fsize{\dimexpr\f@size pt\relax}%
  \newcommand*\lineheight[1]{\fontsize{\fsize}{#1\fsize}\selectfont}%
  \ifx\svgwidth\undefined%
    \setlength{\unitlength}{226.77165354bp}%
    \ifx\svgscale\undefined%
      \relax%
    \else%
      \setlength{\unitlength}{\unitlength * \real{\svgscale}}%
    \fi%
  \else%
    \setlength{\unitlength}{\svgwidth}%
  \fi%
  \global\let\svgwidth\undefined%
  \global\let\svgscale\undefined%
  \makeatother%
  \begin{picture}(1,1)%
    \lineheight{1}%
    \setlength\tabcolsep{0pt}%
    \put(0,0){\includegraphics[width=\unitlength,page=1]{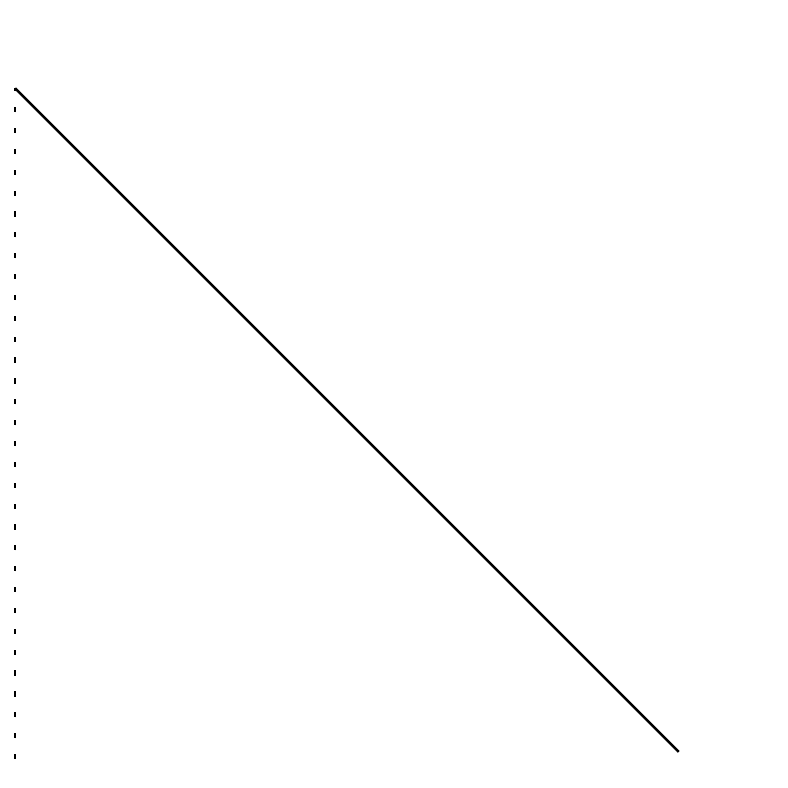}}%
    \put(0,0.92850527){\color[rgb]{0,0,0}\makebox(0,0)[lt]{\lineheight{1.25}\smash{\begin{tabular}[t]{l}$i^+$\end{tabular}}}}%
    \put(0.3829469,0.58201494){\color[rgb]{0,0,0}\makebox(0,0)[lt]{\lineheight{1.25}\smash{\begin{tabular}[t]{l}$\scri$\end{tabular}}}}%
    \put(0.89264297,0.05838833){\color[rgb]{0,0,0}\makebox(0,0)[lt]{\lineheight{1.25}\smash{\begin{tabular}[t]{l}$i^0$\end{tabular}}}}%
  \end{picture}%
\endgroup%
}}
		\qquad
		\subfloat[\centering Analyst's Minkowski diagram\label{fig: analyst}]{\resizebox{.3\textwidth}{!}{
\begingroup%
  \makeatletter%
  \providecommand\color[2][]{%
    \errmessage{(Inkscape) Color is used for the text in Inkscape, but the package 'color.sty' is not loaded}%
    \renewcommand\color[2][]{}%
  }%
  \providecommand\transparent[1]{%
    \errmessage{(Inkscape) Transparency is used (non-zero) for the text in Inkscape, but the package 'transparent.sty' is not loaded}%
    \renewcommand\transparent[1]{}%
  }%
  \providecommand\rotatebox[2]{#2}%
  \newcommand*\fsize{\dimexpr\f@size pt\relax}%
  \newcommand*\lineheight[1]{\fontsize{\fsize}{#1\fsize}\selectfont}%
  \ifx\svgwidth\undefined%
    \setlength{\unitlength}{226.77165354bp}%
    \ifx\svgscale\undefined%
      \relax%
    \else%
      \setlength{\unitlength}{\unitlength * \real{\svgscale}}%
    \fi%
  \else%
    \setlength{\unitlength}{\svgwidth}%
  \fi%
  \global\let\svgwidth\undefined%
  \global\let\svgscale\undefined%
  \makeatother%
  \begin{picture}(1,1)%
    \lineheight{1}%
    \setlength\tabcolsep{0pt}%
    \put(0,0){\includegraphics[width=\unitlength,page=1]{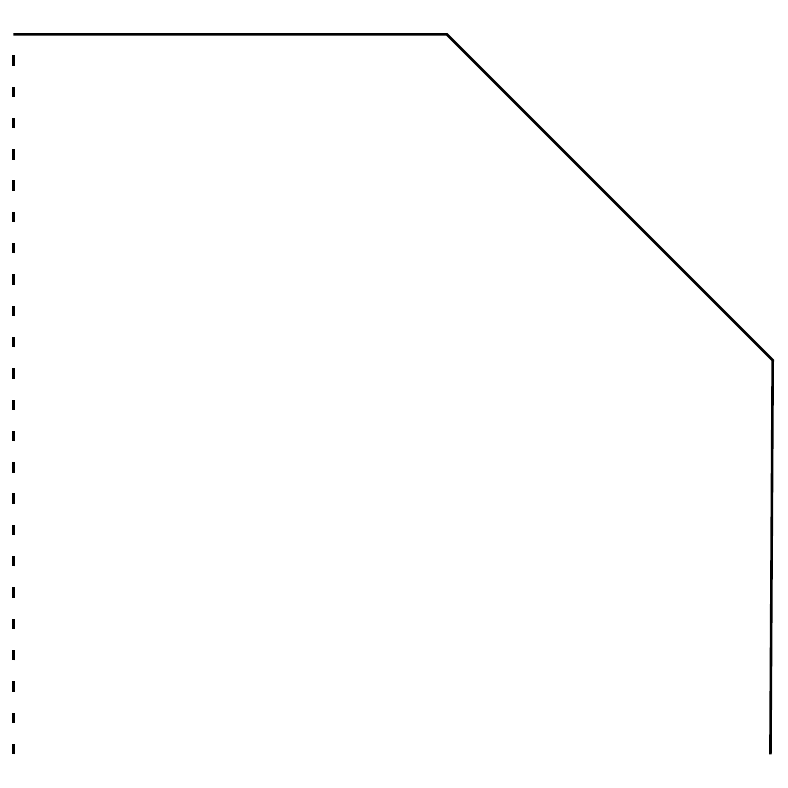}}%
    \put(0.1332356,0.90019854){\color[rgb]{0,0,0}\makebox(0,0)[lt]{\lineheight{1.25}\smash{\begin{tabular}[t]{l}$I^+$\end{tabular}}}}%
    \put(0.73808985,0.78582783){\color[rgb]{0,0,0}\makebox(0,0)[lt]{\lineheight{1.25}\smash{\begin{tabular}[t]{l}$\scri$\end{tabular}}}}%
    \put(0.86392202,0.28139229){\color[rgb]{0,0,0}\makebox(0,0)[lt]{\lineheight{1.25}\smash{\begin{tabular}[t]{l}$I^0$\end{tabular}}}}%
  \end{picture}%
\endgroup%
}}
		\label{fig:diagram comparison}
	\end{figure}
	
	An important part of the wave solution is the radiation field $\psi(u)=\lim_{t\to\infty}\phi|_{t-r=u}$ (for higher dimensions we have to rescale by $t^{\frac{n-1}{2}}$) which has non-trivial limit for non-zero solution. Due to the strong Huygens' principle, \eqref{linear equation} with $F=0,\,\supp\{\phi_0,\phi_1\}\subset \{r<1\}$ has solution with $\supp\phi\subset \{-1\leq t-r\leq 1\}$, in particular $\supp\psi\subset(-1,1)$. Similarly, using the exact solution ($\phi=f(t-r)+g(t+r)$ for some $f,g$) we have $\supp\partial_v\phi|_{t=0}\subset\{r<1\}\implies \psi|_{u>2}=0$. Indeed, one can interpret $\partial_u\phi$ as \textit{outgoing radiation}, while $\partial_v\phi$ \textit{incoming} because the wave operator is simply a transport equation for these two quantities along null geodesics (see \cref{fig: radiation compact}).\footnote{The result of Yang (\cite{yang_global_2013}) can thus be interpreted as large data solution with largeness coming from outgoing radiation} Due to radial symmetry, incoming waves turn to outgoing waves as they pass through the origin (same happens in higher dimensions without radial symmetry, as there's only one asymptotic end). Therefore, if $\partial_v\phi$ doesn't vanish in a neighbourhood of $\scri$, it will generate outgoing radiation: $\partial_v\phi|_{u\in[u_0,u_1]}\sim v^{-q}\implies\psi\sim u^{1-q}$ (see \cref{fig: radiation tail}). The latter is frequently called a \textit{tail} (\cite{angelopoulos_asymptotics_2018}).

	\begin{figure}[h]
		\centering{
		\subfloat[\centering Support of the homogeneous solution from initial data supported on $\Sigma$. a) and c) represent outgoing part ($\partial_u\phi=0$), b) ingoing ($\partial_v\phi=0$) part of the solution\label{fig: radiation compact}]{\resizebox{.4\textwidth}{!}{
\begingroup%
  \makeatletter%
  \providecommand\color[2][]{%
    \errmessage{(Inkscape) Color is used for the text in Inkscape, but the package 'color.sty' is not loaded}%
    \renewcommand\color[2][]{}%
  }%
  \providecommand\transparent[1]{%
    \errmessage{(Inkscape) Transparency is used (non-zero) for the text in Inkscape, but the package 'transparent.sty' is not loaded}%
    \renewcommand\transparent[1]{}%
  }%
  \providecommand\rotatebox[2]{#2}%
  \newcommand*\fsize{\dimexpr\f@size pt\relax}%
  \newcommand*\lineheight[1]{\fontsize{\fsize}{#1\fsize}\selectfont}%
  \ifx\svgwidth\undefined%
    \setlength{\unitlength}{226.77165354bp}%
    \ifx\svgscale\undefined%
      \relax%
    \else%
      \setlength{\unitlength}{\unitlength * \real{\svgscale}}%
    \fi%
  \else%
    \setlength{\unitlength}{\svgwidth}%
  \fi%
  \global\let\svgwidth\undefined%
  \global\let\svgscale\undefined%
  \makeatother%
  \begin{picture}(1,1)%
    \lineheight{1}%
    \setlength\tabcolsep{0pt}%
    \put(0,0){\includegraphics[width=\unitlength,page=1]{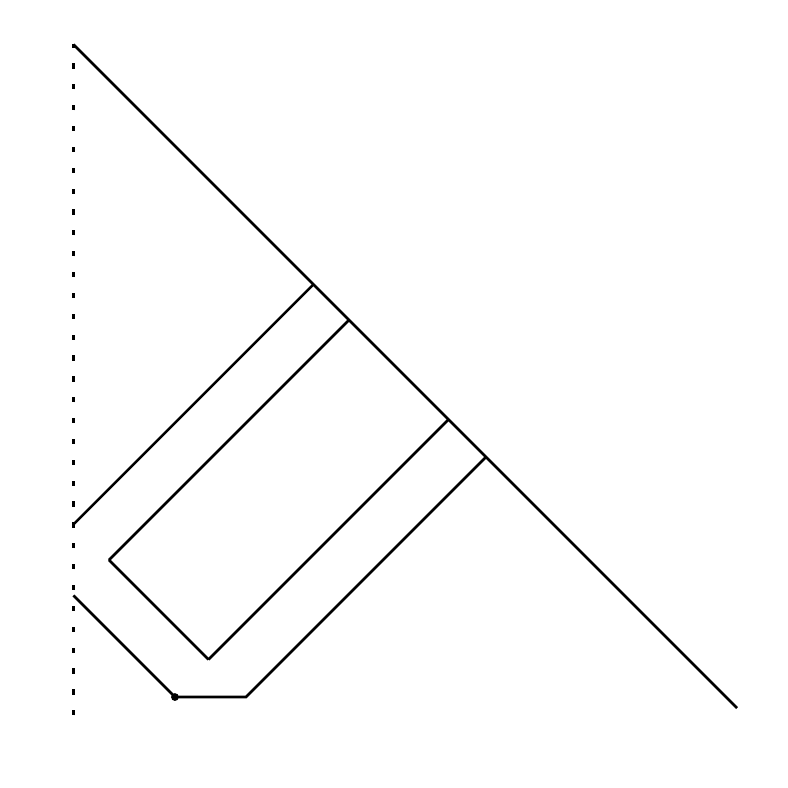}}%
    \put(0.35522587,0.20761113){\color[rgb]{0,0,0}\makebox(0,0)[lt]{\lineheight{1.25}\smash{\begin{tabular}[t]{l}a)\end{tabular}}}}%
    \put(0.15556669,0.22005229){\color[rgb]{0,0,0}\makebox(0,0)[lt]{\lineheight{1.25}\smash{\begin{tabular}[t]{l}b)\end{tabular}}}}%
    \put(0.15751486,0.37933373){\color[rgb]{0,0,0}\makebox(0,0)[lt]{\lineheight{1.25}\smash{\begin{tabular}[t]{l}c)\end{tabular}}}}%
    \put(0.24564855,0.07282658){\color[rgb]{0,0,0}\makebox(0,0)[lt]{\lineheight{1.25}\smash{\begin{tabular}[t]{l}$\Sigma$\end{tabular}}}}%
  \end{picture}%
\endgroup%
}}
		\qquad
		\subfloat[\centering Initial data supported on $\Sigma$ with fall-off $v^{-q}$ will have a tail $u^{-q}$ on $\scri$ \label{fig: radiation tail}]{\resizebox{.4\textwidth}{!}{
\begingroup%
  \makeatletter%
  \providecommand\color[2][]{%
    \errmessage{(Inkscape) Color is used for the text in Inkscape, but the package 'color.sty' is not loaded}%
    \renewcommand\color[2][]{}%
  }%
  \providecommand\transparent[1]{%
    \errmessage{(Inkscape) Transparency is used (non-zero) for the text in Inkscape, but the package 'transparent.sty' is not loaded}%
    \renewcommand\transparent[1]{}%
  }%
  \providecommand\rotatebox[2]{#2}%
  \newcommand*\fsize{\dimexpr\f@size pt\relax}%
  \newcommand*\lineheight[1]{\fontsize{\fsize}{#1\fsize}\selectfont}%
  \ifx\svgwidth\undefined%
    \setlength{\unitlength}{226.77165354bp}%
    \ifx\svgscale\undefined%
      \relax%
    \else%
      \setlength{\unitlength}{\unitlength * \real{\svgscale}}%
    \fi%
  \else%
    \setlength{\unitlength}{\svgwidth}%
  \fi%
  \global\let\svgwidth\undefined%
  \global\let\svgscale\undefined%
  \makeatother%
  \begin{picture}(1,1)%
    \lineheight{1}%
    \setlength\tabcolsep{0pt}%
    \put(0,0){\includegraphics[width=\unitlength,page=1]{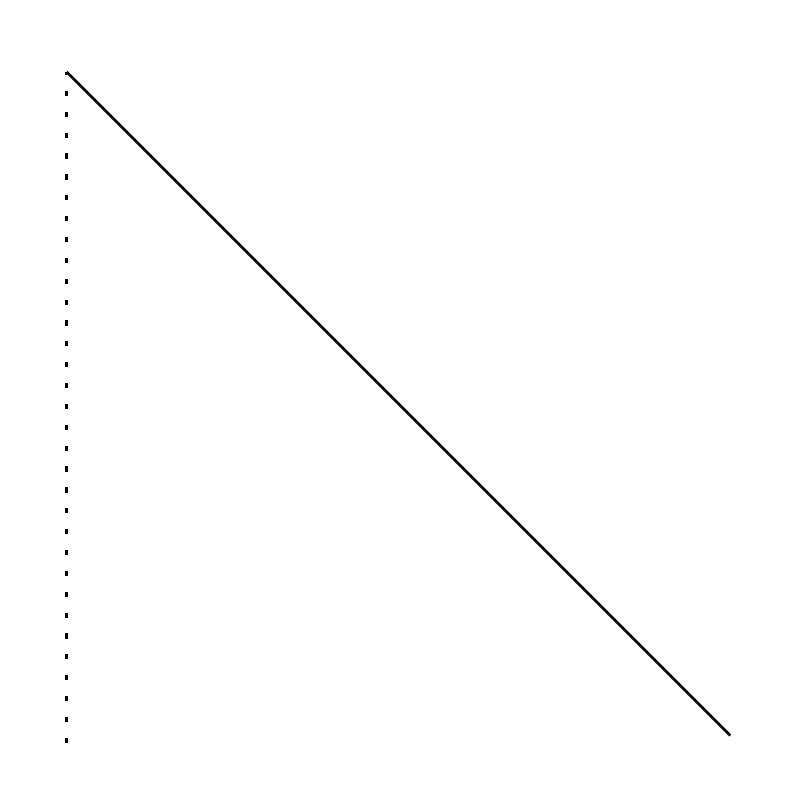}}%
    \put(0.19233691,0.11744938){\color[rgb]{0,0,0}\makebox(0,0)[lt]{\lineheight{1.25}\smash{\begin{tabular}[t]{l}$\Sigma$\end{tabular}}}}%
    \put(0,0){\includegraphics[width=\unitlength,page=2]{in-outgoing_tail.pdf}}%
    \put(0.25970552,0.22582192){\color[rgb]{0,0,0}\makebox(0,0)[lt]{\lineheight{1.25}\smash{\begin{tabular}[t]{l}$v=v_0$\end{tabular}}}}%
    \put(0.4559464,0.30293832){\color[rgb]{0,0,0}\makebox(0,0)[lt]{\lineheight{1.25}\smash{\begin{tabular}[t]{l}$v=2v_0$\end{tabular}}}}%
    \put(0.02939896,0.75448155){\color[rgb]{0,0,0}\makebox(0,0)[lt]{\lineheight{1.25}\smash{\begin{tabular}[t]{l}$u=2v_0$\end{tabular}}}}%
    \put(0.14175582,0.64656303){\color[rgb]{0,0,0}\makebox(0,0)[lt]{\lineheight{1.25}\smash{\begin{tabular}[t]{l}$u=v_0$\end{tabular}}}}%
  \end{picture}%
\endgroup%
}}

		\label{fig:incoming and outgoing waves}
		}
	\end{figure}

	In non-linear applications, it's more important to study the formation of tails from inhomogeneities ($F\neq0$). We study these under the assumption that the support of $F$ is in one of the four regions shown in \cref{fig:formation of tails}. 
	
	Using Huygens' principle, we have $\supp F \subset(\mathcal{N}_{\text{comp}})\implies\supp\psi$ is compact. For $F$ supported on a \textit{compact} neighbourhood of  $\scri$ ($\supp F\subset \mathcal{N}_\scri$), say $F|_{\mathcal{N}_\scri}\sim v^{-q}$, $q>1$ we have an indirect creation of tails. First $F$ generates incoming radiation because 
	\begin{equation*}
		\begin{gathered}
			\partial_u\partial_v\phi=F
			\implies \lim_{v\to\infty}r^q\partial_v\phi(v,u_2)=\lim_{v\to\infty}\Big(r^q\partial_v\phi(v,u_1)+\int_{u_1}^{u_2}\d u r^qF(v,u)\Big)
		\end{gathered}
	\end{equation*}
	which implies that for large retarded times $\partial_v\phi\sim v^{-q}\int\d u (r^qF)|_{\scri}$. Note, that the vanishing of $\int\d u (r^qF)$ gives a cancellation, but this is a codimension 1 requirement, ie. non-generic. This in turn \textit{reflects} off the origin to outgoing $\psi\sim u^{1-q}\int\d u (r^qF)|_{\scri}$. 
	\begin{remark}[Creation of logarithm near $\scri$.]\label{log from scri}
		The same analysis works for $q<1$, but in this case, the correct quantity to look at near $\scri$ is $\lim_{v\to\infty} r^{1-q}\phi(v,u)$, because $\phi|_{\mathcal{N}_\scri}\sim r^{1-q}$ grows toward $\scri$. This new rescaled radiation will be some function on $\scri$, which in general will tend to a nonzero limit as $u\to\infty$ and gives the leading order behaviour at timelike infinity $\phi|_{\mathcal{N}_+}\sim t^q$. The case $q=1$ will be different from the rest, because the leading behaviour is $\phi|_{\mathcal{N}_\scri}\sim \log(r)$, however there is no logarithm near $I^+$ and we get  $\phi|_{\mathcal{N}_+}\sim \log(r/t)$.
	\end{remark}
	
	For $F$ supported near timelike infinity, $F$ will directly generate the outgoing radiation $F_{\mathcal{N}_+}\sim t^{-q}\implies \psi\sim u^{2-q}$. This follows from the explicit solution, but one can understand it as in \cite{baskin_asymptotics_2018}: $\Box$ restricted to $I^+$ is the Laplace operator on 1 dimensional hyperbolic space times $1/t^2$ plus faster decaying terms. Inverting the leading term gives the asymptotic behaviour of $\phi$. 
	
	\begin{remark}[Creation of logarithm near $\scri\cap I^+$]\label{log from corner}
	There is a special behaviour of $F$ that leads to additional logarithms in the asymptotic of $\phi$ at $I^+$. Fix $F$ supported in a compact region of $\scri\cap I^+$ with leading order behaviour $F\sim v^{-q}u^{-1},\,q>0$. The exact solution yield $\phi|_{\mathcal{N}_+}\sim\log(t)t^{1-q}$.
	\end{remark}
	
	We won't detail the effect of $F$ supported near $I^0$, as this will not play a role in the present section due to finite speed of propagation and compact data\footnote{In systems with multiple speeds, this effect will be important for compactly supported data too.}. For details see \cite{hintz_stability_2020}. 
	
	If $F$ is conormal or it has a polyhomogeneous expansion (\cite{baskin_asymptotics_2018} for definitions), than it's immediate from the exact solution, that $\phi$ will be so, and each term can be found as described in the previous paragraphs. However, we are only interested in leading order decay not the full expansion. The former is relatively stable under perturbations of the background \footnote{which explains why the Strauss exponent \cref{Strauss} is insensitive to changes of the background} while latter will depend on the fine structure of spacetime, not only the leading order \textit{b-structure}.

	\begin{remark}[Exceptional cancellation $\R^{1+1}$]
		If $\phi$ has initial data on a cone tending towards $\scri$ (\cref{fig: radiation tail}) with fall off $r^{-p},\,p=0$, than $r\partial_v\phi$ is lower order in decay, ie. the incoming radiation is weaker than for general $p$. Indeed, solving the problem with no lower order decaying components, one finds instead of $t^{-0}$ decay for $\phi$, it has compact in $u$ support. As we've seen, an initial data fall of is generated by one faster decaying forcing. However, for $F\sim r^{-1}$, we don't get vanishing tail, instead $\phi|_{\mathcal{N}_+}\sim \log\frac{r}{t}$.
	\end{remark}
	
	\begin{figure}[h]
		\centering{
			\resizebox{.4\textwidth}{!}{
\begingroup%
  \makeatletter%
  \providecommand\color[2][]{%
    \errmessage{(Inkscape) Color is used for the text in Inkscape, but the package 'color.sty' is not loaded}%
    \renewcommand\color[2][]{}%
  }%
  \providecommand\transparent[1]{%
    \errmessage{(Inkscape) Transparency is used (non-zero) for the text in Inkscape, but the package 'transparent.sty' is not loaded}%
    \renewcommand\transparent[1]{}%
  }%
  \providecommand\rotatebox[2]{#2}%
  \newcommand*\fsize{\dimexpr\f@size pt\relax}%
  \newcommand*\lineheight[1]{\fontsize{\fsize}{#1\fsize}\selectfont}%
  \ifx\svgwidth\undefined%
    \setlength{\unitlength}{226.77165354bp}%
    \ifx\svgscale\undefined%
      \relax%
    \else%
      \setlength{\unitlength}{\unitlength * \real{\svgscale}}%
    \fi%
  \else%
    \setlength{\unitlength}{\svgwidth}%
  \fi%
  \global\let\svgwidth\undefined%
  \global\let\svgscale\undefined%
  \makeatother%
  \begin{picture}(1,1)%
    \lineheight{1}%
    \setlength\tabcolsep{0pt}%
    \put(0,0){\includegraphics[width=\unitlength,page=1]{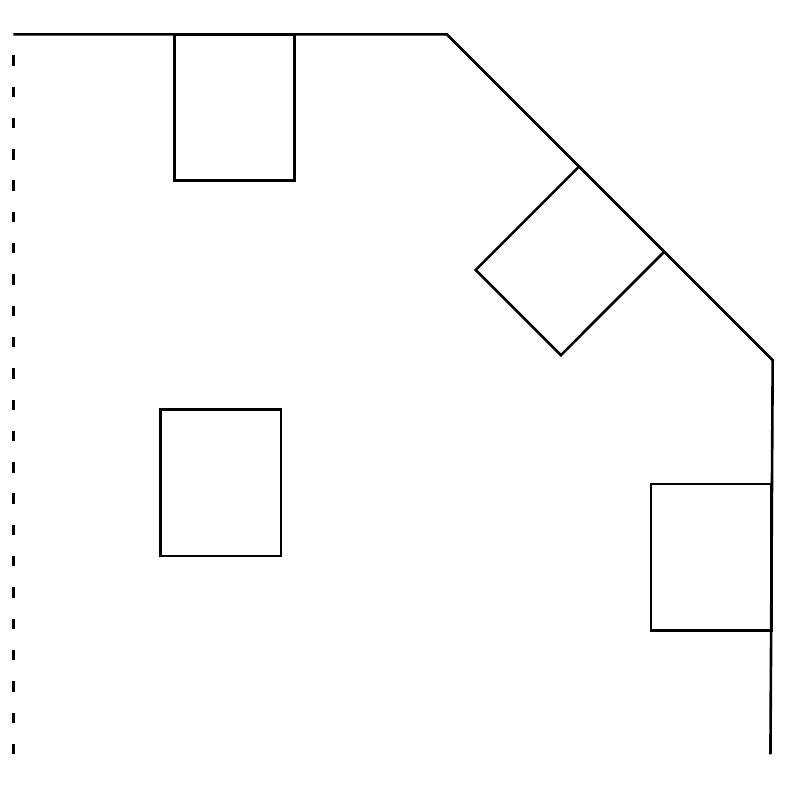}}%
    \put(0.23581321,0.37462602){\color[rgb]{0,0,0}\makebox(0,0)[lt]{\lineheight{1.25}\smash{\begin{tabular}[t]{l}$\mathcal{N}_{\text{comp}}$\end{tabular}}}}%
    \put(0.872264,0.28016562){\color[rgb]{0,0,0}\makebox(0,0)[lt]{\lineheight{1.25}\smash{\begin{tabular}[t]{l}$\mathcal{N}_0$\end{tabular}}}}%
    \put(0.68694949,0.68364024){\color[rgb]{0,0,0}\makebox(0,0)[lt]{\lineheight{1.25}\smash{\begin{tabular}[t]{l}$\mathcal{N}_\scri$\end{tabular}}}}%
    \put(0.26452713,0.87737195){\color[rgb]{0,0,0}\makebox(0,0)[lt]{\lineheight{1.25}\smash{\begin{tabular}[t]{l}$\mathcal{N}_+$\end{tabular}}}}%
  \end{picture}%
\endgroup%
}
			\caption{Asymptotic regions in Minkowski,  $\mathcal{N}_0=\{t/r\in[1/4,3/4],r>1\},\mathcal{N}_\scri=\{t-r\in[1,2],t>1\},\mathcal{N}_+=\{r/t\in[1/4,3/4],t>1\},IV=\{t,r\in[1,2]\}$}
			\label{fig:formation of tails}
		}
	\end{figure}

	\subsection{Higher dimensions}
	The conclusions of the $\R^{1+1}$ case more or less generalise to higher dimensions, but the techniques are very different, see \cite{baskin_asymptotics_2018}. 	If $F$ is supported in region $\mathcal{N}_\scri$ creation of incoming radiation is still the same.  The field $\psi=t^{(n-1)/2}\phi$ at leading order (in decay) satisfies an equation $\partial_u\partial_v\psi=t^{(n-1)/2}F+\frac{1}{t^2}\slashed{\Delta}\psi$ near $\scri$. In particular, the $\slashed{\Delta}$ part is simply a perturbation in this region, thus $\partial_v\psi$ will behave similarly as in the $\R^{1+1}$ case. The \textit{reflection} from the origin however changes significantly. In particular, even for compactly supported initial data (or inhomogeneity in $\mathcal{N}_{\text{comp}}$) there may be non-trivial tails present in $\psi$. In $\R^{n+1}$ with $n$ even, one can read this off from the Green's function, but for more general asymptotically flat metrics, the  resonances of the Laplacian associated to the rescaled metric restricted to timelike infinity (called normal operator) give non-trivial tails.
	
	\begin{remark}[Even and odd dimensions]\label{even odd dimensions}
		As shown in \cite{baskin_asymptotics_2018}, a way to understand the failure of strong Huygens' principle in even dimensions is the existence of resonances (of the above mentioned normal operator) in even dimensional hyperbolic space, while in odd dimensions no such resonances exist. Indeed, for generic compactly supported data, we have $\psi\sim u^{-(n-1)/2}$ in $\R^{n+1}$ $n$ even. This also follows from the Green's function, see eg. \cite{evans_partial_2010} Section 2.4. In most cases, this does not play a role (indeed we will ignore this contribution) as this provides a rather fast decay, but note \cref{importance of dimension in even case}. Note, that this is consistent with the fact that higher even dimensions can be reduced to the $n=2$ high angular mode case.
	\end{remark}
	
	Under radial symmetry in exact odd dimensional Minkowski space, one can use exact representation of 1D solution (see \cite{evans_partial_2010}) to get the leading order decays for inhomogeneous problems from:
	\begin{equation}\label{exact formula in n dim}
		\begin{gathered}
			(\partial_t^2-\partial_r^2-\frac{n-1}{r}\partial_r)\phi=F\implies\\ (\partial_t^2-\partial_r^2)\bigg(\big(\frac{1}{r}\partial_r\big)^{\frac{n-3}{2}}(r^{2k-1}\phi)\bigg)=\big(\frac{1}{r}\partial_r\big)^{(n-3)/2}(r^{n-2}F).
		\end{gathered}
	\end{equation}
	Similar result holds for even dimensions by reducing to the 2D case. The conclusion is that the results from the previous section extend. All this is summarised in \cref{table:decays}.

	\begin{table}
		\begin{center}
			\begin{tabular}{|c|c|c|c|}\hline
				$\alpha$ & $\psi|_{\mathcal{N}_0}$ & $\psi|_{\mathcal{N}_\scri}$ & $\psi|_{\mathcal{N}_+}$  \\ \hline
				$F|_{\mathcal{N}_0}\sim v^{-q}$ & $q-\frac{n-1}{2}-2$  & $\min(0,q-\frac{n-1}{2}-2)$& $q-\frac{n-1}{2}-2$ \\ \hline
				$F|_{\mathcal{N}_\scri}\sim v^{-q}$ &  no support & $\min(0,q-\frac{n-1}{2}-1)$ & $q-\frac{n-1}{2}-1$ \\ \hline
				$F|_{\mathcal{N}_+}\sim v^{-q}$ & no support & $0$ & $q-\frac{n-1}{2}-2$ \\ \hline
				$\supp(F)\subset IV$ & no support & $0$ & $\infty$ \\ \hline
			\end{tabular}
		\end{center}
		\caption{Leading order decay of $\psi\sim t^{-\alpha}$ at different infinite regions (see \cref{fig:formation of tails}) generated by decay of $F$. }
		\label{table:decays}
	\end{table}

	\begin{remark}[Higher $l$ modes and decay]
		As Minkowski space is spherically symmetric, one may restrict to fixed $l$ mode solutions ($\phi^{(l)}$). For such a solution, it is known, that in $\R^{3+1}$ an initial data on a cone with fall-off $\phi^{(l)}_0\sim r^{-k}$ ($k$ not belonging to an exceptional set) produces $\phi^{(l)}|_{r\leq1}\sim t^{-k-l}$ (\cite{luk_tale_2021}). This is hardly surprising if we accept the above picture, as such an $l$ mode solution has to vanish to order $l$ at the origin, and its (rescaled) restriction to $I^+$ will vanish to the same order in $\frac{r}{t}$ coordinates. Thus, the solution will be $t^{-k}h(\frac{r}{t})$ near $I^+$ with $h$ vanishing to order $l$. Restricting to compact regions of $r$ yields the extra decay.
	\end{remark}
	\begin{remark}[Exceptional cancellation]\label{exceptional cancellation}
		For $l\geq1$ modes in $\R^{3+1}$ (reducing higher odd dimensions to this case yield similar results) as for the $\R^{1+1}$ case, we have some exceptional cancellations. The $\phi^{(1)}_0\sim ar^{-1}+br^{-2}+cr^{-3}$ initial data supported on $l=1$ mode will create a tail proportional to $c$. The first two terms do not create a tail. This does not extends to forcing, just like before as can be seen from  \eqref{exact formula in n dim}. 
	\end{remark}
	
	Finally, let us note, that the conormality statement passes also to the nonlinear case (see \cite{hintz_semilinear_2015},\cite{hintz_stability_2020}).

	\subsection{A digression about anisotropic systems}\label{anisotropic section}
	In a system with multiple speeds, it is no longer sufficient to consider the asymptotics in the 3 regions $\mathcal{N}_0,\mathcal{N}_\scri,\mathcal{N}_+$. Indeed consider a system with wave speeds $1,2$ and corresponding field $\phi_1,\phi_2$. To get a full description of the asymptotic regions, we need to compactify $\R^{n+1}$ radially and than blow up the spheres at infinity described by $\abs{x}/t\in{1,2}$, see \cref{fig:anisotropic diagram}. Assuming that the solution is going to be conormal function on this manifold, we see that some of the necessary inhomogeneous problems have already been studied in \cref{table:decays}. In particular, note that $F|_{\mathcal{N}_{e)}}$ is fully understood. For such forcing, even though $\phi_1|_{\mathcal{N}_{d)}}$ is not included in \cref{table:decays}, that's only because its expansion is simply constant with respect to $2t-r$ coordinate at $r/t=2$, and the fall-off is same as at $\phi_1|_{\mathcal{N}_{e)}}$. The inhomogeneous problem not considered before, is if $F$ has singular support to the future or past of a given wave's null infinity corresponding to $\text{tail}(F|_{\mathcal{N}_{d)}})\implies\text{tail}(\phi_1|_{\mathcal{N}_{a)}-\mathcal{N}_{d)}}),\,\text{tail}(F|_{\mathcal{N}_{b)}})\implies\text{tail}(\phi_2|_{\mathcal{N}_{a)}-\mathcal{N}_{d)}})$.

	\begin{figure}[h]
	\centering{
		\resizebox{.4\textwidth}{!}{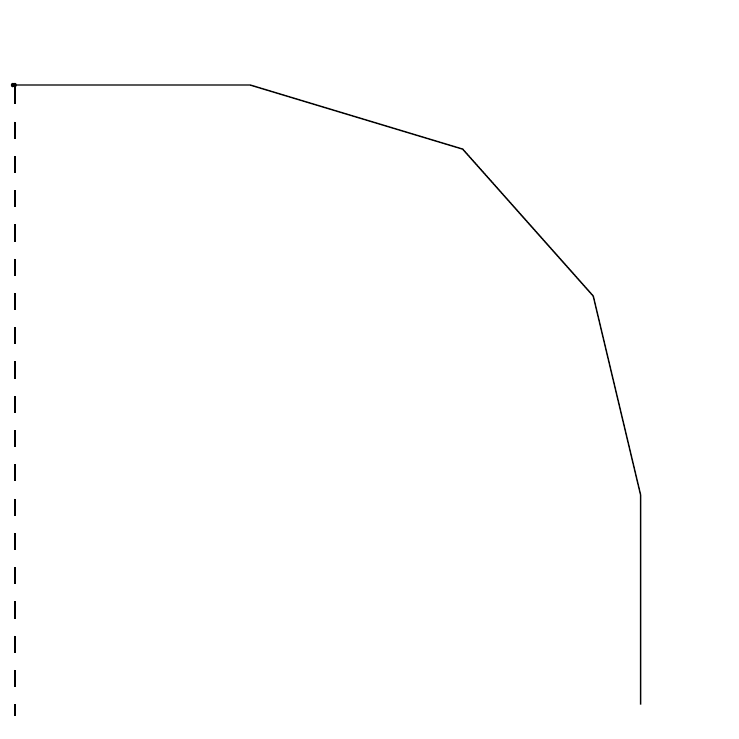}
		\caption{Important asymptotic regions in Minkowski corresponding to a two speed wave system.}
		\label{fig:anisotropic diagram}
	}
	\end{figure}
	
	\subsection{Estimates in higher dimensions}
	In this section, we will present some of the known estimates for wave propagation and discuss their consistency and optimality with respect to the heuristic above, in particular compare them to \cref{table:decays}.
	\subsubsection{John's $L^\infty$ estimates}\label{johns estimate}
	In \cite{john_blow-up_1979} and \cite{glassey_existence_1981} (see also \cite{del_santo_global_1997}) John and Glassey proved\footnote{ Indeed, these results can be read off from Kirchhoff's formula.} that the solution for the inhomogeneous wave equation in $\R^{n+1}$, $n=2,3$ with $\phi_0=\phi_1=0,\supp F\subset\{r<t\}$ satisfies
	\begin{equation}
		\begin{gathered}\label{john l infty}
			\norm{\jpns{v}^\alpha\jpns{u}^\beta\phi}_{L^\infty}\lesssim_{\alpha,\beta,\gamma,\delta}\norm{\jpns{v}^{\gamma}\jpns{u}^{\delta}F}_{L^\infty}
		\end{gathered}
	\end{equation}
	for $\alpha,\beta,\gamma,\delta>0$, $\alpha<(n-1)/2$ and
	\begin{equation*}
		\begin{gathered}
			\beta=\gamma-\frac{n+3}{2}+\min(1,\delta)
		\end{gathered}
	\end{equation*}
	with the additional constraint $\beta<\frac{1}{2}$ for $n=2$.
	
	Let's see how estimate \cref{john l infty} fits with the heuristics of \cref{table:decays}. Fix $F=v^{-s}u^{-\sigma}$. For $\sigma>1$, we argued that we have $\phi\sim v^{-\frac{n-1}{2}}u^{\frac{n+1}{2}-s}$. This saturates the inequality, with the caveat, that $\alpha=\frac{n-1}{2}$ case is not covered. Furthermore, the inequality also shows that $\sigma<1$ implies that the decay is lowered by $\min(0,\sigma-1)$, exactly as described in the previous section. In particular $F|_{\mathcal{N}_+}$ near $I^+$ generates late time behaviour and not $F|_{\mathcal{N}_\scri}$ near $\scri$ in this case.
	
	\subsubsection{Strichartz estimates: inhomogeneous and weighted}
	\paragraph{Inhomogeneous:}
	The original inhomogeneous Strichartz estimates say
	\begin{equation*}
		\begin{gathered}
			\norm{\phi}_{L_t^qL_x^r}\lesssim_{q,r}\norm{F}_{L_t^{\bar{q}^\prime}L_x^{\bar{q}^\prime}}
		\end{gathered}
	\end{equation*}
	where the dimensional condition
	\begin{equation*}
		\begin{gathered}
			\frac{1}{q}+\frac{n}{r}=\frac{1}{\bar{q}^\prime}+\frac{n}{\bar{r}^\prime}-2
		\end{gathered}
	\end{equation*}
	and the admissibility 
	\begin{equation*}
		\begin{gathered}
			\frac{1}{q}+\frac{n-1}{2r}\leq\frac{n-1}{4}\\
			q,r\geq2
		\end{gathered}
	\end{equation*}
	(same for $\bar{q},\bar{r}$) hold. This result, including the endpoint case was proven in the seminal paper of Keel and Tao \cite{keel_endpoint_1998}, and is proved via the Christ Kiselev lemma from the homogeneous counterpart. 
	
	The estimate is strongly attached to the $t,x$ foliation, so it's harder to compare to the heuristics.	In the case $q=r$, the norms are foliation independent, thus the results are more clear. The only admissible exponent is $q=r=2\frac{n+1}{n-1}$. This means that $\phi$ (up to $\epsilon$ loss) must decay like $t^{-\frac{n(n-1)}{2(n+1)}}$ toward $\scri$ and $t^{-\frac{n-1}{2}}$ towards $I^+$ when $F$ has decay like $t^{-\frac{n(n+3)}{2(n+1)}}$ and $t^{-\frac{n+3}{2}}$ respectively. According to the previous discussion, we see that this is optimal towards $I^+$ since $\frac{n+3}{2}-2=\frac{n-1}{2}$, however suboptimal near $\scri$ as $\frac{n(n-1)}{2(n+1)}<\frac{n-1}{2}+\min(0,\frac{n(n+3)}{2(n+1)}-\frac{n-1}{2}-1)$. Indeed, $\phi$ is expected to decay like $t^{-\frac{n-1}{2}}$ there.
	 
	The results of \cite{harmse_lebesgue_1990} (see \cite{foschi_inhomogeneous_2005} also) improve the situation near $\scri$ with an estimate
	\begin{equation*}
		\begin{gathered}
			\norm{\phi}_{L^{p}}\lesssim_p\norm{F}_{L^{\tilde{p}^\prime}}
		\end{gathered}
	\end{equation*}
	where $2\frac{n+1}{n-1}\geq p>\frac{2n}{n-1}$ and $\frac{1}{\tilde{p}^\prime}=1-\big(\frac{n-1}{n+1}-\frac{1}{p}\big)$. Indeed, such an inequality cannot be derived from the homogeneous counterpart. Analysing this estimate at the endpoint case (where it does not hold), shows that this improved estimate captures optimal decay both at $\scri$ and $I^+$. Indeed up to arbitrary small losses, the end point case is saturated by the following asymptotics
	\begin{equation*}
		\begin{gathered}
			F|_{\mathcal{N}_\scri}\sim t^{-\frac{n^2+4n-1}{2(n+1)}}\quad F|_{\mathcal{N}_+}\sim t^{-\frac{n^2+4n-1}{2n}}\\
			\phi|_{\mathcal{N}_\scri}\sim t^{-\frac{n-1}{2}}\quad \phi|_{\mathcal{N}_+}\sim t^{-\frac{(n+1)(n-1)}{2n}}.\\	
		\end{gathered}
	\end{equation*}
	Note, that although, this is optimal decay for $\phi$ at both $I^+$ and $\scri$, $F$ could decay less quickly toward $\scri$. In order to capture optimality in this second sense too, we need to introduce weights.
	
	\paragraph{Weighted:} The weighted Strichartz estimate, made to cure the problems of the previous section are proved in \cite{georgiev_weighted_1997} (see \cite{del_santo_global_1997}):
	\begin{equation*}
		\begin{gathered}
			\norm{\jpns{v}^\alpha\jpns{u}^\beta\phi}_{L^\mu}\lesssim_{\alpha,\beta,\gamma,\delta} \norm{\jpns{v}^\gamma\jpns{u}^\delta F}_{L^\lambda}
		\end{gathered}
	\end{equation*}
	provided the dimensional and Strichartz conditions hold
	\begin{equation*}
		\begin{gathered}
			\frac{1}{\lambda}+\frac{1}{\mu}=1,\qquad \space \frac{n-3}{2}<\frac{n}{\mu}-\lambda
		\end{gathered}
	\end{equation*}
	with $\alpha,\beta,\gamma,\delta>0$, $\alpha<(n-1)/2-n/\mu$ and
	\begin{equation*}
		\begin{gathered}
			\beta=\gamma-\alpha-2+\frac{n}{\lambda}-\frac{n+1}{\mu}+\min(1,\delta+\frac{1}{\lambda})<\frac{n-1}{2}-\frac{n}{\mu}.
		\end{gathered}
	\end{equation*}
	These are isotropic, thus are easier to understand their relation ot the heuristics, but we can further factor out some $u,v$ terms and introduce b-Lebesgue spaces $\norm{\cdot}_{L^\mu}=\norm{\jpns{v}^{n/\mu}u^{1/\mu}\cdot}_{L_b^{\mu}}$ and set $\bar{\alpha}=\alpha-\frac{n}{\mu},\bar{\beta}=\beta-\frac{1}{\mu},\bar{\gamma}=\gamma-\frac{n}{\lambda},\bar{\delta}=\delta-\bar{\delta}$. Then, the above inequality says
	\begin{equation*}
		\begin{gathered}
			\norm{\jpns{v}^{\bar{\alpha}}\jpns{u}^{\bar{\beta}}\phi}_{L^\mu_b}\lesssim_{\alpha,\beta,\gamma,\delta}\norm{\jpns{v}^{\bar{\gamma}}\jpns{u}^{\bar{\delta}}F}_{L^\lambda_b}
		\end{gathered}
	\end{equation*}
	for $\bar{\alpha}<(n-1)/2$
	\begin{equation*}
		\begin{gathered}
			\bar{\beta}=\bar{\gamma}-\bar{\alpha}-2+\min(1,\bar{\delta})<\frac{n-1}{2}-\frac{n-1}{\mu}.
		\end{gathered}
	\end{equation*}
	This is the same as John's $L^\infty$ estimate up to an arbitrary small loss in $\bar{\alpha},\bar{\beta}$ and stronger restriction on maximum for $\bar{\beta}$.
	
	\subsubsection{Morawetz estimate}
	There are many form of this estimate, one particular case is for $\phi$ solution of linear inhomogeneous equation with trivial data (\eqref{linear})  satisfies
	\begin{equation*}
		\begin{gathered}
			\int\d x\d t\, \frac{\abs{\bar{\partial}\phi}^2}{\jpns{r}^{1+\epsilon}}\lesssim_\epsilon\int\d x\d t\, F^2\jpns{r}^{1+\epsilon},
		\end{gathered}
	\end{equation*}
	with $\bar{\partial}=(\partial,1/r)$ for all $\epsilon>0$. For $F\sim v^{-s}u^{-\sigma}$, the right hand side converges only if $2s>n+1$ and $2(\sigma+s)>n+1$. Let's take $\sigma$ large. This in turn yields that $\partial\phi$ cannot decay slower than $v^{-\frac{n-1-\epsilon}{2}}$ towards $\scri$ and $v^{-\frac{n-\epsilon}{2}}$ towards $I^+$. This is optimal near $\scri$,  but doesn't yield best decay at $I^+$, therefore one can only improve on it with additional weights.
	
	\subsubsection{$r^p$ estimate}
		The $r^p$ method of Dafermos and Rodnianski \cite{dafermos_new_2010} works extremely well with the \cref{table:decays}. Consider first the radial case in $\R^{3+1}$ for the inhomogeneous problem with 0 initial data. Then, we have
		\begin{equation*}
			\begin{gathered}
				\sup_u\bigg(\int_{\Sigma_u} \d v\, r^p(\partial_v(r\phi))^2\bigg)^{1/2}\lesssim_p\int\d u \bigg(\int_{\Sigma_u}\d v\, r^p(rF)^2\bigg)^{1/2}.
			\end{gathered}
		\end{equation*}
		Which bounds the incoming radiation in term of the forcing at exactly the right rate (up to $\epsilon$ loss). Obtaining optimal decay for the energy and bounding the other components ($\partial_u\phi,\phi$) is given by a usage of energy, Morawetz and Hardy estimates.
		
		Outside of radial symmetry, and in different dimensions, one has to use $(r^2\partial_v)^j\psi$ and its equation of motion to get almost sharp decay (\cite{angelopoulos_late-time_2018}).

	\section{Finding critical exponents}\label{finding critical exponents}
	This section is a heuristic guide to find critical exponents for non-linear wave equations. The statements here are not precise, and the equations are only schematic. We will find that each system studied has a global solution iff it satisfies the asymptotic decay condition \cref{asymptotic decay condition}.
	
	We will find using the above heuristics the critical exponents for the equations discussed in the introduction. Let's use the notation $\psi=t^{\frac{n-1}{2}}\phi$, so that $\Box\phi=f\implies\tilde{\Box}\psi=t^{\frac{n-1}{2}}f$, where $\tilde{\Box}$ is the operator $t^{-2}L$ in \cite{hintz_stability_2020}. Importantly, $\tilde{\Box}$ has leading order term (in terms of decay) $\partial_v\partial_u$ near $\scri$. Furthermore, let's focus on odd dimensions, to ignore tail effects coming from resonances/Green's function of the spacetime. Nevertheless, most of the results survive outside this setting, as these resonances produce quickly decaying contributions.
	
	We will use a simply lemma to reproduce the curves presented in \cref{zoology}.
	\begin{lemma}\label{min max lemma}
		For real numbers $a,b,c$ with $c>1$, the equation
		\begin{equation*}
			\begin{gathered}
				x=a+\min(0,b+cx)
			\end{gathered}
		\end{equation*}
		has solutions iff $b+ca\geq0$.
		
	\end{lemma}
	\begin{proof}
		The solutions are the $x$ values of the intersection  $\{y=x\}\cap\{y=a+\min(0,b+cx)\}$. As $c>1$, there must be an intersection while the minimum is attained at $0$. Thus $x=a$ is a solution with $\min(0,b+cx)=0$, which implies $b+ca\geq0$.
	\end{proof}
	
	\subsection{Strauss conjecture}
	Rewriting \eqref{Strauss}, we get 
	\begin{equation*}
		\begin{gathered}
			\tilde{\Box}\psi=t^{\frac{n-1}{2}}\abs{\psi}^q t^{-q(n-1)/2}.
		\end{gathered}
	\end{equation*}
	The asymptotic decay condition is satisfied, if iterative solutions have the same decay at infinite. Alternatively, if the exact solution has asymptotic consistent with tail generation from \cref{table:decays}. Say, $\psi$ has a tail $\sim  u^{-s}$, that is $\psi|_{\mathcal{N}_+}\sim u^{-s}$, with well defined radiation field (ie. $\psi_{\mathcal{N}_\scri}\sim1$). The constraints on the fall-off from $\abs{\psi}^q$ term are $s\leq (q-1)\frac{n-1}{2}-1$ from $\scri$ implies , while the contribution form $I^+$ says $s\leq (q-1)\frac{n-1}{2}-2+qs$. The correct value of $s$ will saturate one of these inequalities, in conclusion, we get
	\begin{equation*}
		\begin{gathered}
			s=(q-1)\frac{n-1}{2}-1+\min(0,qs-1).
		\end{gathered}
	\end{equation*}
	\cref{min max lemma}  implies that a solution exists only if $((q-1)\frac{n-1}{2}-1)q>1$, which coincide with the Strauss exponent.
	\subsection{Glassey}
	Rewriting \eqref{Glassey}, ignoring terms where the derivative acts on the rescaling, we have $\tilde{\Box}\psi=t^{\frac{n-1}{2}(1-q)}\abs{\partial\psi}^q$. Solving this iteratively
	\begin{equation*}
		\begin{gathered}
			\tilde{\Box}\psi_0=0\\
			\tilde{\Box}\psi_1=t^{\frac{n-1}{2}(1-q)}\abs{\partial\psi_0}^q\\
			\tilde{\Box}\psi_2=t^{\frac{n-1}{2}(1-q)}\abs{\partial\psi_1}^q
		\end{gathered}
	\end{equation*}
	we see from \cref{table:decays} that $\psi_i|_{\mathcal{N}_+}\sim t^{\sigma_i}$ with $\sigma_0=0,\, \sigma_1=\max(0,1+\frac{n-1}{2}(1-q)+\sigma_0 q),\sigma_2=\max(0,1+\frac{n-1}{2}(1-q)+\sigma_1 q)$. It's easy to see, that this sequence terminates iff $\sigma=\max(0,1+\frac{n-1}{2}(1-q)+\sigma q)$ has a solution.
	, assuming $\psi$ has asymptotic $v^{\sigma},\sigma\geq0$ at $\scri$, \cref{table:decays} implies that the nonlinearity forces $\sigma=\max(0,1+\frac{n-1}{2}(1-q)+\sigma q)$. By \cref{min max lemma}, this is equivalent to $\frac{n-1}{2}(1-q)<-1$. 
	
	To check that this is sufficient condition, say $\psi$ has $u^{-s}$ tail, $\psi|_{\mathcal{N}_+}\sim t^{-s}$. Than, conormality assumption implies that $\partial\psi$ decays one order faster and \cref{role of asymptotics} implies $s=(q-1)\frac{n-1}{2}-1-\max(0,1-q(s+1))$ which has a solution $s=(q-1)\frac{n-1}{2}$ if $q>q_{Glassey}$.
	
	\begin{remark}\label{difference between Strauss and Glassey}
		This shows the strong contrast between the Strauss and Glassey conjectures. The second is more of the flavour of a \textit{weak null condition}, as it relies on the asymptotic towards $\scri$ only. Furthermore, as solutions for \eqref{Strauss} with $q_{Glassey}<q<q_{Strauss}$ only fail to exist because of condition on behaviour toward $I^+$, one expects that these have semi global solution in regions $\{t-r\in[0,u_0]\}$ for small enough data, while the same is not true for \eqref{Glassey} with $q<q_{Glassey}$.
	\end{remark}
	\subsection{Strauss system}\label{strauss system}
	Let's start with the analysis of the more involved known results, that is the 2 component system \cref{2 variable Strauss}
	\begin{equation*}
		\begin{gathered}
			\tilde{\Box}\psi_1=t^{\frac{n-1}{2}(1-q_1)}\abs{\psi_2}^{q_1}\\
			\tilde{\Box}\psi_2=t^{\frac{n-1}{2}(1-q_2)}\abs{\psi_1}^{q_2}.
		\end{gathered}
	\end{equation*}
	If the system has global solution which is conormal, then, we can assume $\psi_i|_{\mathcal{N}_\scri}\sim t^{\sigma_i},\psi_i|_{\mathcal{N}_+}\sim t^{-s_i}$ at leading order. The asymptotic decay condition only holds, if these rates are consistent with the tail creation of \cref{table:decays}
	\begin{equation*}
		\begin{gathered}
			\sigma_1=\max(0,\frac{n-1}{2}(1-q_1)+1+\sigma_2q_1)\\
			\sigma_2=\max(0,\frac{n-1}{2}(1-q_2)+1+\sigma_1q_2)\\
			s_1=\frac{n-1}{2}(q_1-1)-1+\min(-q_1\sigma_2,q_1s_2-1)\\
			s_2=\frac{n-1}{2}(q_2-1)-1+\min(-q_2\sigma_1,q_2s_1-1)\\.
		\end{gathered}
	\end{equation*}
	The claim, is that the PDE has global solutions iff this equation has a solution. To see that this gives the correct curve, we may assume without loss of generality that $q_1\geq q_2>1$. Then, its clear that the first two equations only have a solution if $\sigma_1=0$ \footnote{which implies  $\frac{n-1}{2}(1-q_1)+1+q_1\max(0,\frac{n-1}{2}(1-q_2)+1)<0$ must hold, in particular $q_1\geq q_{\text{Glassey}}$}, so $\sigma_2=\max(0,\frac{n-1}{2}(1-q_2)+1)$. Also,
	\begin{equation*}
		\begin{gathered}
			\min(-q_1\sigma_2,q_1s_2-1)=\min(0,q_1(\frac{n-1}{2}(q_2-1)-1),q_1s_2-1)=\min(0,q_1s_2-1),
		\end{gathered}
	\end{equation*}
	thus we are left with
	\begin{equation*}
		\begin{gathered}
			s_1=a+\min(0,q_1s_2-1)\\
			s_2=b+\min(0,q_2s_1-1),
		\end{gathered}
	\end{equation*}
	with $a=\frac{n-1}{2}(q_1-1)-1,b=\frac{n-1}{2}(q_2-1)-1>0$. It's easy to see that $q_1>q_2\implies q_2a\geq q_1b$. Substituting for $s_1$ gives
	\begin{equation*}
		\begin{gathered}
			s_2=b+\min(0,q_2a-1,q_2a-1-q_2+q_1q_2s_2).
		\end{gathered}
	\end{equation*} 
	If $q_2a\geq1$, then \cref{min max lemma} yields a solution if $q_2a-1-q_2+q_2q_1b\geq0$. For $q_2a<1$, there is a solution if $-q_2+q_2q_1(b+q_2a-1)$ is non-negative. However, we can use $q_2a\geq q_1b$ to conclude that $-q_2+q_2q_1(b+q_2a-1)\leq q_2(1+q_1)(q_2a-1)<0$, which is a contradiction. Therefore, we have a solution in the region 
	\begin{equation*}
		\begin{gathered}
			\{q_1>q_2\}\cap\{q_2a>1\}\cap\{q_2a-q_2(1-q_1b)>1\}=\{q_1>q_2\}\cap\{q_2a-q_2(1-q_1b)>1\}
		\end{gathered}
	\end{equation*}
	where the equality can be checked by a few lines of algebra. One may check that in this region $s_2=b,\,s_1=a+\min(0,q_1b-1)$ solve the original equation. Furthermore, $q_2a-q_2(1-q_1b)>1\iff (q_1q_2-1)\frac{n-1}{2}>2+q_1+q_2^{-1}$ is the condition proved in \cite{del_santo_global_1997}.\footnote{Of course, there are additional requirements due to well-posedness of the equation, see \cref{regularity vs decay}.}
	
	Turning to a general system \cref{n variable Strauss}:
	\begin{equation*}
		\begin{gathered}
			\tilde{\Box}\psi_i=\sum_j a_{ij}\abs{\psi_j}^{c_{ij}}t^{\frac{n-1}{2}(1-c_{ij})}
		\end{gathered}
	\end{equation*}
	where $a_{ij}\in\{0,1\}$. Let's assume conormal solution, with $\psi_i|_{\mathcal{N}_\scri}\sim t^{\sigma_i},\psi_i|_{\mathcal{N}_+}\sim t^{-s_i}$ leading order. The asymptotic decay condition holds if the system
	\begin{equation}\label{n variable Strauss condition}
		\begin{gathered}
			\sigma_i=\max\bigg(\{0\}\cup\Big\{\frac{n-1}{2}(1-c_{ij})+1+\sigma_jc_{ij}:a_{ij}\neq0\Big\}\bigg)\\
			s_i=\min\Big\{\frac{n-1}{2}(c_{ij}-1)-1+\min(0,-c_{ij}\sigma_j,c_{ij}s_j-1):a_{ij}\neq0\Big\}
		\end{gathered}
	\end{equation}
	has a solution. Indeed, we prove that this is equivalent to the small data well-posedness of \eqref{n variable Strauss} in a regime. For well posedness see \cref{n strauss main theorem}, while for singularity formation see \cref{blow up for n strauss theorem}.
	\subsection{Strauss Glassey mix}
	\subsubsection{Scalar}
	Let's start with discussing \eqref{Strauss Glassey}:
	\begin{equation*}
		\begin{gathered}
			\tilde{\Box}\psi=t^{\frac{n-1}{2}(1-q_1)}\abs{\partial_t\psi}^{q_1}+t^{\frac{n-1}{2}(1-q_1)}\abs{\psi}^{q_2}
		\end{gathered}
	\end{equation*}
	where we neglected terms that differentiate the weights in $\partial_t\phi$. Assume global solution with boundary regularity and leading term $\psi|_{\mathcal{N}_\scri}\sim t^{\sigma},\psi|_{\mathcal{N}_+}\sim t^{-s}$. The system satisfies asymptotic decay condition near $\scri$ is equivalent to the solvability of
	\begin{equation*}
		\begin{gathered}
			\sigma=\max(0,\frac{n-1}{2}(1-q_1)+1+q_1\sigma,\frac{n-1}{2}(1-q_2)+q_2\sigma)
		\end{gathered}
	\end{equation*}
	This has a solution if $q_1,q_2>q_{\text{Glassey}}$ and $\sigma=0$. Near $I^+$, we need
	\begin{equation*}
		\begin{gathered}
			s=\min\bigg(a+\min\big(0,q_1(s+1)-1\big),b+\min\big(0,q_2s-1\big)\bigg),
		\end{gathered}
	\end{equation*}
	with $a,b$ as above. For $q_1\geq q_2$, it's not hard to check that $q_2\geq q_{\text{Strauss}}$ is a necessary and sufficient condition for solutions to exist. For $q_1<q_2$, the $s$ equation gives
	\begin{equation*}
		\begin{gathered}
			s=a+\min(0,q_1(s+1)-1,b-a+q_2s-1).
		\end{gathered}
	\end{equation*}
	By \cref{min max lemma}, we have a solution if \begin{equation*}
		\begin{gathered}
			q_1(a+1)-1,b-a+q_2a-1\geq0.
		\end{gathered}
	\end{equation*}
	The first is trivially satisfied, because $a\geq0$, while a few lines of algebra shows that the second is exactly the critical curve discussed in \cref{zoology}.
	
	\subsubsection{System}\label{strauss glassey system}
	We can rewrite \eqref{Strauss Glassey system} as
	\begin{equation*}
		\begin{gathered}
			\tilde{\Box}\psi_1=t^{\frac{n-1}{2}(1-q_1)}\abs{\psi}^{q_1}\\
			\tilde{\Box}\psi_2=t^{\frac{n-1}{2}(1-q_2)}\abs{\partial_t\psi_1}^{q_2}.
		\end{gathered}
	\end{equation*}
	Let's assume fall-off $\psi_i|_{\mathcal{N}_\scri}\sim t^{\sigma_i},\psi_i|_{\mathcal{N}_+}\sim t^{-s_i}$ at leading order. The system satisfies the asymptotic decay condition if
	\begin{equation*}
		\begin{gathered}
			\sigma_1=\max(0,\sigma_2q_1-a)\\
			\sigma_2=\max(0,\sigma_1q_2-b)\\
			s_1=a+\min(-q_1\sigma_2,q_1s_2-1)\\
			s_2=b+\min(-q_2\sigma_1,q_2(s_1+1)-1),
		\end{gathered}
	\end{equation*}
	is solvable where $a,b$ as before. Frome here, it's straightforward to find the allowed region of $q_1,q_2$, but we included it for completeness. Focus on the $q_1>q_2$ range. As in the Strauss system, we get $\sigma_1=0,\,\sigma_2=\max(0,-b)$  and $\min(-q_1\sigma_2,q_1s_2-1)=\min(0,q_1s_2-1)$ simplifying the system to
	\begin{equation*}
		\begin{gathered}
			s_1=a+\min(0,q_1s_2-1)\\
			s_2=b+\min(0,q_2(s_1+1)-1).
		\end{gathered}
	\end{equation*}
	As for the Strauss system \cref{strauss system}, we may substitute for $s_1$
	\begin{equation*}
		\begin{gathered}
			s_2=b+\min(0,q_2(a+1)-1,q_2(a+1)-1-q_2+q_1q_2s_2)=b+\min(0,q_2(a+1)-1-q_2+q_1q_2s_2),
		\end{gathered}
	\end{equation*}
	as $q_2(a+1)-1\geq0$. Using \cref{min max lemma} we conclude the condition $q_2(a+1)-1-q_2+q_1q_2b>0$. Once again, an explicit solution is found in this region by setting $s_2=b$. This agrees with the curve found in \cite{hidano_combined_2016}.
	
	On the range $q_1<q_2$, the simplification for $\sigma$ similarly gives $\sigma_2=0,\,\sigma_1=\max(0,-a)$, but the rest of the system is
	\begin{equation*}
		\begin{gathered}
			\begin{cases}
				s_1=a+\min(0,q_1s_2-1)\\
				s_2=b+\min(0,q_2a,q_2(s_1+1)-1)
			\end{cases}\\
			\implies s_2=b+\min(0,q_2a,q_2(a+1)-1-q_2+q_2q_1s_2)
		\end{gathered}
	\end{equation*}
	Note, that the $s_1+1$ instead of $s_1$ prohibits the previous simplifications in the second minimum. As long as $a>0$, we can use the previous argument to find the same allowed region. For $a<0$ we find
	\begin{equation*}
		\begin{gathered}
			s_2=b+aq_2+\min(0,-1+q_1q_2s_2).
		\end{gathered}
	\end{equation*}
	
	There is a solution if  $q_2q_1(b+q_2a)-1>0$. \textit{To the author's knowledge, this is the first place where this part of the critical curve is derived for the above problem.} The region of study in \cite{hidano_combined_2016} is detailed in the caption of \cref{fig: Glassey Strauss system}. In this region, setting $s_2=b+q_2a$ indeed gives a solution for $s_1,s_2$. At this point, it's good practice, to check that the condition for the existence of $\sigma_i$, that is $b+q_2a>0$ is still satisfied. Indeed this is the case.

	\begin{figure}[h]
		\centering{
			\resizebox{.4\textwidth}{!}{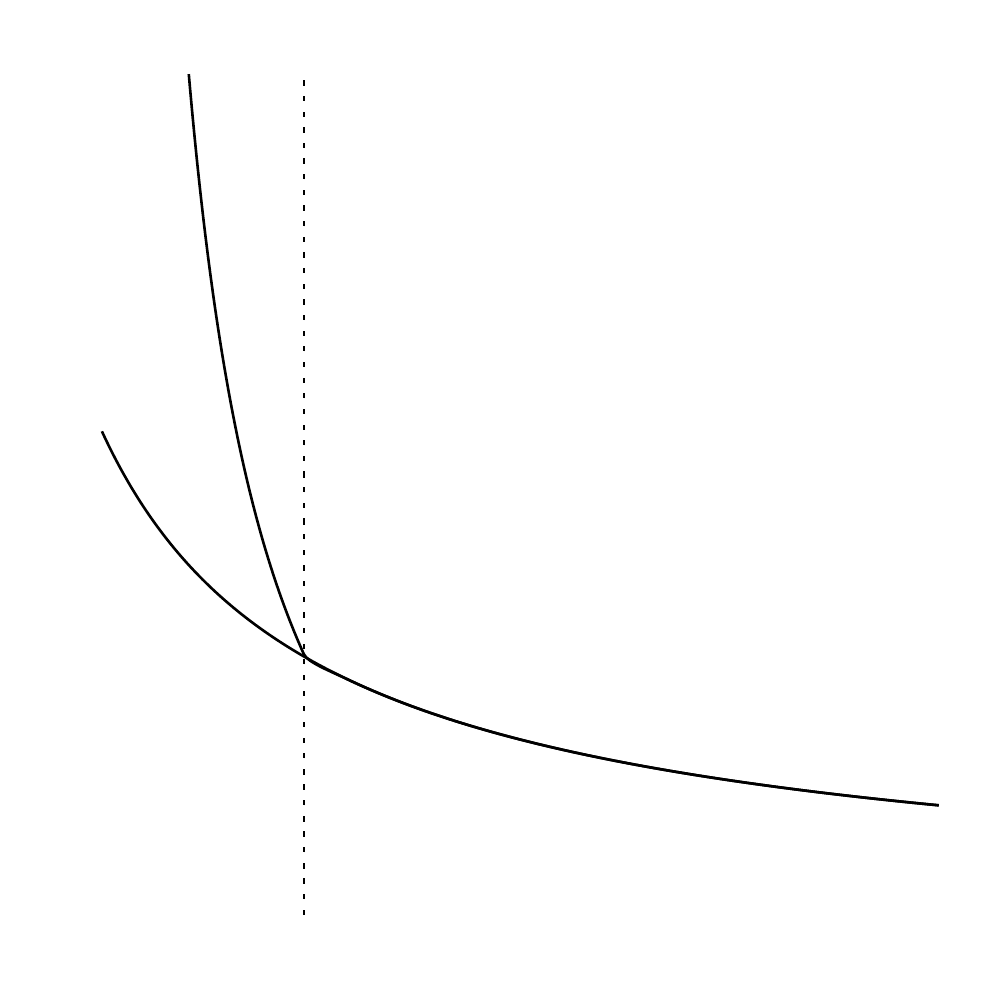}
			\caption{Critical curve for Glassey Strauss system \eqref{strauss glassey system} in $n=3$. Global solutions are expected using the heuristics above curve b), but not below. This was proved for $q_1>q_{Glassey}$ and below the curve a), however $q_1<q_{Glassey}$ and above curve a) have not been investigated so far.}
			\label{fig: Glassey Strauss system}
		}
	\end{figure}
	
	\subsection{Null conditions}\label{null condition}
	
	\paragraph{a)}As a warm-up, consider \eqref{integrable null}:
	\begin{equation*}
		\begin{gathered}
			\tilde{\Box}\psi=t^{-\frac{3-1}{2}(2-1)}\partial_v\psi\partial_u\psi+t^{-\frac{3-1}{2}(2-1)-2}\slashed{\nabla}\psi\cdot\slashed{\nabla}\psi
		\end{gathered}
	\end{equation*}
	where as usual, we dropped terms where the derivative hit the weight. We see, that $\psi|_{\mathcal{N}_\scri}\sim 1,\psi|_{\mathcal{N}_+}\sim t^{-1}$ is a good leading order behaviour. 
	
	\paragraph{b)} The weak null system \eqref{weak null}
	\begin{equation*}
		\begin{gathered}
			\tilde{\Box}\psi_1=t^{-1}(\partial_t\psi_2)^2\\
			\tilde{\Box}\psi_2=t^{-1}\partial_v\psi_1\partial_u\psi_1+t^{-3}\slashed{\nabla}\psi_1\cdot\slashed{\nabla}\psi_1
		\end{gathered}
	\end{equation*}
	has the following consistent leading order behaviour (under the discussion of \cref{role of asymptotics})	$\psi_1|_{\mathcal{N}_\scri}\sim \log t,\psi_1|_{\mathcal{N}_+}\sim \log t,\psi_2|_{\mathcal{N}_\scri}\sim1,\psi_2|_{\mathcal{N}_+}\sim (\log t)^2t^{-1}$. In particular, note that, both systems are far from being borderline.
	
	\paragraph{c)}
	A typical hierarchical weak null system considered in $\R^{3+1}$ by Keir \cite{keir_global_2019} is
	\begin{equation*}
		\begin{gathered}
			\Box\phi_1=(\partial_t\phi_2)^3\\
			\Box\phi_2=\partial_t\phi_1\partial_t\phi_2+(\partial_t\phi_1)^2\\
		\end{gathered}
	\end{equation*}
	which yields
	\begin{equation*}
		\begin{gathered}
		\tilde{\Box}\psi_1=t^{-2}(\partial_t\psi_2)^3\\
		\tilde{\Box}\psi_2=t^{-1}\partial_t\psi_1\partial_t\psi_2+t^{-1}(\partial_t\psi_1)^2
		\end{gathered}
	\end{equation*}

	This system \textit{does not} satisfy the asymptotic decay condition, because the leading order behaviour for $\psi_2$ has growing logarithmic term at each iteration. However, note that for fixed $\psi_1$, the equation for $\psi_2$ is linear, with leading order behaviour $-\partial_u\partial_v\psi_2=t^{-1}\partial_u\psi_1\partial_u\psi_2$. For small data problem and bounded conormal behaviour for $\psi_1$ this ODE yields $v^\epsilon$ growth for $\partial_u\psi_2$. Similar conclusion applies for $\psi_3$. Indeed, this is a paraphrasing of the discussion found in \cite{keir_weak_2018}.
	
	We conclude, that the asymptotic decay condition does not see the \textit{almost linear} structure present in the above equation, which is a key for global wellposedness.

	\paragraph{d)} Finally, turn to \eqref{Strauss null}:
	\begin{equation*}
		\begin{gathered}
			\tilde{\Box}\psi_1=t^{-\frac{n-1}{2}(q_1-1)}\abs{\psi_2}^{q_1}\\
			\tilde{\Box}\psi_2=t^{-1}\abs{\partial_t\psi_1}^{\frac{n+1}{n-1}}+t^{-\frac{n-1}{2}(q_2-1)}\abs{\psi_1}^{q_2}	.
		\end{gathered}
	\end{equation*}
	with assumed asymptotics $\psi_j|_{\mathcal{N}_\scri}\sim t^{\sigma_j},\,\psi_j|_{\mathcal{N}_+}\sim t^{-s_j},\, j\in\{1,2\}$. The asymptotic decay condition holds if\footnote{note, there are logarithmic losses that are explicitly introduced by the $\partial_t\psi_1$ term, but as we are not interested in border line cases, we do not discuss this. Indeed, by putting assuming power type behaviour, we implicitly ignore this contribution.}
	\begin{equation*}
		\begin{gathered}
			\sigma_1=\max(0,-a+q_1\sigma_2)\\
			\sigma_2=\max(0,-b+q_2\sigma_1,\frac{n+1}{n-1}\sigma_1)\\
			s_1=a+\min(0,-q_1\sigma_2,q_1s_2-1)\\
			s_2=\min\bigg(\min\Big(0,\frac{n+1}{n-1}(1+s_1)-1\Big),b+\min\Big(0,-q_2\sigma_1,q_2s_1-1\Big)\bigg)
		\end{gathered}
	\end{equation*}
	has a solution. The asymptotic at $\scri$ only close if $\sigma_1=0$, which implies $a\geq0$. Focus first on the case $b\geq0$, ie. $\sigma_2=0$. Then, as $s_2\leq0$, the system simplifies to
	\begin{equation*}
		\begin{gathered}
			s_2=\min\bigg(0,\frac{n+1}{n-1}a-1+q_1\frac{n+1}{n-1}s_2,b-1+q_2(a-1)+q_1q_2s_2\bigg).
	\end{gathered}
	\end{equation*}
	This only has a solution if
	\begin{equation}\label{hidano null}
		\begin{gathered}
			\frac{n+1}{n-1}a-1\geq0\implies q_1\geq1+\frac{4n}{n^2-1}\\
			b-1+q_2(a-1)\geq0\implies\frac{n-1}{2}(q_1q_2-1)-2q_2>2,
		\end{gathered}
	\end{equation}
	which defines the critical curve for $b\geq0$.The second equation gives the curve found in \cite{hidano_global_2022}. Note however, that in \cite{hidano_global_2022}, the authors restrict to the set $q_1>1+\frac{3}{n-1}$ with $n=2,3$, which is necessary in light of the first condition above.

	For $q_2<q_{\text{Gl}}=\frac{n+1}{n-1}$, $\sigma_2=-b>0$ and $\min(-q_1\sigma_2,q_1s_2-1)=q_1s_2-1$, so the system simplifies to
	\begin{equation*}
		\begin{gathered}
			s_1=a-1+q_1s_2\\
			s_2=b+\min(0,q_2s_1-1).
		\end{gathered}
	\end{equation*}
	which implies
	\begin{equation*}
		\begin{gathered}
			s_2=b+\min(0,q_2(a-1)-1+s_2q_1q_2).
		\end{gathered}
	\end{equation*}
	A solution exists only if $q_2(a-1+q_1b)-1>0$. This is part of the curve for the 2-Strauss system \eqref{2 variable Strauss}.

	\begin{figure}[h]
		\centering{
			\resizebox{.4\textwidth}{!}{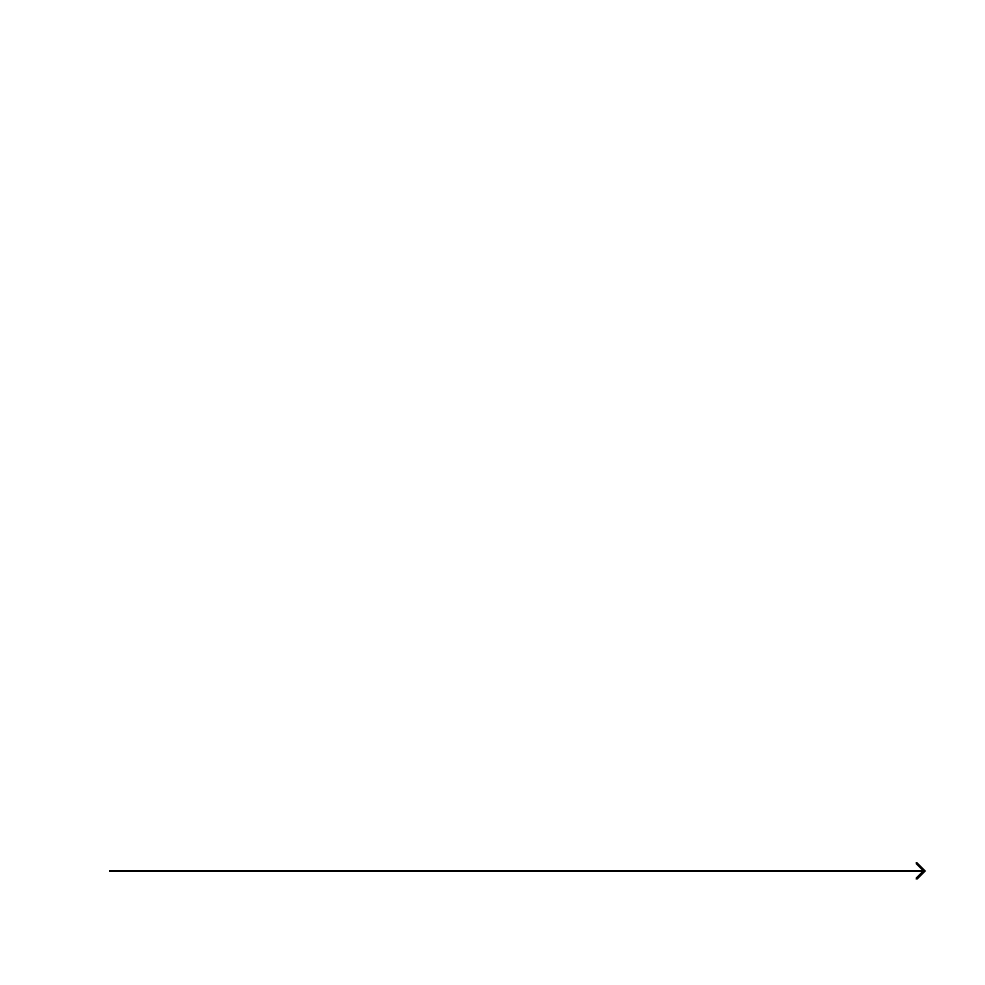}
			\caption{Critical curve for \eqref{Strauss} in $n=3$. Global solutions are expected using the heuristics above solid line, but not below. Existence was shown for $q_2\in[2,6]$ above the solid line, but breakdown of global solutions only below the dashed one in \cite{hidano_global_2022}.}
			\label{fig: hidano nullsystem}
		}
	\end{figure}

	\subsection{Weighted problems}
	\paragraph{Null type:}
	To quantify the amount of sub criticality (in terms of decay) let us investigate
	\begin{equation*}
		\begin{gathered}
			\Box\psi=t^{\alpha-1} u^\beta t^{-1}\psi u^{-1}\psi.
		\end{gathered}
	\end{equation*} 
	in $\R^{3+1}$ which corresponds to $\Box\phi=t^{\alpha}u^{\beta}\partial\phi\cdot\partial\phi$ where we ignored all the non-dominant terms. Assuming asymptotics $\psi|_{\mathcal{N}_\scri}\sim v^{\sigma},\,\psi|_{\mathcal{N}_+}\sim t^{-s}$, we get (using \cref{role of asymptotics})
	\begin{equation*}
		\begin{gathered}
			\sigma=\max(0,\alpha-1+2\sigma)\\
			s=1-\alpha+\min(0,s+(s+1)-1-\beta).
		\end{gathered}
	\end{equation*}
	This has solution if both $\sigma=0(\iff\alpha<1)$ and $s=1-\alpha(\iff \beta+2\alpha<2)$ hold.
	\paragraph{$\partial_t\phi$ nonlinearity:}
	In the recent work \cite{kitamura_semilinear_2022}, Kitamura studied the equation
	\begin{equation*}
		\begin{gathered}
			\Box\phi=t^\alpha u^\beta \big(\partial_t\phi\big)^q
		\end{gathered}
	\end{equation*}
	in $\R^{1+1}$. The corresponding rescaled system is approximately
	\begin{equation*}
		\begin{gathered}
			\tilde{\Box}\psi=t^{\alpha-\frac{n-1}{2}(q-1)}u^{\beta-q}\abs{\psi}^q.
		\end{gathered}
	\end{equation*}
	Assuming fall off $\psi|_{\mathcal{N}_+}\sim t^{-s},\, \psi|_{\mathcal{N}_\scri}\sim t^{\sigma}$ the asymptotic decay condition requires a solution for
	\begin{equation*}
		\begin{gathered}
			\sigma=\max\Big(0,1-\frac{n-1}{2}(q-1)+\alpha+q\sigma\Big)\\
			s=\frac{n-1}{2}(q-1)-1-\alpha+\min(0,q(s+1)-1-\beta).
		\end{gathered}
	\end{equation*}
	These have solution only if $\sigma=0$ and $s=\frac{n-1}{2}(q-1)-1-\alpha$ is a solution implying
	\begin{equation*}
		\begin{gathered}
			\frac{n-1}{2}(q-1)-1-\alpha\geq0\\
			q\bigg(\frac{n-1}{2}(q-1)-1-\alpha\bigg)+q-1-\beta\geq0.
		\end{gathered}
	\end{equation*} 
	Indeed, in the $n=1$ case this agrees with \cite{kitamura_semilinear_2022}.
	
	\subsection{A new exponent appears}
	Finally, consider the (not yet studied) system \cref{partial v}:
	\begin{equation*}
		\begin{gathered}
			\tilde{\Box}\psi=t^{-\frac{n-1}{2}(q-1)}\abs{\partial_v\psi}^q+t^{-\frac{n-1}{2}(q-1)-q}\abs{\psi}^q
		\end{gathered}
	\end{equation*}
	where we also included a schematic term when $\partial_v$ hits the weight. Assuming $\psi|_{\mathcal{N}_\scri}\sim t^{\sigma},\,\psi|_{\mathcal{N}_+}\sim t^{-s}$, we get the compatibility conditions
	\begin{equation*}
		\begin{gathered}
			\sigma=\max(0,-\frac{n-1}{2}(q-1)+1-q+\sigma q)\\
			s=\frac{n-1}{2}(q-1)+q-1-\max(0,-q\sigma,1-sq).
		\end{gathered}
	\end{equation*}
	We need $\sigma=0$ for the first equation, while the second has solution iff $1-q(\frac{n-1}{2}(q-1)+q-1)<0\iff1<q(q-1)\frac{n+1}{2}$. 
	
	\begin{remark}\label{importance of dimension in even case}
		It's important to note, that the predicted fall-off at the critical exponent $s(n)=\frac{n-1}{2}(q_c(n)-1)+q_c(n)-1$ takes the value $s(2)\approx0.68,\, s(2k)<1\,\forall k\geq1$. Following \cref{even odd dimensions}, we see that for $n$ even and $n\geq4$, there's no modification to the above heuristic due to the decay coming from the Green's function, but the inherent slow decay in $n=2$ changes the above exponent. Indeed, we have
		\begin{equation*}
			\begin{gathered}
				s=\min(\frac{n-1}{2}(q-1)+q-1-\max(0,1-sq),1/2).
			\end{gathered}
		\end{equation*}
		This implies the conditions $\frac{n-1}{2}(q-1)+q-1-1+q/2>1/2\iff q>1.5$. Thus $q_c(2)=1.5$ instead $\approx1.46$.
	\end{remark}

	\subsection{A digression on initial data and critical exponents}\label{initial data critical exponent}
	So far, we only considered compactly supported initial data. As discussed in \cref{role of asymptotics}, this leads to arbitrary fast decay (via the Huygens principle) for the homogeneous wave equation in odd space dimensions. As also discussed in \cref{role of asymptotics}, a generic\footnote{Genericity means that the decay is not coming from purely outgoing data. In $\R^{3+1}$ this amounts to $\partial_v r\phi\sim r^{-\alpha}$} initial data with $r^{-\alpha}$ decay has $u^{-\alpha}$ fall off on $\scri$ with possible exceptional cancellations (\cref{exceptional cancellation}). This in particular leads to further constraint on the minimal value of the decay in  region $\mathcal{N}_+$. To see how this influences the critical curve, consider the Strauss problem in $\R^{3+1}$ with tails
	\begin{equation*}
		\begin{gathered}
			\Box\phi=\abs{\phi}^{q_1}\\
			\phi|_{\{t=0\}},(\partial_v r\phi)|_{\{t=0\}}\sim r^{-q_2}.
		\end{gathered}
	\end{equation*}
	Assuming $\psi|_{\mathcal{N}_\scri}\sim t^{\sigma},\,\psi|_{\mathcal{N}_+}\sim t^{-s}$, we get the compatibility conditions
	\begin{equation*}
		\begin{gathered}
			\sigma=\max(0,1-q_2,-q_1+2+\sigma q_1)\\
			s=\min(q_1-2-\max(0,-q\sigma,1-sq_1),q_2-1).
		\end{gathered}
	\end{equation*}
	For $q_2>1$, the first equation has only $\sigma=0$ solution implying $q_1>2$, while the second simplifies to
	\begin{equation*}
		\begin{gathered}
		s=\min(q_1-2,q_2-1,(s+1)q_1-3).
		\end{gathered}
	\end{equation*}
	In the case $q_2>q_1-1$, we end up with the Strauss condition $q_1>q_{Strauss}$. For $1<q_2<q_1-1$, we need 
	\begin{equation*}
		\begin{gathered}
			s=q_2-1<(s+1)q_1-3\implies q_2>\frac{2}{q_1-1}
		\end{gathered}
	\end{equation*}
	which is exactly as in \cite{karageorgis_existence_2005}\footnote{Note, that the results in \cite{strauss_existence_1997} are also consistent with these decay rate, but comparing results is tricky as they assume $\phi,\partial_t\phi\sim r^{-q_2}$ decay.}.
	
	We have so far avoided discussing what happens on the critical curve for all the problems, but we give a short comment here, by examining 3 examples. In all of them the asymptotic decay condition still works, that is, if the iterated approximations have the same decay, there are global solutions, but the logarithmic correction of \cref{log from scri} and \cref{log from corner} are going to be important.
	\begin{itemize}
		\item  Let's start with the previous problem, $\Box\phi=\abs{\phi}^{q_1}$ with data $\sim r^{-q_2}$, $q_2=2/(q_1-1),\, q_1>1+\sqrt{2}$. The first iterate $\Box\phi_0=0$ will be bounded by $t^{-1}u^{-(q_2-1)}$. Next, consider the asymptotics of $\Box\phi_1=\abs{\phi_0}^{q_1}$ with trivial data. The forcing satisfies $\abs{\phi_0}^{q_1}\lesssim t^{-q_1}u^{-(q_2-1)q_1}$ with $(q_2-1)q_1<1$ and $q_1q_2-2=q_2$. The inequality implies that the behaviour toward $I^+$ is going to determine the asymptotic for $\phi_1$ with no logarithmic correction, while the equality says that $\phi_1|_{\mathcal{N}_+}\sim u^{-q_2}$. Therefore one can close a contraction mapping in a function space with the above decay.
		
		\item Consider the Strauss problem ($\Box\phi=\abs{\phi}^q$), with exponent $q_c$ which is the larger root of ${\big((q-1)\frac{n-1}{2}-1\big)q=1}$. For compactly supported data $\Box\phi_0=0$ has asymptotic behaviour 
		\begin{equation*}
			\begin{gathered}
				\phi_0|_{\mathcal{N}_\scri}\sim t^{-\frac{n-1}{2}}\\
				\phi_0|_{\mathcal{N}_+}\sim t^{-\infty}.
			\end{gathered}
		\end{equation*}
		The iterate $\Box\phi_1=\abs{\phi_0}^{q_c}$ behaves like
		\begin{equation*}
			\begin{gathered}
				\phi_1|_{\mathcal{N}_\scri}\sim t^{-\frac{n-1}{2}}\\
				\phi_1|_{\mathcal{N}_+}\sim t^{-\frac{n-1}{2}-1/q_c}.
			\end{gathered}
		\end{equation*}
		Note, that $\abs{\phi_1}^{q_c}$ behaves like $t^{-\frac{n-1}{2}q_c}u^{-1}$. The non-integrable decay in $u$ gives logarithmic correction at $I^+$ compared to the previous iterate: $\phi_2|_{\mathcal{N}_+}\sim t^{-\frac{n-1}{2}-1/q_c}\log(t)$.  All the other iterates will receive extra logarithms of growing power near $I^+$, which is sufficient to show blow-up of the nonlinear solution. This is proved similarly as the result in \cref{critical section}.
		\item  Finally, consider the border line problem
		\begin{equation*}
			\begin{gathered}
				\Box_{\R^{3+1}}\phi=\eta^3\\
				\Box_{\R^{3+1}}\eta=(\partial_t\phi)^2+\phi^3
			\end{gathered}
		\end{equation*}
		with compactly supported data. As the weak null condition is satisfied there will be a solution for bounded retarded time, though with a logarithmically growing radiation field for $\eta$. Indeed the role of $(\partial_t\phi)^2$ is only important away from $I^+$, effectively it behaves as imposing $r^-1$ slowly decaying initial data for $\psi$. The first iterate will have compact in retarded time support, while the the second iterate will behave like $\phi_1|_{\mathcal{N}_\scri}\sim t^{-1},\,\phi_1|_{\mathcal{N}_+}\sim t^{-2},\,\eta_1|_{\mathcal{N}_+}\sim t^{-1},\,\eta_1|_{\mathcal{N}_+}\sim t^{-1}\log(t)$. For the next iterate, the behaviour of $\eta$ near $I^+$ is going to create the late time behaviour for $\phi$, with  $\phi_2|_{\mathcal{N}_+}\sim t^{-1}$. Importantly, there are no further logarithmic corrections, and all other iterates will have the same decay.
	\end{itemize}

	Motivated by the last problem, and the borderline behaviour of cubic and quadratic terms, we note that the system
	\begin{equation}\label{critical problem}
		\begin{gathered}
			\Box\phi_1=\abs{\phi_3}^3\\
			\Box\phi_2=(\partial_t\phi_1)^2\\
			\Box\phi_3=(\partial_t\phi_2)^2+\abs{\phi_3}^3
		\end{gathered}
	\end{equation}	
	has no non-trivial global solutions in $\R^{3+1}$. See \cref{failure of weak null condition}.
	
	\section{Technical lemmas for decay}
	In this section, we are going to prove/quote the main estimates capturing decay for the inhomogeneous wave equation. Many of the results presented here are exactly taken from other works, or with a slight change. In the latter case, we only indicated how the proof changes to reach the modified conclusion. We resorted to make this work not self contained as the estimates presented in the current section are standard, though lengthy to prove and it would distract attention from the main point of interest.

	\subsection{$L^2$ theory}\label{L2 technical estimates}
	Throughout this section, fix $\phi:\R^{n+1}\to\R$ ($n\in\{3\}\cup\{n\geq5\})$ to be a solution to
	\begin{equation}\label{linear}
		\begin{gathered}
			\Box\phi=F\\
			\phi(0)=\partial_t\phi(0)=0
		\end{gathered}
	\end{equation}
with $\supp F\subset\{r<t\}$ and $F$ spherically symmetric $C^1$ function with sufficient integrability as required in the lemmas below. Furthermore, let $\psi=r^{(n-1)/2}\phi$, $\Psi=r^2\partial_v\psi$, $\tilde{F}=r^{(n-1)/2}F$ and use notation $\overline{\partial}\phi\in\{\phi/r,\partial_u\phi,(\frac{v}{u})^{0.5}\partial_v\phi\}$. Define the foliations $\Sigma_u=\{r<1,t=u\}\cup\{t-r=u-1\},\tilde{\Sigma}_u:=\{r<t/2,t=u\}\cup\{r>t/2,t-r=u/2\}$ and interior region $\mathcal{D}_{u_1,u_2}$ which lies between $\Sigma_{u_1},\Sigma_{u_2}$. Let $\hat{\Sigma}$ stand for both $\Sigma,\tilde{\Sigma}$. Define the norms
\begin{equation}\label{definitino of E norms}
	\begin{gathered}
		E^p[f;u]=\int_{r>1,t-r=u}\d r f^2 r^p\\
		\mathcal{E}_{\hat{\Sigma}}[f]=\int_{\hat{\Sigma}}J[f]\cdot\d n\\
		E^{I}[f;u_1,u_2]=\int_{\mathcal{D}_{u_1,u_2}} \d t \d r \frac{r^{n-1}f^2}{r^{1-\epsilon_0}\jpns{r}^{2\epsilon_0}}\\
		E^{\star I}[f;u_1,u_2]=\int_{\mathcal{D}_{u_1,u_2}} \d t \d r\, r^{n-1}f^2r^{1-\epsilon_0}\jpns{r}^{2\epsilon_0}\\
		E^{\scri}[f;u_1,u_2]=\lim_{v\to\infty}\int_{u_1}^{u_2}\d u\,((\partial_t-\partial_r) f)^2(v-u,v+u)r^{n-1}
	\end{gathered}
\end{equation}
for some hypersurface $\hat{\Sigma}$ unit normal $\d n$ and energy current $J$. Note that the unit normal is not well defined for null hypersurface, so care is needed, in Minkowski space we can simply set 
\begin{equation*}
	\begin{gathered}
		\mathcal{E}_{\Sigma_u}[f]=\int_{r\leq1,t=u}((\partial_tf)^2+\abs{\nabla f}^2)r^{n-1}\d r\d\omega+\int_{r\geq1,t-r=u}((\partial_v f)^2+\abs{\slashed{\nabla}f}^2)r^{n-1}\d r \d\omega,
	\end{gathered}
\end{equation*}
and similarly for $\tilde{\Sigma}$, with the angular derivatives dropping out under spherical symmetry. Write $E^{I}[f;u]:=E^{I}[f;u,\infty]$ and $E^{\star I}[f;u]:=E^{\star I}[f;u,\infty]$. Now, we recall some key estimates.

\begin{lemma}[Boundedness towards $\scri$, \cite{yang_global_2013-1} Lemma 1]\label{boundedness}
	For $f\in C^{1}(\R^{n+1})$ with $\supp f\subset\{t-r>0\}$ and $E^\scri[f;0,u_1]<\infty$. Then
	\begin{equation*}
		\begin{gathered}
			r^{n-3}rf^2(t,r)\leq \mathcal{E}[f,t-r]
		\end{gathered}
	\end{equation*}
	for $r\geq1$.
\end{lemma}
\begin{proof}
	The case $n=3$ is exactly as in \cite{yang_global_2013-1}. By inspection of the proof, we see that the result is independent of dimension with $r^{n-3}$ modification. 
\end{proof}

\begin{lemma}[Hardy inequality, see eg. \cite{angelopoulos_vector_2018} Lemma 2.2]\label{hardy inequality}		
	For $q\neq1$ and $f\in C^1$ with $\lim_{r\to\infty}f(r)=0$ we have
	\begin{equation*}
		\begin{gathered}
			\int_1^\infty\d r\, r^qf^2\lesssim_q \int_1^\infty\d r\,r^{q+2}(\partial_rf)^2.
		\end{gathered}
	\end{equation*}
\end{lemma}

\begin{lemma}[Radial Sobolev embedding,\cite{hidano_glassey_2012} Lemma 2.2]\label{radial Sobolev inequality}
	For $f\in H^{1}(\R^{n}\to\R)$ ($n\geq3$) spherically symmetric $s\in(0,1]$ we have 
	\begin{equation*}
		\begin{gathered}
			\norm{r^{n/2-s}f}_{L^\infty}\lesssim_s\norm{\partial_rf}^s_{L^2}\norm{f}^{1-s}_{L^2}.
		\end{gathered}
	\end{equation*}
\end{lemma}

\begin{lemma}[Morawetz and energy estimates,\cite{hidano_glassey_2012} Lemma 3.2]\label{morawetz}
	
	We have the following linear estimate
	\begin{equation*}
		\begin{gathered}
			E^{\scri}[\phi;u_1,u_2]+E^{I}[\overline{\partial}\phi;u_1,u_2]+\mathcal{E}_{\Sigma_{u_2}}[\phi]\lesssim_\epsilon \mathcal{E}_{\Sigma_{u_1}}[\phi]+E^{\star I}[F;u_1,u_2]
		\end{gathered}
	\end{equation*}
\end{lemma}
\begin{proof}
	The statement was proved for the foliation $\{t=u_0\}$, but it uses the standard technique of a multiplier $f(r)\partial_r$. Indeed, take
	\begin{equation*}
		\begin{gathered}
			f(r)=\frac{r^\epsilon}{1+r^\epsilon},
		\end{gathered}
	\end{equation*}
	so that $\partial_rf=\frac{\epsilon}{r^{1-\epsilon}(1+r^\epsilon)^2}$. Then, using the multiplier in the region $\mathcal{D}_{u_1,u_2}$ gives an estimate
	\begin{equation*}
		\begin{gathered}
			\int_{\mathcal{D}_{u_1,u_2}}\abs{\partial\phi}^2\partial_r f\lesssim\int_{\mathcal{D}_{u_1,u_2}}\, f F \partial_r\phi+\mathcal{E}_{\Sigma_{u_1}}[\phi]+\mathcal{E}_{\Sigma_{u_2}}[\phi]
		\end{gathered}
	\end{equation*}
	The energy on the second surface ($\Sigma_{u_2}$) can be bounded using a standard energy estimate
	\begin{equation*}
		\begin{gathered}
			\mathcal{E}_{\Sigma_{u_2}}\leq\mathcal{E}_{\Sigma_{u_1}}+\int_{D_{u_1,u_2}}\abs{\partial\phi}F.
		\end{gathered}
	\end{equation*}
	After a Cauchy-Swartz, the result follows for $\partial\phi$. 
	To estimate $\phi$, one uses the modified currents (see \cite{yang_global_2013-1} equation (10)). To obtain the improved weight for $\partial_v\phi$ in the $r\geq1$ region, one instead uses $f=(1+u)^{-\epsilon}\partial_t$ as a multiplier (see \cite{luk_introduction_2015} Proposition 11.2).
\end{proof}
\begin{remark}
	These estimates are standard, but keeping the factors of $r$ instead $\jpns{r}$ will become important when solving the non-linear problem, because we only have small number of derivatives to use.
\end{remark}

\begin{lemma}[$r^p$ estimate,\cite{yang_global_2013-1} Prop. 4,\cite{angelopoulos_vector_2018} Prop. 4.1]\label{rp}
	
	a)
	For $p\in[0,2)$ we have the following linear estimate
	\begin{equation*}
		\begin{gathered}
		E^\scri[\phi;u_1,u_2],E^p[\partial_v\psi;u_2],\int_{u_1}^{u_2}\d u E^{p-1}[\partial_v\psi;u]\\\lesssim_p \int_{u_1}^{u_2}\d u E^{p+1}[\tilde{F};u]+E^p[\partial_v\psi;u_1]+E^{\star I}[F;u_1,u_2].
		\end{gathered}
	\end{equation*}
	If $n=3$, the above holds for $p\in[0,\infty)$.
	
	b) For $n=3$, $p\in[2,\infty)$ we have
	\begin{equation*}
		\begin{gathered}
				E^\scri[\partial_v\phi;u_1,u_2],E^p[\partial_v^2\psi;u_2],\int_{u_1}^{u_2}\d u E^{p-1}[\partial_v^2\psi;u]
			\\\lesssim_p \int_{u_1}^{u_2}\d u E^{p+1}[\partial_v \tilde{F};u]+E^p[\partial_v^2\psi;u_1]+E^{\star I}[\partial F;u_1,u_2].
		\end{gathered}
	\end{equation*}
	
	c)
	For $n\geq5$, and $p\in[-2,4)$ we have the linear estimate:
	\begin{equation*}
		\begin{gathered}
			E^{p-2}[\partial_v\Psi;u_2],\int_{u_1}^{u_2}\d u E^{p-3}[\partial_v\Psi;u],\int_{u_1}^{u_2}\d u E^{p-5}[\Psi;u]\\
			\lesssim_p\int_{u_1}^{u_2}\d u E^{p+1}[\tilde{F};u]+\int_{u_1}^{u_2}\d u E^{p+1}[r\partial_v\tilde{F};u]+E^{p-2}[\partial_v\Psi;u_1]+E^{\star I}[F;u_1,u_2]+E^{\star I}[\partial F;u_1,u_2].
		\end{gathered}
	\end{equation*}
\end{lemma}
\begin{proof}
	a) is in \cite{yang_global_2013-1} for the range $p\in(0,2)$ and $n=3$, with a different Morawetz term used, but this only depends on the integrated local energy estimate used. For higher dimensions, one can reduce to fixed $l$ mode case in $\R^{3+1}$  with the redefinition $\tilde{\phi}=r^{(n-3)/2}\phi$. For $n$ odd this works exactly, while in the even case, we get a non-integer $l$. However, the $r^p$ estimate only depend on the positivity of a certain term, which remains true in the even $n$ case as well. The extension for $n=3$ follows from the lack of constraints that are imposed by the angular derivatives 
	
	b) follows from a) after differentiating the equation of motion. 
	
	The homogeneous part of c) is in \cite{angelopoulos_vector_2018} for $n$ odd (by reducing it to fixed $l$ mode in $\R^{3+1}$). The even cases is similarly dealt by reducing to $n=3$ and using positivity of one part of the equation that holds for $n>4$. This is exactly the reason why we must exclude $n=4$. The inhomogeneous case, can be added the same way as in a), ie using a Cauchy-Schwartz inequality on the error term $F\partial_v\Psi$ appearing on the right hand side.
\end{proof}
\begin{remark}
	The requirement
	\begin{equation*}
		\begin{gathered}
			\int_{u_1}^{u_2}\d u E^{p+1}[\tilde{F};u]
		\end{gathered}
	\end{equation*}
	may be replaced by
	\begin{equation*}
		\begin{gathered}
			\bigg(\int_{u_1}^{u_2}\d u E^{p}[\tilde{F};u]^{1/2}\bigg)^2,
		\end{gathered}
	\end{equation*}
	if one uses Cauchy-Schwartz on a null piece instead of one in space-time. Indeed, this is the one we will use further on, as it requires less decay for $F$ towards $\scri$.
\end{remark}
\begin{remark}
	Note, that this estimate is almost sharp from the view of the heuristics. If $\tilde{F}=cv^{-s}u^{-1-\epsilon}$ near $\scri\cap I^+$ than, $s_c=p/2+1/2$ is the critical exponent for given $p$, ie. the estimates above are satisfied for $s>s_c$. This in turn shows that the estimates are "almost" sharp, in the sense, that the critical decay for $\partial_v\psi$ is $s_c$.
\end{remark}

\begin{lemma}[Energy decay, restricted,\cite{yang_global_2013-1} Proposition 5]\label{energy decay p<2}
	
	For $p\in[0,2)$ and $F$ such that
	\begin{equation*}
		\begin{gathered}
			\int\d u E^{p}[\tilde{F};u]^{1/2},\sup_uu^pE^{\star I}[F;u,2u]\leq1
		\end{gathered}
	\end{equation*}
	we have
	\begin{equation*}
		\begin{gathered}
			\sup_u E^p[\partial_v\psi;u]+ u^p(E^0[\partial_v\psi;u]+\mathcal{E}_{\hat{\Sigma}_u}[\phi]+E^{I}[\overline{\partial}\phi;u,\infty])\lesssim1.
		\end{gathered}
	\end{equation*}
	
	Moreover, for $n=3$, the above holds for $p\in[0,\infty]$.
\end{lemma}
\begin{proof}
	The $E^p[\partial_v\psi;u]$ statement follow from the $r^p$ estimate above.
	
	The energy decay on the $\Sigma_u$ foliation follows from the hierarchy as in \cite{yang_global_2013-1}, with the following additional steps. At each dyadic interval, one has to use the inhomogeneous $r^p$ and Morawetz estimates \cref{morawetz}, to gain decay from the $F$ terms. In particular, one needs 
	\begin{equation*}
		\begin{gathered}
			\sup_u u^{p-q}\Bigg(\int_u^{2u}\d u' E^{q}[\tilde{F};u']^{1/2}\Bigg)^2\lesssim1	
		\end{gathered}
	\end{equation*}
	or
	\begin{equation*}
		\begin{gathered}
			\sup_u u^{p-q}\int_u^{2u}\d u E^{q+1}[\tilde{F};u]\lesssim1.
		\end{gathered}
	\end{equation*}
	Note that for $0\leq q\leq p$ we have
	\begin{equation*}
		\begin{gathered}
			\int_{r>u,t-r=u}\d r \,\tilde{F}^2r^q\leq u^{q-p}\int_{r>u,t-r=u}\d r \, \tilde{F}^2r^p\lesssim u^{q-p}E^p[\tilde{F},u]\\
			\int_{u_1}^{u_2}\int_{1<r<u,t-r=u}\d r\,\tilde{F}^2r^{q+1}\leq u_1^q E^{\star I}[f;u_1,u_2]\leq u_1^{q-p},
		\end{gathered}
	\end{equation*}
	thus the required decay of $F$ follows. Therefore, the right hand side of the inhomogeneous $r^p$ estimates have the correct decay using the assumptions of the lemma.
	
	Using this energy decay, $E^I$ bound follows from the Morawetz estimate.
	
	The decay along $\hat{\Sigma}_u$ foliation follows from that on $\Sigma_u$ and the decay assumptions on $F$, similar to \cref{morawetz}.
\end{proof}

\begin{lemma}[Energy decay, extended, \cite{angelopoulos_vector_2018} Prop 7.5]\label{energy decay p<4}
	Fix $n\geq5$ and $p\in[-2,4)$. Then, we have the following linear estimate:
	\begin{equation*}
		\begin{gathered}
			\int\d u E^{p}[\tilde{F};u]^{1/2},\sup_uu^pE^{\star I}[F;u],\int\d u E^{p}[r\partial_v\tilde{F};u]^{1/2},\sup_uu^pE^{\star I}[u\partial F;u,2u]\leq1
		\end{gathered}
	\end{equation*}
	implies
	\begin{equation}\label{extended energy decay estimates}
		\begin{gathered}
			\sup_u E^{p-2}[\partial_v\Psi;u]+E^p[\partial_v\psi;u]+ u^p(E^0[\partial_v\psi;u]+E^{-2}[\partial_v\Psi;u]+\mathcal{E}_{\hat{\Sigma}_u}[\phi]+E^{I}[\overline{\partial}\phi;u,\infty])\\
			\sup_uu^{p+2}(\mathcal{E}_{\hat{\Sigma}_u}[\partial\phi]+E^{I}[\overline{\partial}\partial\phi;u,\infty])\lesssim1.
		\end{gathered}
	\end{equation}
\end{lemma}
\begin{proof}
	Using the hierarchy from \cite{angelopoulos_vector_2018}, one can build up the decay statements the same way as in the previous case.
\end{proof}

\begin{cor}\label{Psi to psi}
	For $\phi$ as above, we also have
	\begin{equation*}
		\begin{gathered}
			\sup_uE^{p+2}[\partial_v^2\psi;u],u^{p+2}E^0[\partial_v^2\psi;u]\lesssim1
		\end{gathered}
	\end{equation*}
\end{cor}

\begin{proof}
	This is a consequence of triangle inequality and the estimate \cref{extended energy decay estimates}.
\end{proof}

\begin{lemma}[$L^\infty$ estimates]\label{l infty estimates from l2}
	For $\phi$ with $E^\scri[\phi;u,0]<\infty$, $\supp\phi\subset\{t-r>0\}$ and $\epsilon>0$, spherically symmetric we have 
	\begin{equation*}
		\begin{gathered}
			E^p[\partial_v\psi;u],u^pE^0[\partial_v\psi;u],u^p\mathcal{E}_{\Sigma_u}[\phi]\leq1\\
			\implies u^{(p-1-\epsilon)/2}\norm{\psi(v,u)}_{L^\infty(\Sigma_u\cap\{r>1\})}\lesssim 1
		\end{gathered}
	\end{equation*}
	
	\begin{equation*}
		\begin{gathered}
			E^p[\partial_v\psi;u],u^pE^0[\partial_v\psi;u],\sup_uE^{p+2}[\partial_v^2\psi;u],u^{p+2}E^0[\partial_v^2\psi;u],u^p\mathcal{E}_{\hat{\Sigma}_u}[\phi],u^{p+2}\mathcal{E}_{\hat{\Sigma}_u}[\partial\phi]\leq1\\
			\implies u^{(p+2)/2}\norm{r^{n/2-1}\partial\phi(v,u)}_{L^\infty(\hat{\Sigma}_u)},\norm{v^{(p+1-\epsilon)/2}\partial_v\psi(v,u)}_{L^\infty(\Sigma_u\cap \{v>2u\})}\lesssim 1
		\end{gathered}
	\end{equation*}	
\end{lemma}
\begin{proof}
	We proceed in order.
	
	From \cref{boundedness} and $\mathcal{E}$ bound, we get $u^{p/2}\psi|_{\{r=1\}}\leq1$. From here, we can integrate towards $\scri$ to get
	\begin{equation*}
		\begin{gathered}
			\norm{\psi(v,u)}_{L^\infty(\Sigma_u\cap\{r\geq1\})}\leq u^{-p/2}+\int_{\Sigma_u\cap\{r>1\}}\abs{\partial_v\psi}\lesssim_\epsilon u^{-p/2}+E^{1+\epsilon}[\partial_v\psi;u]^{1/2}\leq u^{-(p-1-\epsilon)/2}.
		\end{gathered}
	\end{equation*}

	The second one is exactly the estimate \cref{radial Sobolev inequality}, with $f=\partial\phi$.
	
	Finally, using the previous two results, we see that
	\begin{equation*}
		\begin{gathered}
			\norm{\partial_v\psi|_{r=u}}\lesssim u^{-(n-1)/2}\partial\phi+u^{-(n+1)/2}\phi\lesssim u^{-(p+1)/2}.
		\end{gathered}
	\end{equation*}
	Integrating from this point towards $\scri$, we get
	\begin{equation*}
		\begin{gathered}
			\norm{r^{(p+2-\epsilon)/2}\partial_v\psi}_{L^\infty(\Sigma_u\cap\{r>u\})}\lesssim 1+\int_{\Sigma_u\cap\{r>u\}}\abs{\partial_v(r^{(p+2-\epsilon)/2}\partial_v\psi)}\lesssim_\epsilon 1+E^p[\partial_v\psi;u]+E^{p+2}[\partial_v^2\psi;u].
		\end{gathered}
	\end{equation*}
\end{proof}	
	\subsection{$L^\infty$ theory in $\R^{1+1}$}\label{l infinite theory}
	\begin{lemma}
		For $\psi$ a solution to the inhomogeneous wave equation with trivial data (\eqref{linear}) in $\R^{1+1}$ with $\supp F\subset\{t-r>1\},\, \norm{F}_{L^\infty}<1$, $F$ spherically symmetric ($F(t,r)=F(t,-r)$) we have 
		\begin{equation*}
			\begin{gathered}
				\partial_v\psi(u,v)=\int_1^v\d u^\prime\, F(u^\prime,v)\\
				\partial_u\psi(u,v)=-\int_1^u\d u^\prime\, F(u^\prime,u)+\int_u^v\d v^\prime F(u,v^\prime)\\
				\psi(u,v)=\int_1^u\d u^\prime\int_u^v\d v^\prime  F(u^\prime,v^\prime)
			\end{gathered}
		\end{equation*}
	\end{lemma}
	
	\begin{lemma}
		For $\psi$ a solution to \eqref{linear} in $\R^{1+1}$ with $\supp F\subset\{t-r>1\}$ spherically symmetric and $z\neq1,\sigma\neq 1$, we have
		\begin{equation*}
			\begin{aligned}
				\norm{Fv^{\sigma}u^{z}}_{L^\infty}\leq1\implies \norm{\frac{\psi}{r}vu^{-\max(0,1-z)}(u^{1-\sigma}+v^{1-\sigma})^{-1}}_{L^\infty}\lesssim 1.
			\end{aligned}
		\end{equation*}
		Similarly for the derivatives
		\begin{equation*}
			\begin{gathered}
				\norm{Fv^{\sigma}u^{z}}_{L^\infty}\leq1\implies \norm{\partial_v\psi v^{\sigma}u^{-\max(0,1-z)}}_{L^\infty},\norm{\partial_u\psi u^{-1-\max(0,1-z)}(u^{1-\sigma}+v^{1-\sigma})^{-1}}_{L^\infty}\lesssim 1 \\
				\norm{Fv^{\sigma}u^{z}}_{L^\infty}+\norm{v^{\sigma}u^{z}(v\partial_v)F}_{L^\infty}\leq1\implies \norm{\frac{\partial_t\psi}{r}v u^{-1-\max(0,1-z)}(u^{1-\sigma}+v^{1-\sigma})^{-1}}_{L^\infty}\lesssim1
			\end{gathered}
		\end{equation*}
	\end{lemma}
	\begin{proof}
		Using positivity of the fundamental solutions, we have
		\begin{equation*}
			\begin{gathered}
				\abs{\psi}\leq \int_1^u\d u^\prime\int_u^v\d v^\prime  \abs{F(u^\prime,v^\prime)}\leq \int_1^u\d u^\prime\int_u^v\d v^\prime {v^\prime}^{-\sigma}{u^\prime}^{-z}\\
				=\frac{u^{-z+1}-1}{1-z}\frac{v^{-\sigma+1}-u^{-\sigma+1}}{1-\sigma}.
			\end{gathered}
		\end{equation*}
		Now we end the proof by pulling out the appropriate number of $u,v$ factors and using that $\frac{\rho^\alpha-1}{\rho-1}\lesssim_\alpha1$ for $\rho\in[1,2]$:
		\begin{equation*}
			\begin{gathered}
				\abs{\psi}\leq u^{-\sigma+1}\frac{v-u}{v}\frac{u^{1-z}-1}{1-z}\frac{(\frac{v}{u})^{2-\sigma}-1}{(2-\sigma)(1-\frac{u}{v})}
			\end{gathered}
		\end{equation*}
		with the last term is bounded for $\sigma>1$. For $\sigma<1$, $\frac{(\frac{v}{u})^{1-\sigma}-1}{(1-\sigma)(1-\frac{u}{v})}v^{-\sigma}\lesssim_\sigma1$.
		
		The last derivative inequality is bounded similarly:
		\begin{equation*}
			\begin{gathered}
				\partial_t\psi=\int_1^u\d u^\prime (F(u^\prime,v)-F(u^\prime,u))+\int_u^v\d v^\prime F(u,v^\prime)=\int_1^u\d u^\prime\int_u^v\d v^\prime \partial_vF(u^\prime,v^\prime)+\int_u^v\d v^\prime F(u,v^\prime)\\
				\leq \int_1^u\d u^\prime\int_u^v\d v^\prime v^{-1-\sigma}u^{-z}+\int_u^v\d v^\prime {v^\prime}^{-\sigma}u^{-z}=\frac{v^{-\sigma}-u^{-\sigma}}{-\sigma}\frac{u^{1-z}-1}{1-z}+\frac{v^{1-\sigma}-u^{1-\sigma}}{1-\sigma}u^{-z}
			\end{gathered}
		\end{equation*}
		
		The first two derivative inequalities are easier to bound, as we do not need the above cancellation:
		\begin{equation*}
			\begin{gathered}
				\partial_v\psi\leq\int_1^u\d u^\prime\, v^{-\sigma}{u^\prime}^{-z}=v^{-\sigma}\frac{u^{1-z}-1}{1-z}\\
				\partial_u\psi\leq \int_1^u \d u^\prime u^{-\sigma}{u^\prime}^{-z}+\int_u^v\d v^\prime u^{-z}{v^\prime}^{-\sigma}=\frac{u^{-z+1}-1}{1-z}u^{-\sigma}+\frac{v^{1-\sigma}-u^{1-\sigma}}{1-\sigma}u^{-z}.
			\end{gathered}
		\end{equation*}
	\end{proof}
	\begin{remark}
		The $z=1$ and $\sigma=1$ cases yield logarithmic losses to the above estimates. This will never be crucial for us, as we don't consider borderline cases, thus such losses can be absorbed into other parts of the decay statement.
	\end{remark}
	
	Finally, some lower bounds
	\begin{lemma}[Lower bound]\label{lower bound}
		Let be $\phi$ the solution of
		\begin{equation*}
			\begin{gathered}
				\Box_{\R^{3+1}}\phi=F
			\end{gathered}
		\end{equation*}
		 with $F\geq0$, $F\in C_{rad}^0$ . If $F(v+u,v-u)\geq c v^{-q}$ for $u\in[u_0,u_0+\epsilon]$ with $c,\epsilon>0$, then, there exists $u_1$ sufficiently large such that
		\begin{equation*}
			\begin{gathered}
				\psi(t,r)\gtrsim_{\epsilon,q,c} t^{-1}u^{1-q}\qquad \forall t-r>u_1.
			\end{gathered}
		\end{equation*}
		Similarly, if $F(v+u,v-u)\geq c v^{-q}u^{-s}$ for $u>u_0$, than
		\begin{equation*}
			\begin{gathered}
				\psi(t,r)\gtrsim_{c,q,s} t^{-1}u^{1-q+\max(1-s,0)}\qquad \forall t-r>u_1.
			\end{gathered}
		\end{equation*}
	\end{lemma}
	\begin{proof}
		As the kernel is positive, we can bound the spherically symmetric part of the solution for $u>u_0+\epsilon$ as
		\begin{equation*}
			\begin{gathered}
				\phi\geq\frac{1}{v-u}\int_{u_0}^{u_0+\epsilon}\d u'\int_u^v\d v' c{v'}^{-q}=\frac{\epsilon c}{q-1}\frac{u^{-q+1}-v^{-q+1}}{v-u}=\frac{\epsilon c}{q-1}\frac{u^{1-q}}{v}\frac{1-(\frac{u}{v})^{q-1}}{1-\frac{u}{v}}
			\end{gathered}
		\end{equation*}
		for $q\neq1$. In the case $q>1$, restricting to the region $u>0$, we get the desired bound using $v\sim t$ and $\frac{1-x^{q-1}}{1-x}|_{x\in[0,1]}\gtrsim_q1$. As for $q<1$, may factor out $v^{1-q}$ and use $u\leq v$. For $q=1$, there is a logarithmic term, which gives even faster growth. 
		
		For the second estimate, we calculate similarly
		\begin{equation*}
			\begin{gathered}
				\phi\geq\frac{1}{v-u}\int_{u_0}^{u}\d u'\int_u^v\d v' c{v'}^{-q}{u'}^{-s}=\frac{\epsilon c}{q-1}\frac{u^{-q+1}-v^{-q+1}}{v-u}\frac{u^{-s+1}-u_0^{-s+1}}{1-s}
			\end{gathered}
		\end{equation*}
		for $q\neq1,s\neq1$. The other cases induce further logarithmic terms giving even more growth.
	\end{proof}

	\section{Testing the hunch: global existence}
	\subsection{$\partial_v\phi$ nonlinearity}
	\subsubsection{Existence of solution}
	Throughout this section, fix $\phi$ a spherically symmetric solution to \eqref{partial v} with $n\in\{3\}\cup\{n\geq5\}$, with $\supp\subset\{t-r>0\}$. As we have a heuristic for the critical power, let's fix $q_c(q_c-1)\frac{n+1}{2}=1$ with $q>q_c>0$, and $0<\epsilon,\epsilon_0\ll q-q_c$. We separate $\epsilon$  --loss due to almost sharp decay--  and $\epsilon_0$ --loss in Morwatz estimate-- for sake of clarity. Solving the quadratic gives $q_c=1+2/n-6/n^2+\mathcal{O}(1/n^3)$. Also, fix $p(q)=(n+1)q-n=3+\mathcal{O}(1/n)$, which will be the optimal power in $r^p$ estimates. Importantly $3<p(n)<4$ for all $n$.

	We proceed in two steps. First, we mimic (\cite{hidano_glassey_2012}) for a local existence theorem, which in turn gives a \textit{controlling norm} (as discussed in \cite{tao_nonlinear_2006}). Then, we perform a bootstrap to conclude that the controlling norm stays bounded, thus the solution is global.
	
	\begin{theorem}[Local existence]
	Let $\phi_0,\phi_1:\R^n\to\R\, ,\, \norm{\partial\phi_0,\phi_1}_{H^1\times H^1}<C$ be spherically symmetric initial data for \eqref{partial v} with $q\in(1+1/n,1+2/(n-2))$. Then $\exists\, \tau(C)$ such that \eqref{partial v} has a unique solution  $\phi\in C^0[[0,\tau];H_{rad}^2]\cap C^1[[0,\tau],H_{rad}^1]$.
	\end{theorem}

	Note, that $q_c\in (1+1/n,1+2/(n-2))$.
	\begin{proof}
	In the local existence, it doesn't matter what type of derivative falls on $\phi$, so we'll use $\partial$ instead $\partial_v$.
	
	Let's define a solution iteratively 
	\begin{equation*}
		\begin{gathered}
			\Box\phi^{(n)}=\abs{\partial\phi^{(n-1)}}^q\\
			\phi^{(n)}(0)=\phi_0,\,\partial_t\phi^{(n)}(0)=\phi_1,
		\end{gathered}
	\end{equation*}
	with $\phi^{(-1)}=0$. Recalling the Morawetz estimate \cref{morawetz}, we know 
	\begin{equation}\label{local existence iteration}
		\begin{gathered}
			\norm{\phi^{(n)}(t)}_{\dot{H}^1}+\int_0^t\frac{r^{n-1}(\partial\phi^{(n)})^2}{r^{1-\epsilon_0}\jpns{r}^{2\epsilon_0}}\lesssim_{\epsilon_0}\norm{\phi^{(n)}(0)}_{\dot{H}^1}+\Bigg(\int_0^t r^{n-1}r^{1-\epsilon_0}\jpns{r}^{2\epsilon_0}\abs{\partial\phi^{(n-1)}}^{2q}\Bigg)^{1/2}\\
			\norm{\phi^{(n)}(t)}_{\dot{H}^2}+\int_0^t\frac{r^{n-1}(\partial^2\phi^{(n)})^2}{r^{1-\epsilon_0}\jpns{r}^{2\epsilon_0}}\lesssim_{\epsilon_0}\norm{\phi^{(n)}(0)}_{\dot{H}^2}+\Bigg(\int_0^t r^{n-1}r^{1-\epsilon_0}\jpns{r}^{2\epsilon_0}\abs{\partial\phi^{(n-1)}}^{2(q-1)}(\partial^2\phi^{(n-1)})^2\Bigg)^{1/2}.
		\end{gathered}
	\end{equation}
	Let's define $\norm{\cdot}_{Z}$ to be the sum of the norms on the left hand side of both equations and $\norm{\cdot}_W$ to be the sum in the first equation.
	
	\textit{Boundedness:} There exist $T$ sufficiently small $c>0$ such that for all $0\leq t<T$, we have $\norm{\phi^{(n)}}_{Z}$ is bounded by $c\delta=c(\norm{\phi_0,\phi_1}_{\dot{H^2}\cap\dot{H}^1\times\dot{H}^1\cap\dot{H}^0})$: The square of non-linear part of the right hand side of \eqref{local existence iteration} may be bounded using \cref{radial Sobolev inequality} and interpolation:
	\begin{equation*}
		\begin{gathered}
			\int_0^t r^{n-1}r^{1-\epsilon_0}\jpns{r}^{2\epsilon_0}(\partial^2\phi^{(m)})^2\abs{\partial\phi^{(m)}}^{2(q-1)}\\
			\lesssim_{\alpha}\norm{\partial\phi^{(m)} r^{\frac{(1+\alpha)(1-\epsilon_0)}{2(q-1)}}\jpns{r}^{\frac{(1+\alpha)\epsilon_0}{q-1}}}_{L^\infty}^{2(q-1)}\int r^{n-1}(\partial^2\phi^{(m)})^2 \Bigg(\frac{1}{r^{1-\epsilon_0}\jpns{r}^{2\epsilon_0}}\Bigg)^\alpha\\ \lesssim_{s_i,\alpha}\max_i\bigg(\norm{\phi^{(m)}}_{\dot{H}^1}^{s_i}\norm{\phi^{(m)}}_{\dot{H}^2}^{1-s_i}\bigg)^{2(q-1)}t^{2(1-\alpha)}\norm{\phi^{(m)}}_{\dot{H}^2}^{2(1-\alpha)}\Bigg(\int\frac{r^{n-1}(\partial^2\phi^{(m)})^2}{r^{1-\epsilon_0}\jpns{r}^{2\epsilon_0}}\Bigg)^{\alpha}\lesssim t^{1-\alpha}\norm{\phi^{(m)}}_Z^{2q}
		\end{gathered}
	\end{equation*}
	with $\alpha,s_i\in(0,1)$ for $i\in\{1,2\}$ and $s_i=\frac{n}{2}-\frac{(1+\alpha)(1+\epsilon_0(-1+2i))}{2(q-1)}$. For $s_i$ to satisfy the constraints,  we need to choose $\alpha$ such that $\frac{1+\alpha}{2(q-1)}\in(n/2-1,n/2)$. A bit of algebra shows that this choice is possible for $q-1\in(1/n,2/(n-2))$. Afterwards, $\epsilon_0$ can be chosen sufficiently small such that $s_i$ satisfy the bounds.	
	
	Choosing $t$ sufficiently small -- based on initial data ($\delta$) -- we see that the solution is indeed bounded in $Z$.
	
	\textit{Contraction:}
	Let $\chi^{(m)}=\phi^{(m)}-\phi^{(m-1)}$. Using \eqref{local existence iteration} again, we get
	\begin{equation*}
		\begin{gathered}
			\norm{\chi^{(m)}}_W^2\lesssim_{\epsilon_0}\int_0^t(\abs{\partial\phi^{(m-1)}}^q-\abs{\partial\phi^{(m-2)}}^q)^2r^{1-\epsilon_0}\jpns{r}^{2\epsilon_0}r^{n-1}\\
			\lesssim_q \max_{j\in\{m-1,m-2\}}\int_0^t(\partial\chi^{(m-1)})^2\abs{\partial\phi^{(j)}}^{2(q-1)}r^{n-\epsilon_0}\jpns{r}^{2\epsilon_0}\\ \lesssim\norm{\chi^{(n-1)}}_W^2t^{1-\alpha}\norm{\partial\phi^{(j)}r^{\frac{(1+\alpha)(1-\epsilon_0)}{2(q-1)}}\jpns{r}^{\frac{(1+\alpha)2\epsilon_0}{2(q-1)}}}_{L^\infty}^{2(q-1)}
			\lesssim\norm{\chi^{(n-1)}}_W^2t^{1-\alpha}\norm{\phi^{(n)}}_Z^{2(q-1)}
		\end{gathered}
	\end{equation*}
	where correct $\alpha$ can be chosen under the same conditions on $q$ as before. We see, that for $t$ sufficiently small,  we get a contraction: $\norm{\chi^{(m)}}_W\leq1/2\norm{\chi^{(m-1)}}_W$. Therefore, $\phi^{m}$ have a limit and by boundedness in $Z$, the limit is also in $Z$.
	\end{proof}
	
	To understand the global behaviour of the solution, it suffices to control the following norms:
	\begin{taggedsubequations}{linear}\label{linear bootstrap}
		\begin{equation}\label{key}
			\begin{gathered}
				\sup_uE^{p-\epsilon}[\partial_v\psi;u],u^{p-\epsilon}E^0[\partial_v\psi;u]
			\end{gathered}
		\end{equation}
		\begin{equation}\label{key}
			\begin{gathered}
				\sup_uE^{p+2-\epsilon}[\partial_v^2\psi;u],u^{p+2-\epsilon}E^0[\partial_v^2\psi;u]
			\end{gathered}
		\end{equation}
		\begin{equation}\label{key}
			\begin{gathered}
				\sup_uu^{p-\epsilon}\mathcal{E}_{\hat{\Sigma}_u}[\phi],u^{p+2-\epsilon}\mathcal{E}_{\hat{\Sigma}_u}[\partial\phi]
			\end{gathered}
		\end{equation}
		\begin{equation}\label{key}
			\begin{gathered}
				\sup_uu^{p-\epsilon}E^{I}[\bar{\partial}\phi,u],u^{p+2-\epsilon}E^{I}[\bar{\partial}\partial\phi,u]
			\end{gathered}
		\end{equation}
		\begin{equation}
			\sup_u\mathcal{E}_{\{t=u\}}[\phi],\mathcal{E}_{\{t=u\}}[\partial\phi]
		\end{equation}
	\end{taggedsubequations}
where $p=p(q)$ defined in the beginning of this section and the definition of each expression is contained in \cref{definitino of E norms}.

	\begin{definition}\label{x norm definition}
		For $\mathcal{A}\subset \R^{n+1}$, let's define $\norm{\phi}_{X(\mathcal{A})}$ to be the sum of all the norms above (with $\psi=r^{(n-1)/2}\phi$ implicit) restricted to $\mathcal{A}$ and $\norm{\cdot}_{X}=\norm{\cdot}_{X(\{t\geq0\})}$. 
	\end{definition}
	These in turn will imply the boundedness of the following norms for the nonlinearity:
	
	\begin{taggedsubequations}{non-linear}\label{nonlinear bootstrap}
		\begin{equation}\label{key}
			\begin{gathered}
				\int \d u\,E^{p-\epsilon}[\tilde{F};u]^{1/2},\int \d u\,E^{p-\epsilon}[r\partial_v \tilde{F};u]^{1/2}
			\end{gathered}
		\end{equation}
		\begin{equation}\label{key}
			\begin{gathered}
				\sup_u u^{p}(E^{\star I}[F;u],E^{\star I}[u\partial F;u])
			\end{gathered}
		\end{equation}
		
\end{taggedsubequations}
	\begin{definition}\label{y norm definition}
		For $\mathcal{A}\subset\R^{n+1}$ let $\norm{F}_{Y(\mathcal{A})}$ be the sum of the above norms with integrals taken at a restriction to $\mathcal{A}$.
	\end{definition}
	\begin{remark}
		Note, that for both norms, we have
		\begin{equation*}
			\begin{gathered}
				\norm{f+g}_i\leq\norm{f}_i+\norm{g}_i\qquad i\in\{X,Y\}
			\end{gathered}
		\end{equation*}
	since all quantities are based on $L^2$ norms.
	\end{remark} 
	
	The local existence can be extend to a global result using the bootstrap  \cref{bootstrap}:
	
	\begin{theorem}[Global existence]\label{partival v main theorem}
		For $q_c<q<\frac{2}{n-2}+1$, $n\in\{3\}\cup\{n\geq5\}$ there exists $\delta,c>0$ such that the following holds. All spherically symmetric initial data $\norm{\phi_0,\phi_1}_{H^2\times H^1}\leq\delta$ with $\supp\phi_0,\phi_1\subset B_1$ for \eqref{partial v}, has global solution $\phi\in C^0[[0,\tau];H_{rad}^2]\cap C^1[[0,\tau],H_{rad}^1] $ with $\norm{\phi}_X\leq c\delta$.
	\end{theorem}
	\begin{proof}
		Let $\delta_0$ be the constant appearing in \cref{bootstrap}. Let $\delta\in(0,\delta_0)$ such that the linear wave with initial data $\phi_0,\phi_1$ has $\norm{\phi^{(lin)}}_X<\delta_0$. Such delta exist by the homogeneous part of \cref{linear part of bootstrap}.
		
		Let's set $I=\{ t\geq0 :\norm{\phi(t)}_{X[0,t]}\leq10\delta \}$. By continuity and local existence, we know that $I$ has non empty interior and is closed.  By the bootstrap and continuity, we know it's open, thus $I=[0,\infty)$.
	\end{proof}
	
	\begin{prop}\label{bootstrap}
		Fix $q_c<q<\frac{2}{n-2}+1$. There exists $\delta>0$ such that the following holds.
		
		For $\phi$ a spherically symmetric solution to \eqref{partial v} in $\mathcal{A}=[0,T]\times\R^n$ ($n\in\{3\}\cup\{n\geq5\}$) and $\phi_{lin}$ the solution to the wave equation with the same initial data we have the following bootstrap estimate
		\begin{equation*}
			\begin{gathered}
				\norm{\phi_{lin}}_{X}\leq \delta,\, \norm{\phi}_{{X}(\mathcal{A})}\leq 10\delta\implies\norm{\phi}_{{X}(\mathcal{A})}\leq 2\delta.
			\end{gathered}
		\end{equation*}
		with $\norm{\cdot}_{X}$ are as in \cref{x norm definition}.
	\end{prop}
	
	This proposition is immediate given the following two parts:
	\begin{lemma}[Linear part]\label{linear part of bootstrap}
		For a solution $\phi$ to \cref{linear} with $\norm{F}_{Y[0,T]}\leq1,\,\supp F\subset\{t-r>0\}$ , we have $\norm{\phi}_{X([0,T])}\lesssim1$. Importantly, the implicit constant doesn't depend on $T$. Moreover we can also have initial data $\norm{\phi_0,\phi_1}_{H^2\times H^1}\leq1$ with $\supp\phi_0,\phi_1\{r<1\}$.
	\end{lemma}
	
	\begin{lemma}[Non-linear part]\label{non-linear part of bootstrap}
		For $\phi\in\R^{n+1}$ spherically symmetric with $\supp\phi\subset\{t-r>0\}$ and $\norm{\phi}_{X[0,T]}\leq\delta$, we have $\norm{\abs{\partial_v\phi}^q}_{Y[0,T]}\lesssim \delta^q$.
	\end{lemma}
	
	\begin{proof}[Proof of \cref{bootstrap}]
		Let $\hat{\phi}=\phi-\phi_{lin}$ denote the non-linear part of the solution. Then, it satisfies an inhomogeneous wave equation
		\begin{equation*}
			\begin{gathered}
				\Box\hat{\phi}=\abs{\partial_v(\phi_{lin}+\hat{\phi})}^q
			\end{gathered}
		\end{equation*}
		with trivial data for $t\in[0,T]$. Let $F=\abs{\partial_v(\phi_{lin}+\hat{\phi})}^q$. Then, \cref{non-linear part of bootstrap} and the assumptions of the propositions imply that $\norm{F}_{Y[0,T]}\lesssim\delta^q$: 
		\begin{equation*}
			\begin{gathered}
				\norm{\phi_{lin}+\hat{\phi}}_X\leq\norm{\phi_{lin}}_X+\norm{\hat{\phi}}_X\leq 11\delta\quad\xRightarrow[\cref{non-linear part of bootstrap}]\quad \norm{F}_{Y[0,T]}\lesssim \delta^q.
			\end{gathered}
		\end{equation*}
		
		Pick cut-off function $\chi\in\mathcal{C}^\infty(\R^+,[0,1])$ with $\chi|_{\{r<1/2\}}=1, \chi|_{\{r>1\}}=0$ and set
		\begin{equation}\label{F extension}
			\begin{gathered}
				\tilde{F}(t,r)=\begin{cases}
					F(t,r) & t\leq T\\
					F(2T-t,r+T-t)\chi(T-t) & \epsilon\in[0,1],r>1,t>T\\
					F(2T-t,r-(1-\chi(r))(T-t))\chi(T-t) & \epsilon\in[0,1],r\leq1\\
					0 & \text{else}.
				\end{cases}
			\end{gathered}
		\end{equation}
		As the $Y$ norm only includes first derivatives of $F$, it's clear that $\norm{F}_{Y([0,T])}\lesssim\delta^q\implies\norm{\tilde{F}}_{Y}\lesssim\delta^q$. This is clear in the region $r\geq1$, because only $v$ derivatives are glued, and these are exactly the ones on which we have control. For $r\leq1$ we have bounds on all derivatives, thus the fact that the gluing mixes these is no problem. 
		
		Therefore by \cref{linear part of bootstrap}, $\tilde{\phi}$ solving the inhomogeneous wave equation with forcing $\tilde{F}$ satisfies $\norm{\tilde{\phi}}_{X}\lesssim \delta^q$. Using finite speed of propagation, we have $\hat{\phi}=\tilde{\phi}$ for $t<T$. Furthermore, using triangle inequality we get 
		\begin{equation*}
			\begin{gathered}
				\norm{\phi}_{X[0,T]}\leq \delta+C \delta^q,
			\end{gathered}
		\end{equation*}
		where $C$ only depends on parameters of the equation and not $\delta$. Choosing $\delta$ sufficiently small yields the result.
	\end{proof}
	\subsubsection{Linear estimates, proof of \cref{linear part of bootstrap}}
	This part essentially follows from the work done and quoted in \cref{L2 technical estimates}. First, we extend $F$ to $\tilde{F}$ as in \eqref{F extension} with $\norm{\tilde{F}}_Y\lesssim1$. Then, using \cref{energy decay p<2}, we get decay statements for first derivatives: 
	\begin{equation*}
		\begin{gathered}
			\sup_uE^{2-\epsilon}[\partial_v\psi;u],u^{2-\epsilon}E^0[\partial_v\psi;u]\lesssim1\\
			\sup_uu^{2-\epsilon}\mathcal{E}_{\hat{\Sigma}_u}[\phi],	\sup_uu^{2-\epsilon}E^{I}[\tilde{\partial}\phi,u],E^\scri[\phi;0]\lesssim1.
		\end{gathered}
	\end{equation*}
	We can use \cref{energy decay p<4} to get decay right up to $p-\epsilon$ for $\Psi$ and using \cref{Psi to psi} and boundedness of radiation at $\scri$ to get optimal decay statements as required by the $X$ norm. Finally, an energy estimate implies the boundedness of $L^2$ norms on constant time slices.
	\subsubsection{Non-linear estimates, proof of \cref{non-linear part of bootstrap}}

	\paragraph{\textit{Far region:}} In this part, we make the implicit assumption that $r>t/2$ implying $r\sim v\sim t$. All integrals are to be interpreted in this region.
	
	In the exterior region, we must split our nonlinearity as discussed before $r^{(n-1)/2}\partial_v\phi=\partial_v\psi-\psi\frac{n-1}{4r}$ and its derivative $r^{(n-1)/2}r\partial_v^2\phi=r\partial_v^2\psi-\partial_v\psi\frac{n-1}{2r}+\psi\frac{(n-1)(n-3)}{16r^2}$. The undifferentiated term is expected to decay much slower than the rest towards $\scri$\footnote{thus creating the tail for the solution}, but the derivative terms will present more problems with respect to regularity issues. This splitting is captured by
	\begin{lemma}
		For $f$ sufficiently regular, $q>1$
		\begin{equation*}
			\begin{gathered}
				\abs{f-r\partial_vf}^q\lesssim_q \abs{f}^q+\abs{r\partial_v f}^q\\
				\abs{r\partial_v\abs{f-r\partial_vf}^q}\lesssim_q \abs{f}^q+\abs{r\partial_vf}^q+\abs{r^2\partial_v^2f}\abs{r\partial_vf}^{q-1}+\abs{f}^{q-1}\abs{r^2\partial_v^2f}
			\end{gathered}
		\end{equation*}
	\end{lemma}
	\begin{proof}
		Follows from differentiation and triangle inequality.
	\end{proof}
	We turn to the estimates:
	\begin{lemma}		For $F=\abs{\partial_v\phi}^q$,  $\norm{\phi}_{X(u,\infty)}\leq\delta^2\implies\norm{F}_{Y(u,\infty)}\lesssim\delta^q$.
	\end{lemma}
	\begin{proof}
		By scaling it suffices to prove the above for $\delta=1$. The $Y$ norm includes the following two estimates:
		
		a) For $\tilde{F}\in r^{-\frac{n-1}{2}(q-1)-q}\{\abs{r\partial_v\psi}^q,\abs{\psi}^q,r^2\partial_v^2\psi \abs{r\partial_v\psi}^{q-1},r^2\partial_v^2\psi \abs{\psi}^{q-1}\}$
		\begin{equation*}
			\begin{gathered}
				\norm{\phi}_{X(u,\infty)}\leq 1 \implies \int_u^\infty\d u'\, E^{p-\epsilon}[\tilde{F};u']^{1/2}\lesssim1.
			\end{gathered}
		\end{equation*}
		
		b) 	For $F\in \{\abs{\partial_v\phi}^q,u\abs{\partial_v\phi}^{q-1}\abs{\partial\partial_v\phi}\}$
		\begin{equation*}
			\begin{gathered}
				\norm{\phi}_{X(u,\infty)}\leq 1 \implies  u^pE^{\star I}[F;u]\lesssim1
			\end{gathered}
		\end{equation*}
		
		Let's start with a).

		The least decaying term towards $\scri$ is $\abs{\psi}^q$, so we start with that estimate. For the rest, we have a trade between $u$ and $v$, but the decay in the region $t/2<r<3t/4$ is same for all quantities.
		
		Using the $L^\infty$ estimate \cref{l infty estimates from l2}, we get
		\begin{equation*}
			\begin{gathered}
				\int_{v>u}\d v\, r^{(n-1)(1-q)-2q} r^{p-\epsilon}\abs{\psi}^{2q} \lesssim\int_{v>u}\d v v^{(n-1)(1-q)-2q+p-\epsilon}u^{2q((1-p+2\epsilon)/2)}\lesssim u^{q(1-p)+2(q-1)\epsilon}
			\end{gathered}
		\end{equation*}
		since $n-1+p-q(n+1)+1=0$.\footnote{Note that an $\epsilon$ loss is only necessary for closing the $r^p$ estimate bound on the nonlinearity.} The $2\epsilon$ loss for $u$ decay of $\psi$ comes first from $E^{p-\epsilon}$ control and \cref{l infty estimates from l2}. For $\epsilon$ sufficiently small taking square root and integrating in $u$ yields the $E^{p-\epsilon}$ estimate as $-q(1-p)/2>1$. 
		
		Let's next analyse $\abs{r\partial_v\psi}^q$. The $L^\infty$ estimate gives $\norm{r\partial\psi}_{L^\infty}\lesssim v^{(1-p+2\epsilon)/2}$ instead $\norm{\psi}_{L^\infty}\lesssim u^{(1-p+2\epsilon)/2}$. As $v\geq u$, we can use the same estimates as before. 

		For $F=r^2\partial_v^2\psi \abs{r\partial_v\psi}^{q-1},r^2\partial_v^2\psi \abs{\psi}^{q-1}$ we again use the $L^\infty$ estimate on the lower order term, but we need to use $L^2$ bounds on 2nd order part:
		\begin{equation*}
			\begin{gathered}
				\int _{v>2u}\d v\,r^{(n-1)(1-q)-2q} r^{p-\epsilon}(r^2\partial_v^2\psi)^2u^{-(q-1)(p-1)}\lesssim_{\alpha}u^{-(q-1)(p-1-2\epsilon)}\int_{v>2u}\d v\, v^{-1+\mathcal{O}(\epsilon)}(r^2\partial_v^2\psi)^2\\
				\lesssim u^{-(q-1)(p-1)-(p-1)+\mathcal{O}(\epsilon)}\int_{v>2u}\d v\, v^{p+2+\mathcal{O}(\epsilon)}(\partial_v^2\psi)^2\lesssim u^{-q(p-1)+\mathcal{O}(\epsilon)}E^{p+2-\epsilon}[\partial_v^2\psi;u]\lesssim u^{-q(p-1)+\mathcal{O}(\epsilon)},
			\end{gathered}
		\end{equation*}
		where we used $(n-1)(1-q)-2q+p=-1$, $p>0$.
		As before, taking square root and integrating in $u$ yields the required estimates.
		
		b)Recall that
		\begin{equation*}
			\begin{gathered}
				E^{\star I}[F;u]=\int_{\mathcal{D}_{u_1,u_2}} \d t \d r\,r^{n-1}F^2r^{1-\epsilon_0}\jpns{r}^{2\epsilon_0}.
			\end{gathered}
		\end{equation*}
		
		Using \cref{l infty estimates from l2}, it follows that $\abs{\partial_v\phi}^q\lesssim_\epsilon u^{-(p-1-\epsilon)/2}v^{-(n+1-\epsilon)/2}$, thus we get
		\begin{equation*}
			\begin{gathered}				
				\int_{v>2u} \d v r^{n-1}r^{1+\epsilon_0}\abs{\partial_v\phi}^{2q}\lesssim\int_{v>2u}\d v v^{(n-1)(1-q)-2q+1+\mathcal{O}(\epsilon_0)}\lesssim u^{1-p+q(1-p)+\mathcal{O}(\epsilon,\epsilon_0)},
			\end{gathered}
		\end{equation*}
		where we used $p>1$ and $\epsilon$ sufficiently small. Integrating this expression in $u$ yields the $E^{\star I}$ estimate using $1+q(1-p)<-1$ and $\epsilon,\epsilon_0$ sufficiently small.

		For $F=u\abs{\partial\partial_v\phi}\abs{\partial_v\phi}^{q-1}$, we need to use spacetime $L^2$ estimate. Let's integrate over the region $\mathcal{D}=\{v>2u,u\in[u_0,2u_0]\}$
		\begin{equation*}
			\begin{gathered}
				\int_{\mathcal{D}}\d v \d u u^2\abs{\partial\partial_v\phi}^2\abs{\partial_v\phi}^{2(q-1)} r^{n+\epsilon_0}\\
				\lesssim\int_{\mathcal{D}}\d v \d u\,v^{n+\epsilon_0}\abs{\partial\partial_v\phi}^2v^{-(n+1-\epsilon)(q-1)}u^{2-(p-1-2\epsilon)(q-1)}\\
				\lesssim\int_{\mathcal{D}}\d v \d u\,\frac{v^{n-1+\epsilon_0}}{u}\abs{\partial\partial_v\phi}^2 v^{1-(n+1)(q-1)+\mathcal{O}(\epsilon)}u^{3-(p-1)(q-1)+\mathcal{O}(\epsilon)}\\
				\lesssim \int_{\mathcal{D}}\d v \d u\,\frac{v^{n-1+\epsilon_0}}{u}\abs{\partial\partial_v\phi}^2 v^{2-p+\mathcal{O}(\epsilon)}u^{p+2-(p-1)q+\mathcal{O}(\epsilon)}\\
				\lesssim u_0^{2+2-(p-1)q+\mathcal{O}(\epsilon_0)} E^I[\partial\partial_v\phi;u_0,2u_0] \lesssim u^{2-(p-1)q-p+\mathcal{O}(\epsilon_0)},
			\end{gathered}
		\end{equation*}
		where we used $p>2$. Using again that $q(p-1)>2$ we can sum this expression over dyadic intervals to obtain the desired bound.
	
	\end{proof}

	\paragraph{\textit{Near region:}}
	In this part, we only need to work in the region $t/2>r$. Furthermore, without explicitly stating, all integrals are restricted to this region. Also, for this part, there is no difference between $\partial_v$ and $\partial_u$ derivatives in terms of decay, so we'll use a general derivative $\partial$ throughout.
	\begin{lemma}
		For $F=\abs{\partial\phi}^q$,  $\norm{\phi}_{X(\mathcal{A})}\leq\delta^2\implies\norm{F}_{Y(\mathcal{A})}\lesssim\delta^q$.
	\end{lemma}
	\begin{proof}
		Using the appropriate scaling it suffices to prove the above for $\delta=1$.
		
		Let's start with $E^{\star I}[\partial F;u]$. In the equation below, all norm are restricted to the region ${\mathcal{D}_{u_0,2u_0}}$. We use \cref{l infty estimates from l2} and the bounds of the $X$ norm to get
		\begin{equation*}
			\begin{gathered}
				E^{\star I}[u\partial F;u,2u]=\int\d t\d r\, r^{n-1} r^{1-\epsilon_0}\jpns{r}^{2\epsilon_0}(u\partial^2\phi\abs{\partial\phi}^{q-1})^2\\\leq u_0^2\norm{r^{2-2\epsilon_0}\jpns{r}^{4\epsilon_0}\abs{\partial\phi}^{2q-2}}_{L^\infty}\int\d t \d r\,\frac{r^{n-1}(\partial^2\phi)^2}{r^{1-\epsilon_0}\jpns{r}^{2\epsilon_0}}
				\lesssim u_0^{4+2\epsilon_0-(n-2)(q-1)}\norm{r^{n/2-1}\partial\phi}^{2q-2}_{L^\infty}E^{I}[\partial^2\phi;u_0,2u_0]\\\lesssim u_0^{4+2\epsilon_0-(n-2)(q-1)}u_0^{-(p+2)(q-1)}u_0^{-(p+2-\epsilon)}\\\lesssim u_0^{-p}u^{2+2\epsilon_0+\epsilon-(p+n)(q-1)}\leq u_0^{-p}u_0^{-\epsilon_0}
			\end{gathered}
		\end{equation*}
		where we used $2-(n-2)(q-1)-2\epsilon_0>0$ and  $2-(p+n)(q-1)<0$\footnote{Note that $2-(p_c+n)(q_c-1)=0$}.  By choosing $u_0$ as a dyadic sequence, we may sum all these different contributions to get boundedness.
		
		Similarly for the other derivative estimate, in the region ${\mathcal{D}_{u_1,2u_1}}\cap\{r>1\}$ we have
		\begin{equation*}
			\begin{gathered}
				\bigg(\int_{u_1}^{2u_1}\d u E^{p-\epsilon}[r\partial_v\tilde{F}]^{1/2}\bigg)^2\lesssim\int\d t\d r\,u^{1+\epsilon_0}r^{p+2+(n-1)-\epsilon}(\partial^2\phi)^2\abs{\partial\phi}^{2(q-1)}\\
				\lesssim u_1^{1+\epsilon_0}\norm{r^{3+\epsilon_0}r^p\abs{\partial\phi}^{2(q-1)}}_{L^\infty}\int\d t\d r\,\frac{r^{n-1}(\partial^2\phi)^2}{r^{1+\epsilon_0}}\\
				\lesssim  u_1^{1+\epsilon_0}u^{3+\epsilon_0+p-(n-2)(q-1)}\norm{r^{n/2-1}\partial\phi}^{2(q-1)}_{L^\infty}E^{I}[\partial^2\phi;u_1,2u_1]\\
				\lesssim u_1^{1+\epsilon_0}u_1^{3+\epsilon+p-(q-1)(p+3-\epsilon)}u_1^{-(p+2-\epsilon)}
				\lesssim u_1^{2-(p+n)(q-1)+\mathcal{O}(\epsilon,\epsilon_0)}\lesssim u_1^{-\epsilon_0}
			\end{gathered}
		\end{equation*}
		with $3+p>(n-2)(q-1)$ and $\epsilon,\epsilon_0$ sufficiently small. 
		
		Undifferentiated terms are bounded similarly:
		\begin{equation*}
			\begin{gathered}
				E^{\star I}[F;u,2u]=\int\d t\d r\, r^{n-1} r^{1-\epsilon_0}\jpns{r}^{2\epsilon_0}\abs{\partial\phi}^{2q}\leq\norm{r^{4-2\epsilon_0}\jpns{r}^{4\epsilon_0}\abs{\partial\phi}^{2(q-1)}}\int\d t\d r\, \frac{r^{n-1}(\partial\phi)^2}{r^{1-\epsilon_0}\jpns{r}^{2\epsilon_0}}\frac{1}{r^2}\\
				\lesssim u^{4+2\epsilon_0-(n-2)(q-1)}\norm{r^{n/2-1}\partial\phi}^{2q-2}_{L^\infty}E^I[\tilde{\partial}\partial\phi;u]\lesssim u^{-p}u^{2+2\epsilon_0+\epsilon-(p+n)(q-1)}\leq u^{-p}u^{-\epsilon_0}.
			\end{gathered}
		\end{equation*}
		where all norms are restricted to the region ${\mathcal{D}_{u_1,2u_1}}$.
		
		\begin{equation*}
			\begin{gathered}
				\bigg(\int_{u_1}^{2u_1}\d u E^{p-\epsilon}[\tilde{F}]^{1/2}\bigg)^2\lesssim\int\d t\d r\,u^{1+\epsilon_0}r^{p+(n-1)-\epsilon}(\partial\phi)^2\abs{\partial\phi}^{2(q-1)}\\
				\lesssim u_1^{1+\epsilon_0}\norm{r^{3+\epsilon_0}r^p\abs{\partial\phi}^{2(q-1)}}_{L^\infty}\int\d t\d r\,\frac{r^{n-1}(\partial\phi)^2}{r^{1+\epsilon_0}}\frac{1}{r^2}\lesssim u_1^{-\epsilon_0}.
			\end{gathered}
		\end{equation*}
		where all norms are restricted to the region ${\mathcal{D}_{u_1,2u_1}}\cap\{r>1\}$.
	\end{proof}
	
	\subsection{Glassey Strauss system}\label{gs gwp}
	To recap, we're to prove global existence for
	\begin{equation*}
		\begin{gathered}
			\Box\phi_1=\abs{\phi_2}^{q_1}\\
			\Box\phi_2=\abs{\partial_t\phi_1}^{q_2}
		\end{gathered}
	\end{equation*}
	with 
	\begin{equation}\label{gs conditions}
		\begin{gathered}
			q_1<q_{Glassey}=\frac{n+1}{n-1},\,1<q_2q_1(b+q_2a)
		\end{gathered}
	\end{equation}
	
	where $a=\frac{n-1}{2}(q_1-1)-1<0,\,b=\frac{n-1}{2}(q_2-1)-1>0$. The heuristic suggests the following behaviour
	\begin{equation}\label{Glassey Strauss system asymptotic behaviour}
		\begin{aligned}
			&\psi_1|_{\mathcal{N}_+}\sim t^{-s_1} && s_1=a+\min(0,q_1s_2-1)\\
			&\psi_1|_{\mathcal{N}_\scri}\sim t^{-a}\\
			&\psi_2|_{\mathcal{N}_+}\sim t^{-s_2} && s_2=b+q_2a\\
			&\psi_2|_{\mathcal{N}_\scri}\sim1\\
			&\partial_v\psi_2|_{\mathcal{N}_\scri}\sim t^{-(s_2+1)}.
		\end{aligned}
	\end{equation}
	The condition for existence of solution is that the behaviour of $\psi_1$ at $I^+$ creates no stronger tail than its behaviour at $\scri$ for $\psi_2$. 
	
	\begin{remark}[Limitation of current $r^p$ method]
		Note, that the $r^p$ method understands dispersion via decay of energy on a certain foliation. However, for the above system, we expect global solutions even when the usual energy associated to $\phi_1$ doesn't decay, even worse, it isn't finite. Using the asymptotics from \eqref{Glassey Strauss system asymptotic behaviour}, we get optimal value $p_1=2(a+1)-1=(q_1-1)(n-1)-1$. This must be positive for the $r^p$ method to be applicable, so we get the \textit{technical} condition \footnote{as opposed to physical, thereby we might be able to get around this barrier} $q_1>1+\frac{1}{n-1}$. Indeed, doing a simple energy estimate in $\mathcal{D}_{u_1,u_2}$ shows that this requirement on $q_1$ is necessary and sufficient for the energy to be bounded (assuming control on $\phi_2$). To overcome these, one needs to introduce (negatively) weighted energy, eg. via the multiplier $1/t^\alpha\partial_t,\,\alpha>0$. To overcome similar problems, see \cite{keir_weak_2018}.
	\end{remark}
	
	We restrict to $n=3$ and spherically symmetry, thus the system simplifies to a system of nonlinear wave equations in $\R^{1+1}$:
	\begin{equation}\label{glassey strauss 1+1}
		\begin{gathered}
			\Box\psi_1=r\abs{\frac{\psi_2}{r}}^{q_1}\\
			\Box\psi_2=r\abs{\frac{\partial_t\psi_1}{r}}^{q_2}.
		\end{gathered}
	\end{equation}
	
	As such, we can use $L^\infty$ estimates, moreover, there is a  \textit{gain of derivative}, as the first nonlinearity is undifferentiated. This will play a role behind the scenes. The norms that we are going to control are
	\begin{equation*}
		\begin{gathered}
			\norm{\frac{\partial_t\psi_1}{r}\jpns{v}^{1+a}\jpns{u}^{s_1+1-a}}_{L^\infty},\norm{\frac{\psi_2}{r}\jpns{v}\jpns{u}^{s_2}}_{L^\infty},
			\norm{\partial_v\psi_2 \jpns{v}^{s_2+1}}_{L^\infty}.
		\end{gathered}
	\end{equation*}
	Let's call the sum of all three $\norm{\psi_1,\psi_2}_X$ and $X_i$ the ones restricted to either component. The restriction to some subset $\mathcal{A}\subset\R^{3+1}$ will be called $X(\mathcal{A})$.

	\begin{lemma}\label{x bound for linear}
		For $\psi=r\phi$  a solution to
		\begin{equation*}
			\begin{gathered}
				(\partial_t^2-\partial_r^2)\psi_i=0\\
				\partial_t^{(j)}\phi_i(0)=\phi_i^{(j)}\qquad i\in\{1,2\},j\in\{0,1\}
			\end{gathered}
		\end{equation*}
		with $\phi_i^{(j)}$ supported in $\{r<1\}$ we have 
		\begin{equation*}
			\begin{gathered}
				\norm{\psi_1,\psi_2}_X\lesssim \norm{\phi_i^{(0)},\phi_i^{(1)}}_{W^{2,\infty}\times W^{1,\infty}}.
			\end{gathered}
		\end{equation*}
	\end{lemma}
	\begin{proof}
		The support of $\psi_i$ will be in $u\in[-1,1]$ due to the strong Huygens principle, thus we can drop the $u$ weights. All estimates follow from 
		\begin{equation*}
			\begin{gathered}
				\sup_{x,y} \abs{\frac{f(x+y)-f(x-y)}{y}}\leq \norm{f}_{W^{1,\infty}}
			\end{gathered}
		\end{equation*}
		and the exact solution
		\begin{equation*}
			\begin{gathered}
				\partial_t\psi_1(t,r)=\begin{cases}
					\partial_v\psi_1(0,t+r)+\partial_u\psi_1(0,r-t)\qquad t\leq r\\
					\partial_v\psi_1(0,t+r)-\partial_v\psi_1(0,t-r)\qquad t\geq r
				\end{cases}\\
			 \partial_v\psi_2(t,r)=\partial_v\psi_2(0,t+r)\\
			 2\psi_2(t,r)=\psi_2(0,t-r)+\psi(0,t+r)+\int_{r-t}^{r+t}\partial_t\psi_2(0,x).
			\end{gathered}
		\end{equation*}
	We show it for $\psi_1$, the rest follows similarly. For $t\leq r$ region we use $\frac{1}{r}\leq\frac{1}{t}$ and $v\leq1$ to get 
	\begin{equation*}
		\begin{gathered}
			\norm{\frac{\partial_t\psi_1}{r}\jpns{v}^{1+a}}_{L^\infty}\lesssim \norm{\frac{1}{r}(\phi_1(r+t)-\phi_1(r-t))}_{L^\infty}+\norm{(\partial_t\phi_1(0,r\pm t),\partial_r\phi_1(0,r\pm t))}_{L^\infty}.
		\end{gathered}
	\end{equation*}
	This is clearly bounded by $W^{1,\infty}$ norm. For $t\geq r\geq1$ the the same estimate holds, but we need to use the cancellation $v\sim r\sim t$ in $\supp\psi_1$. For $t\geq r,r\leq1$, we get $v\leq1$, so
	\begin{equation*}
		\begin{gathered}
			\norm{\partial_t\psi_1/r}_{L^\infty}\lesssim \norm{\frac{1}{r}(\partial\phi_1(0,t+r)-\partial\phi_1(t-r))}_{L^\infty}+\norm{\phi_1^{(0)},\phi_1^{(1)}}_{W^{1,\infty}\times W^{0,\infty}},
		\end{gathered}
	\end{equation*}
	where we grouped terms that are bounded as before in the second term. Note, that in order to bound the first quantity, we need to control second derivative of $\phi_1$. 
	\end{proof}
	\begin{remark}
		One may exchange compactly supported initial data, with data such that the $X$ norm of the solution is of size $\delta$.
	\end{remark}

The main iteration step to bound the nonlinear terms is the following:
	\begin{prop}\label{gs iteration}
		For $q_1,q_2$ satisfying \cref{gs conditions} and $(\psi_1^{(0)},\psi^{(0)}_2)$ with finite $X$ norm, there is a unique solution $(\psi_1^{(1)},\psi^{(1)}_2)$ to
		\begin{equation}\label{iterative equation}
			\begin{gathered}
				\Box\psi_1^{(1)}=r\abs{\frac{\psi^{(0)}_2}{r}}^{q_1}\\
				\Box\psi^{(1)}_2=r\abs{\frac{\partial_t\psi^{(0)}_1}{r}}^{q_2}\\
				\psi^{(1)}_1(0)=\psi^{(1)}_2(0)=\partial\psi^{(1)}_1(0)=\partial_t\psi^{(1)}_2(0)=0
			\end{gathered}
		\end{equation}
		with estimate
		\begin{equation*}
			\begin{gathered}
				\norm{\psi_1^{(1)},\psi_2^{(1)}}_{X}\lesssim\norm{\psi_1^{(0)},\psi_2^{(0)}}_X^{q_1}
			\end{gathered}
		\end{equation*}
		provided the right hand side is smaller or equal to 1.
	\end{prop}
	\begin{remark}
		Since we work in the small data regime, the $\norm{\cdot}_X<1$ is not a real restriction, but in this linear case, its needed because we have two distinct powers $q_1,q_2$.
	\end{remark}
	\begin{proof}
		Using scaling, it suffices to prove the statement for normalised data.
		
		\textit{Part 1: Non-linear estimates}
		We claim that
		\begin{equation*}
			\begin{gathered}
				\norm{\psi^{(0)}_1,\psi^{(0)}_2}_{X[0,T]}\leq1\\
				\implies\norm{\partial_v^\alpha\bigg(r\abs{\frac{\psi^{(0)}_2}{r}}^{q_1}\bigg)\jpns{v}^{q_1-1+\alpha}\jpns{u}^{s_2q_1}}_{L^\infty[0,T]},\norm{r\abs{\frac{\partial_t\psi^{(0)}_1}{r}}^{q_2}\jpns{v}^{(1+a)q_2-1}\jpns{u}^{(1+s_1-a)q_2}}_{L^\infty[0,T]}\lesssim1,
			\end{gathered}
		\end{equation*}
		for $\alpha\in\{0,1\}$. Indeed this follows using some algebra and the fact $r\leq v$.
		
		\textit{Part 2: Linear estimates}
		Using the estimates from \cref{l infinite theory}, we conclude that first
		\begin{equation*}
			\begin{gathered}
				\norm{\psi_2}_{X_2}\lesssim1
			\end{gathered}
		\end{equation*}
		using 
		\begin{equation*}
			\begin{gathered}
				(1+s_1-a)q_2>1\\
				(1+a)q_2>2\\
				s_2=(1+a)q_2-2.
			\end{gathered}
		\end{equation*}
		The first condition defines the critical curve. The second follows from it by algebra (and guarantees the existence of radiation field for $\psi_2$). The third one is simply the definition of $s_2$.
		
		Similarly, the non-linear estimates on $\psi_1$ imply
		\begin{equation*}
			\begin{gathered}
				\norm{\psi_1}_{X_1}\lesssim1
			\end{gathered}
		\end{equation*}
	using
	\begin{equation*}
		\begin{gathered}
			1+a=q_1-1\\
			s_1+1-a=1-\max(0,1-s_2q_1).
		\end{gathered}
	\end{equation*}

	\end{proof}
	
	It's a simple exercise to prove global existence from here using a contraction mapping
	\begin{theorem}\label{strauss glassey main theorem}
		There exists $\epsilon>0$ sufficiently small such that for all initial data supported in $\{r<1\}$ and $\norm{\phi_1,\phi_2}_{W^{2,\infty}}<\epsilon$, \eqref{glassey strauss 1+1} has global solution $\phi_i=\psi_i/r$ in the sense
		\begin{equation}\label{gs weak sol}
			\begin{gathered}
				\psi_1=\psi_1^{(0)}+\int_{\mathcal{D}_{t,r}}\d r'\d t' r'^{1-q_1}\abs{\psi_2}^{q_1}\\
				\psi_2=\psi_2^{(0)}+\int_{\mathcal{D}_{t,r}}\d r'\d t' r'^{1-q_2}\abs{\partial_t\psi_1}^{q_2}
			\end{gathered}
		\end{equation}
	where $\mathcal{D}_{t,r}=\{(r',t'):t-r\leq t'+r'\leq t+r\}$ and $\phi_i^{(0)}=\psi_i^{(0)}/r$ are the linear solutions with the same data.
	\end{theorem}
	\begin{proof}
		\textit{Boundedness:}
		Let's set $\psi_i^{(0)}$ to be a solution to the linear equation with the above data.\cref{x bound for linear} says that $\norm{\psi_1^{(0)},\psi_2^{(0)}}_X\leq c_1\epsilon$. Let's iteratively define $\hat{\psi}^{(m)}_i$ as the solution to \eqref{iterative equation}, with the right hand side replaced by $\psi_i^{(0)}+\hat{\psi}_i^{(m-1)}$ terms. By \cref{gs iteration}, it is clear that $\norm{\hat{\psi}_1^{(1)},\hat{\psi}_2^{(1)}}_X\leq c_2 (\epsilon^{q_1}+\epsilon^{q_2})$. Picking $\epsilon$ sufficiently small, we know that $\norm{\hat{\psi}_1^{(m)},\hat{\psi}_2^{(m)}}_X\leq c_3\epsilon\,\forall m$.

		\textit{Contraction:} We study convergence in the rougher space
		\begin{equation*}
			\begin{gathered}
				\norm{\psi_1}_{Y_1}=\norm{\partial_t\psi_1\jpns{v}^{a}\jpns{u}^{s_1+1-a}}_{L^\infty}\\
				\norm{\psi_2}_{Y_2}=\norm{\frac{\psi_2}{r}\jpns{v}\jpns{u}^{s_2}}_{L^\infty}.
			\end{gathered}
		\end{equation*}
		The upshot, is that we can use the estimate on $\partial_v\psi,\partial_u\psi$ proved in \cref{l infinite theory} to get
		\begin{equation*}
			\begin{gathered}
				\norm{\partial_t\hat{\psi}_1^{(m)}}_{Y_1}\lesssim\norm{r\abs{\frac{\hat{\psi}_2^{(m-1)}}{r}}^{q_1}\jpns{v}^{q1-1}\jpns{u}^{s_2q_1}}_{L^\infty}\lesssim\norm{\hat{\psi}^{(m-1)}_2}_{X_2}^{q_1}.
			\end{gathered}
		\end{equation*}
		Importantly, we do not need to take $\partial_v$ derivative.
		
		We have for all $f,g\in\mathcal{C}^\infty$
		\begin{equation*}
			\begin{gathered}
				\norm{r\bigg(\abs{\frac{f}{r}}^{q_1}-\abs{\frac{g}{r}}^{q_1}\bigg)\jpns{
						v}^{q_1-1}\jpns{u}^{s_2q_1}}_{L^\infty}\lesssim
				\norm{r\bigg(\abs{\frac{f}{r}}^{q_1-1}+\abs{\frac{g}{r}}^{q_1-1}\bigg)\bigg(\abs{\frac{f}{r}}-\abs{\frac{g}{r}}\bigg)\jpns{
						v}^{q_1-1}\jpns{u}^{s_2q_1}}_{L^\infty}\\\lesssim(\norm{f}_{X_2}+\norm{g}_{X_2})^{q_1-1}\norm{f-g}_{Y_2}
			\end{gathered}
		\end{equation*}
		and similarly
		\begin{equation*}
			\begin{gathered}
				\norm{r\bigg(\abs{\frac{\partial_tf}{r}}^{q_2}-\abs{\frac{\partial_tg}{r}}^{q_2}\bigg)}_{L^\infty}\lesssim\norm{(\partial_tf-\partial_tg)\bigg(\abs{\frac{\partial_tf}{r}+\frac{\partial_tg}{r}}\bigg)^{q_2-1}}_{L^\infty}\\
				\lesssim\norm{f-g}_{Y_1}(\norm{f}_{X_1}+\norm{g}_{X_1})^{q_2-1}.
			\end{gathered}
		\end{equation*}
		These two estimates together with \cref{l infinite theory} show that 
		\begin{equation*}
			\begin{gathered}
				\norm{\hat{\psi}_1^{(m)}-\hat{\psi}^{(m-1)}_1,\hat{\psi}_2^{(m)}-\hat{\psi}^{(m-1)}_2}_Y\lesssim\epsilon^{q_1-1}\norm{\hat{\psi}_1^{(m-1)}-\hat{\psi}_1^{(m-2)},\hat{\psi}_2^{(m-1)}-\hat{\psi}^{(m-2)}_2}_Y.
			\end{gathered}
		\end{equation*}
		Picking $\epsilon$ sufficiently small and using the contraction mapping theorem, we get a convergent sequence in $Y$ space $\hat{\psi}_i^{(m)}\underset{Y}{\to}\hat{\psi}_i$. By boundedness in $X$, the limit must also lie in the $X$ space.
		
		\textit{Solving the equation:} As $\hat{\psi}^{(m)}$ converges in $Y$ and is bounded in $X$, thus it also converges in 
		\begin{equation*}
			\begin{gathered}
				\norm{\hat{\psi_1},\hat{\psi}_2}_Z=	\norm{\frac{\partial_t\psi_1}{r^(1-1/q_1)}\jpns{v}^{1+a}\jpns{u}^{s_1+1-a}}_{L^\infty}+\norm{\frac{\psi_2}{r}\jpns{v}\jpns{u}^{s_2}}_{L^\infty}.
			\end{gathered}
		\end{equation*}
		Therefore, by dominated convergence theorem we can take the limit in \cref{gs weak sol} inside the integral, to conclude that the limit also solve \cref{gs weak sol}.
	\end{proof}

	\subsection{$n$-Strauss system}
 	In this section, we will prove that the equation \cref{strauss system} has global solutions if \cref{n variable Strauss condition} has a solution. More precisely, we will require that this property is stable under small perturbations of the $c_{ij}$. Furthermore, we will impose the constraint $c_{ij}>2$ and work in $\R^{3+1}$. These simplify $\cref{n variable Strauss condition}$ to 
 	\begin{equation}\label{nStraussCondition3}
 		\begin{gathered}
 			s_i=\min_j(c_{ij}-2+\min(0,c_{ij}s_j-1):a_{ij}\neq 0).
 		\end{gathered}
 	\end{equation}
 
 Before starting the PDE estimates, we need to understand \cref{nStraussCondition3} a bit more. 
 
 \begin{lemma}\label{minmax convergence}
 	For $\alpha_{ij},\beta_{ij},\gamma_{ij}\in\R$ with $\gamma_{ij}>0$ if the equation
 	\begin{equation}\label{min max equation}
 		\begin{gathered}
 			x_{i}=F_i[x] \qquad\text{where} \quad 	F_i[y]=\min_j(\alpha_{ij},\alpha_{ij}+\beta_{ij}+\gamma_{ij}y_j),
 		\end{gathered}
 	\end{equation}
 	has a solution, than it has a maximal one. That is there exists $x_i$ solution to \cref{min max equation} with the property that for any other $y_i$ solving \cref{min max equation} $x_i\geq y_i$. Moreover,  $x=\lim_nF^n[\hat{\alpha}]$, $\hat{\alpha}=\min_j{\alpha_{ij}}$.
 \end{lemma}
 \begin{proof}
 	All solutions satisfy $x_i\leq \alpha_i$. For two solutions, $x_i,\,y_i$, $\max(x_i,y_i)$ solve \cref{min max equation}. Indeed
 	\begin{equation*}
 		\begin{gathered}
 			\min_j(\alpha_{ij},\alpha_{ij}+\beta_{ij}+\gamma_{ij}\max(x_j,y_j))\\
 			=\max(\min_j(\alpha_{ij},\alpha_{ij}+\beta_{ij}+\gamma_{ij}x_j),\min_j(\alpha_{ij},\alpha_{ij}+\beta_{ij}+\gamma_{ij}y_j))=\max(x_i,y_i).
 		\end{gathered}
 	\end{equation*}
 	Therefore the maximum must be unique. As the equation defines a closed set with upper bound on all $x_i$, the maximum $\bar{\alpha}$ must be attained.
 	
 	For the second part, let $F_i[z]$ be the right hand side of \cref{min max equation} evaluated at $z_i$. Then, given any solution of $x_i=F_i[x]$ with $x_i\leq z_i$, $x_i\leq F_i[z]$, as $\gamma_{ij}\geq0$. In particular $F_i^n[\hat{\alpha}]\geq \bar{\alpha}$. As $F_i$ is monotone decreasing in all components, and there is at least one solution, therefore $F_i^n[\alpha]$ must converge to the maximal solution.
 \end{proof}

\begin{definition}
	Given a solution $x$ to \cref{min max equation}, we may define a directed graph $G_x$ as follows. For each index $i$, let $G_x$ have a vertex, and an edge $j\to i$ if the minimum $\min_k(\alpha_{ik},\alpha_{ik}+\gamma_{ik}s_k+\beta_{ik})$ is attained at $k=j$ with the second expression.
\end{definition}

\begin{lemma}\label{loop lemma}
	Fix $\alpha_{ij},\beta_{ij},\gamma_{ij}$ and $\epsilon>0$ such that \cref{min max equation} has a solution for all $\tilde{\alpha}_{ij}\in B_\epsilon{\alpha_{ij}},$ $\tilde{\beta}_{ij}\in B_\epsilon{\beta_{ij}},$ $\tilde{\gamma}_{ij}\in B_\epsilon{\gamma_{ij}}$. Than $G_{\bar\alpha}$ corresponding to the maximum solution of \cref{min max equation} with $\alpha_{ij},\beta_{ij},\gamma_{ij}$ and $\gamma_{ij}>1$ cannot have a loop.
\end{lemma}
\begin{proof}
	Due to the monotonicity of $F$, \cref{min max equation} has a solution iff 
	\begin{equation}\label{minmax inequality}
		\begin{gathered}
			x_i\leq F_i[x].
		\end{gathered}
	\end{equation}
	Indeed a solution of \cref{min max equation} solves \cref{minmax inequality}. For the other direction, note that \cref{minmax inequality} is defined by the finite intersection of hyperplanes $\mathcal{A}$. Thus, given $x\in\mathcal{A}$, we can run a simplex algorithm to find the maximal solution $\hat{\alpha}$ as defined in \cref{minmax convergence}. 
	
	Furthermore, note that the requirement that equations with coefficients close to the original must have solutions is simply the statement that $\mathcal{A}$ is open.
	
	The maximal solution is characterised by $\bar{\alpha}\cdot v\geq x\cdot v $ $\forall x\in\mathcal{A}$ and $v_i\geq0$, furthermore inequality holds when $v_i>0$. Assume $G_{\bar\alpha}$ has a loop over indices $\{1,2,...,d'\}$. As $\mathcal{A}$ is open, there is $\epsilon>0$ such that $x^{(\epsilon)}=\bar\alpha-\epsilon v\in\mathcal{A}$ for $v_i>0$. Note, however, that the existence of such a loop implies
	\begin{equation*}
		\begin{gathered}
			x^{(\epsilon)}_1\leq x^{(0)}_1+\gamma_{12}\gamma_{23}...\gamma_{d'1}(x^{(\epsilon)}_1-x^{(0)}_1).
		\end{gathered}
	\end{equation*}
	As $\gamma_{ij}>1$, this inequality has no solution for $x^{(\epsilon)}<x^{(0)}$, which is a contradiction.
\end{proof}
 
	\begin{theorem}\label{n strauss main theorem}
		Consider the equation \eqref{n variable Strauss}. Fix $a_{ij},c_{ij}$ and $\hat{\epsilon}>0$ such that \eqref{n variable Strauss condition} has a solution for all $\tilde{c}_{ij}\in[c_{ij}-\hat{\epsilon},c_{ij}+\hat{\epsilon}]$, $c_{ij}>2$. Then, there exists $\delta>0$ such that \eqref{n variable Strauss} has global solution for all initial data supported in $\{\abs{x}<1\}$ with $\norm{\phi_i}_{H^{2}}\leq\delta$.
	\end{theorem}
	\begin{remark}
		Note, that the restriction $c_{ij}>2$ restricts all field to have a radiation field ($\sigma_i=0\, \forall i$).
	\end{remark}
	\begin{proof}
		Let's consider the graph $G_x$ for the equation \cref{nStraussCondition3}, and let $s_i$ be a maximal solution. \cref{loop lemma} implies in particular, that there are no loops in $G_x$.
		
		Therefore, we can associate to each vertex an integer $\mathcal{N}(i)$ that is its maximum distance on a directed path, starting with $0$. Fix $\epsilon>0$ sufficiently small, fixed only later, and $A=\max_{ij}c_{ij}$. Let
		\begin{equation}\label{n Strauss decay coefficients}
			\begin{gathered}
				\alpha=1-\epsilon\\
				\beta_i=s_i-A^{1+2\mathcal{N}(i)}\epsilon\\
				\gamma_{ij}=c_{ij}(1-\epsilon)\\
				\delta_{ij}=c_{ij}\beta_j
			\end{gathered}
		\end{equation}
		Restrict $\epsilon$, such that $c_{ji}s_i>1\implies\delta_{ji}>1$.  This implies, that a directed graph constructed from $\delta_{ij}$ has the same edges as one from $c_{ij}s_j$.
		
		It's easy to see that the problem is locally well posed in $H^{2}\times H^{1}$, because $H^{2}\subset L^\infty$. Let's say a solution to \eqref{n variable Strauss} exists in some time slice $\mathcal{D}_T=[0,T]$. Then, John's $L^\infty$ estimate \cref{john l infty} says
		\begin{equation*}
			\begin{gathered}
				\norm{v^{\alpha}u^{\beta_i}\phi_i}_{L^\infty(\mathcal{D}_T)}\lesssim c_{data}+\sum_j\norm{v^{\gamma_{ij}}u^{\delta_{ij}}\abs{\phi_j}^{c_{ij}}}_{L^\infty}
			\end{gathered}
		\end{equation*}
		if
		\begin{equation*}
			\begin{gathered}
				\alpha<1\\
				\beta_i\leq\min_j(\gamma_{ij}-2+\min(0,\delta_{ij}-1)).
			\end{gathered}
		\end{equation*}
		One checks that using the definitions in \eqref{n Strauss decay coefficients}, the above conditions hold. Indeed
		\begin{equation*}
			\begin{gathered}
				\min_j(\gamma_{ij}-2+\min(0,\delta_{ij}-1))+\epsilon A^{1+2\mathcal{N}(i)}\\
				\geq \min(\{c_ij-2\}\cup\{c_{ij}-3+c_{ij}s_j:\delta_{ij}<1\})-\epsilon(\max_j c_{ij}+\max_{\mathcal{N}(j)<\mathcal{N}(i)}A^{1+2\mathcal{N}(j)})+\epsilon A^{1+2\mathcal{N}(i)}\geq s_i.
			\end{gathered}
		\end{equation*}
		
		This allows all $L^\infty$ norms to be bounded via a bootstrap. More precisely, the norm
		\begin{equation*}
			\begin{gathered}
				\norm{\phi}_X=\sum_i \norm{v^{\alpha}u^{\beta_i}\phi_i}_{L^\infty}
			\end{gathered}
		\end{equation*}
		will stay bounded throughout the evolution if it is small enough for the linear problem.
		
		Since $L^\infty$ is a controlling norm, that is, if it stays bounded during the maximal time of existence, the solution must be global \footnote{for details on this concept, see \cite{tao_nonlinear_2006} chapter 3.3}. Indeed, assume that the maximal solution $\phi_i$ is only defined on $[0,T)$ for $T<\infty$, with $\norm{\phi_i}_{L^\infty([0,T)}\leq c$. Then, a simple energy estimate will yield that the $H^2$ norm is bounded up to time $T$. As the system is well posed in $H^2$, we can extend past $T$.
	\end{proof}

	\section{Blow-up results}\label{blow up results}
	All the global ill-posedness results will be proved in $\mathcal{C}^2(\R^{3+1})$ using the techniques of John \cite{john_blow-up_1979} and Yang, Zhou \cite{yang_revisit_2016}. We will further restrict to spherical symmetry, but this condition may be dropped using spherical averages as done in both the above works. We will also impose more conditions on the initial data to make the proofs easier, but we believe that these can be weakened in many scenarios. In any case, the restrictions are in line with the general philosophy, that to disprove global well-posedness its sufficient to create specific examples. 
	
	Furthermore, note that for equations which have non-smooth non-linearity, $\mathcal{C}^2$ global ill-posedness is not always the most satisfactory statement, as even local well-posedness may not be available for any space including $\mathcal{C}^2$. This is however not a major problem, and these global ill-posedness results survive in weaker spaces using the techniques from \cite{hidano_life_2016}. The reason being, that the ill-posedness is proved via integrated quantities, thus regularity of the solution can be weakened significantly.
	
	The non-existence results will all depend on an extension of generalised Gronwall inequality from \cite{yang_revisit_2016}.
	\begin{lemma}[Generalised Gronwall inequality]\label{gronwall}
		Consider the system of ordinary differential inequalities
		\begin{equation*}
			\begin{gathered}
				\partial_tx_i(t)\geq c_i t^{-\alpha_i}x_{i-1}^{p_i}(t)
			\end{gathered}
		\end{equation*}
		for $i\in\{1,2,...,n\}$ with $t>1,\,c_i>0,\,p_i\geq1\,\Pi_jp_j>1,\,\Sigma_i\alpha_i\leq n$ and convention $x_{0}:=x_n$. The above has no global solution if all initial data are positive.
	\end{lemma}
	\begin{proof}
		\textit{Change of variables:}
		Without loss of generality, assume that $p_n\leq p_{n-1}\leq...\leq p_1$. The conditions of the lemma imply $p_1>1$.
		
		Using variable $s=\log(t)$  we get (by abusing notation)
		\begin{equation*}
			\begin{gathered}
				\partial_sx_i(s)\geq c_ie^{s\beta_i}x_{i-1}^{p_i}(s)
			\end{gathered}
		\end{equation*}
		with $\beta_1=1-\alpha_i$ and $\sum_i\alpha_i<n\implies\beta=\sum_i\beta_i\geq0$. A short computation yields
		\begin{equation}\label{ode in gronwall}
			\begin{gathered}
				\partial_sx\geq xce^{s\beta} \sum_i \frac{x_{i-1}^{p_i}}{x_i}.
			\end{gathered}
		\end{equation}
		where $x=\prod_i x_i,\,c=\prod_i c_i$.

		Fix $q_i=p_1^{\frac{i}{n}}>1$ and $Q=\sum_iq_i$. so that $q_ip_i-q_{i-1}>p_1^{1/n}-1=c$ for all $i\in{1,2..,n}$. Therefore, by Young's inequality we get		
		\begin{equation*}
			\begin{gathered}
				\sum_i \frac{x_{i-1}^{p_i}}{x_i}\geq \sum_i \frac{q_i}{Q}\frac{x_{i-1}^{p_i}}{x_i}\geq \big(\prod_i x_{i-1}^{p_{i}q_i-q_{i-1}}\big)^{1/Q}\geq \big(\prod_i x_{i-1}\big)^{c/Q}.
			\end{gathered}
		\end{equation*}
		Substituting this into \cref{ode in gronwall} we get
		\begin{equation*}
			\begin{gathered}
				\partial_s x\geq c e^{s\beta}x^{1+c/Q}.
			\end{gathered}
		\end{equation*}
		A usual comparison argument against 
		\begin{equation*}
			\begin{gathered}
				\partial_sy=cy^{1+c/Q}
			\end{gathered}
		\end{equation*}
		yields finite time blow-up.
	\end{proof}

	\subsection{$n$-Strauss system}\label{strauss section}
	For a solution $\bar{\phi}_i$ for \eqref{n variable Strauss}, let's split it to linear ($\tilde{\phi}$) and nonlinear part ($\phi$):
	\begin{equation}\label{split}
		\begin{gathered}
			\Box\tilde{\phi}_i=0\\
			\Box\phi_i=\Box\bar{\phi}_i\\
			\phi_i(0)=\partial\phi_i(0)=0,\,\tilde{\phi}_i(0)=\bar{\phi}_i(0),\,\partial_t\tilde{\phi}_i(0)=\partial_t\bar{\phi}_i(0).
		\end{gathered}
	\end{equation}

	\begin{definition}\label{sufficiently positive data}
		A spherically symmetric initial data $\bar{\phi}_i^{(0)},\bar{\phi}_i^{(1)}\in\mathcal{C}^2(B_1\to\R)$ such that the corresponding linear solution \eqref{split} satisfies
		\begin{equation*}
			\begin{gathered}
				(r\tilde{\phi}_i)\in[c,2c] \qquad t+r>1,t-r\in(-1,-1+\delta)
			\end{gathered}
		\end{equation*}
		for some $c,\delta>0$ is called sufficiently positive data.
	\end{definition}
	
	\begin{theorem} Fix $a_{ij}\geq 0$ and $c_{ij}>2$. Then, \label{blow up for n strauss theorem}
		\cref{n variable Strauss} has no global $\mathcal{C}^2$ solutions starting from sufficiently positive data if \cref{n variable Strauss condition} does not hold.
	\end{theorem}
	\begin{remark}
		Note, that some restriction on initial data is necessary as seen from
		\begin{equation*}
			\begin{gathered}
				\Box\phi_1=\abs{\phi_2}^{q_1}+\abs{\phi_3}^{q_2}\\
				\Box\phi_2=\abs{\phi_1}^{q_1}\\
				\Box\phi_3=0
			\end{gathered}
		\end{equation*}
		with $q_1=2.5,q_2=2.1$. This equation has global small data solution if the initial data for $\phi_3$ is trivial (\cite{del_santo_global_1997}), but as we'll see, not for all small data. Alternatively, we could require that the system doesn't decouple to subsystems, as $\phi_3$ in the above example.
	\end{remark}
		
	We will prove the above theorem by contradiction. First, let's create the sufficient tails for the system.
	
	\begin{definition}
		Let's define the size of the tails that are iteratively generated:
		\begin{equation}\label{iteratively defined tails}
			\begin{gathered}
				s^0_i=\min_j\{c_{ij}-2\} \\
				s^{(m)}_i=\min_j\{c_{ij}-2+\min(0,c_{ij}s_j^{(m-1)}-1):a_{ij}\neq 0\} \qquad m\geq 1
			\end{gathered}
		\end{equation}
	\end{definition}

	The tail creation for the system now follows an iterative procedure:
	\begin{lemma}\label{iterative tail generation}
		For sufficiently positive initial data $a_{ij}>0$ and  $\forall \,m\geq0$, any global solution to \eqref{n variable Strauss} satisfies
		\begin{equation}
			\begin{gathered}
				r\phi\gtrsim_{c,m} u^{-s^{(m)}_i} \qquad u>u_0.
			\end{gathered}
		\end{equation}
		for sufficiently large $u_0$ with constant that depends on $c$ in \cref{sufficiently positive data} and $m$.
	\end{lemma}
	\begin{proof}
		This follows from iterated application of \cref{lower bound}. 	
	\end{proof}
	
	Now, we can finish the proof with Gronwall lemma

	\begin{proof}[Proof of \cref{blow up for n strauss theorem}] 
		Assume there exists a global solution ($\bar{\phi}$). 
		
		Note the following. As \cref{n variable Strauss condition} is assumed not to have a solution, by \cref{minmax convergence}, we know that $s^{(m)}$ does not converge in $\R^d$. Furthermore, from the iteration scheme, it follows that $s^{(m)}_i$ is monotone decreasing for each index $i$. Therefore, there exist some $j_1\in\{1,2,...,d\}$ such that $s_{j_1}^{(m)}\underset{m\to\infty}{\to}-\infty$. Observing the iteration scheme, there must be some $j_2$ for which $a_{j_1j_2}\neq0$ and $s_{j_2}^{(m)}\underset{m\to\infty}{\to}-\infty$. Repeat this process until we find $j_l\in\{j_1,j_2,..,j_{l-1}\}$. Restricting to a subset and renaming indices, we proved that there is a subset $J=\{1,2,...,d'\}$ such that $a_{12},a_{23},....,a_{d'1}\neq0$ and $s_j\to -\infty$ for $j\in[1,d']$. In particular, because $c_{ij}>1$ this means that for $N$ sufficiently large $\forall m\geq N$
		\begin{equation*}
			\begin{gathered}
				\sum_{i\in J}s_i^{(m+1)}\leq \sum_{i\in J}s_i^{(m)}.
			\end{gathered}
		\end{equation*}
		
		From now on, restrict indices to $J$, fix the convention that $0$ as an index for $\psi_i,c_{ij},a{ij}$ means $d'$, and $q_i=c_{i(i-1)}$. As $a_{i(i+1)}>0$, we will set all of them equal to 1, as it will not change the argument.\footnote{alternatively, this step may be justified by rescaling $\psi_i$} Define
		\begin{equation*}
			\begin{gathered}
				H_i(t)=\int \d x\, \bar{\phi}_i(x,t).		
			\end{gathered}
		\end{equation*}
		Using the lower bounds, we get 
		\begin{equation*}
			\begin{gathered}
				H_i(t)=\int \d x \, \phi(x,t)+\int\ d x\, \tilde{\phi}(x,t)\gtrsim t^{2-s_i^{(N+1)}}-c
			\end{gathered}
		\end{equation*}
		for $t>t_0$ with $t_0$ large enough and $c$ that depends on the $W^{1,\infty}$ norm of initial data. As $s_i^{(m_0+1)}<0$, we know that for $t_0$ sufficiently large, we can drop $c$. Furthermore, using Holder inequality, we have
		\begin{equation}\label{n H bound}
			\begin{gathered}
				\partial_t^2 H_i(t)=\int\d x\, \abs{\bar{\phi}_{i-1}}^{q_i}\gtrsim (1+t)^{-3(q_i-1)}H_{i-1}^{q_i}\gtrsim (1+t)^{(q_i-1)(-1-s_i^{(N+1)})+c_\epsilon}H_{i-1}^{1+\epsilon}\\
				\partial_t^2H_i(t)\gtrsim (1+t)^{3-q_1-q_1s_1^{(N+1)}}=(1+t)^{-s_i^{(N+2)}}
			\end{gathered}
		\end{equation}
		for sufficiently large $t$ and $c_\epsilon=\mathcal{O}(\epsilon)$. The second inequality and $s_i^{(N+2)}\leq s_i^{(N+1)}\leq1/q_{i+1}<1$ means
		\begin{equation*}
			\begin{gathered}
				\partial_tH_i\gtrsim t^{1-s_i^{(N+2)}}
			\end{gathered}
		\end{equation*}
		for $t$ sufficiently large. Therefore, there exists $t_0$ such that $H_i(t_0),\partial_tH_i(t_0)>0$.
		
		Consider the system
		\begin{equation*}
			\begin{gathered}
				\partial_t y_{2i+1}=c_i(1+t)^{-(q_i-1)(1+s_i^{(N+1)})+c_\epsilon}y_{2(i-1)}^{1+\epsilon}\\
				\partial_ty_{2i}=y_{2i+1}\\
				y_{2i}(t_0)=H_i(t_0),\, y_{2i+1}(t_0)=\partial_tH_i(t_0)
			\end{gathered}
		\end{equation*}
		where $c_i$ are the implicit constants from \cref{n H bound} and $i\in\{1,2,...,d'\}$ with the convention $y_{i+2M}=y_i$. The positivity of initial data and inequalities for $H$ (using Gronwall) imply that $H_i(t)\geq y(t)$ for $t>t_0$ as long as both exist. Then, we have
		\begin{equation*}
			\begin{gathered}
				\sum_{i=1}^M -(q_i-1)(1+s_i^{(N+1)})=-M+\Sigma_{i=1}^M (2-q_i+s_i^{(N+1)})-q_is_i^{(N+1)}\\
				=-M+\Sigma_{i=1}^M (-1+q_is_i^{(N)})=-2M+\Sigma_{i=1}^Mq_i(s^{(N)}_i-s^{(N+1)}_{i})>-2M.
			\end{gathered}
		\end{equation*}
		Using the strict inequality, we can choose $\epsilon$ sufficiently small such that the same conclusion holds with $c_\epsilon$ included. Therefore, using generalised Gronwall lemma \cref{gronwall}, we get that the $y$ system blows up in finite time. Due to the upper bounds, the same holds for $H$, which is in contradiction with global $C^2$ solution.
	\end{proof}

	\subsection{Strauss Glassey system}
	
	\begin{theorem}\label{blow up for Strauss Glassey system}
		The equation \eqref{Strauss Glassey system}:
		\begin{equation*}
			\begin{gathered}
				\Box\bar{\phi}_1=\abs{\bar{\phi}_2}^{q_1}\\
				\Box\bar{\phi}_2=\abs{\partial_t\bar{\phi}_1}^{q_2}.
			\end{gathered}
		\end{equation*}
		in $\R^{3+1}$ has no global spherically symmetric $\mathcal{C}^2$ solution if the data is compactly supported and $q_1<2,q_1q_2(q_2-2+q_2(q_1-2))<1$.
	\end{theorem}
	
	Let's split the solution into a linear
	\begin{equation*}
		\begin{gathered}
			\Box\tilde{\phi}_i=0\\
			\tilde{\phi}_i(0)=f_i,\,\partial_t\tilde{\phi}_i(0)=g_i\in\mathcal{C}^\infty_c
		\end{gathered}
	\end{equation*}
	and non-linear part
	\begin{equation*}
		\begin{gathered}
			\Box\phi_1=\abs{\phi_2}^{q_1}\\
			\Box\phi_2=\abs{\partial_t\phi_1}^{q_2}
			\phi_i=\partial_t\phi_i=0
		\end{gathered}
	\end{equation*}		
	for $i\in\{1,2\}$.

	As in \cite{john_blow-up_1979} and \cite{yang_revisit_2016}, \cref{blow up for Strauss Glassey system} follows from the lemma
	\begin{lemma}\label{compact support lemma}
		Fix $q_1<2,q_1q_2(q_2-2+q_2(q_1-2))<1$. Any spherically symmetric global solution to \cref{Strauss Glassey system}:
		\begin{equation*}
			\begin{gathered}
				\Box\bar{\phi}_1=\abs{\bar{\phi}_2}^{q_1}\\
				\Box\bar{\phi}_2=\abs{\partial_t\bar{\phi}_1}^{q_2}
			\end{gathered}
		\end{equation*}
		in $\R^{3+1}$ from data supported in ball of radius one satisfies $\supp\phi_i\subset\Gamma^-(0,1):=\{(r,t):r<1-t\}$.
	\end{lemma}
	\begin{proof}
		
		Without loss of generality suppose there is a point $(r_1,t_1)\ni\Gamma^-(0,1)$ such that $(\partial_t\phi_1,\phi_2)|_{r=r_1,t=t_1}\neq(0,0)$. Assume, that $\partial_t\phi_1\neq0$, for the other case, see \cref{creation of tails}. Using continuity of $\partial_t\phi_1$, we get that $\abs{\partial_t\phi_1}\gtrsim1$ on $B_{\epsilon}(r_1,t_1)$ for some $\epsilon$. Then, using the positivity of the solution operator (\cref{lower bound}), we get
		\begin{equation}\label{sg lower 1}
			\begin{gathered}
				\phi_2|_{t-r\in B_\epsilon(u_1)}\geq \frac{c}{\jpns{r}}
			\end{gathered}
		\end{equation}
		for $t_1+\abs{r_1}$ $c>0$. From now on, we will work in the region $\{t-r\geq u_1\}\subset\Gamma^+(r_0,t_0)$. Therefore, we can use the vanishing of the linear solution in this region to get $\abs{\phi_2}= \abs{\bar{\phi}_2}$.

		Now, we use the solution operator for the derivative (in $\R^{1+1}$) to obtain
		\begin{equation*}
			\begin{gathered}
				r\partial_t\phi_1|_{t-r= u_1}=\int_{u_1}^v\d v^\prime\,r\abs{\bar{\phi_2}(u,v^\prime)}^{q_1}+\int_{-1}^{u_1}\d u^\prime r\abs{\bar{\phi_2}(u^\prime,v)}^{q_1}-\int_{-1}^{u_1}\d u^\prime r\abs{\bar{\phi_2}(u^\prime,u)}^{q_1}
			\end{gathered}
		\end{equation*}
		where $u=t-r,v=t-r$ and we abused notation to denote $\bar{\phi}$ evaluated at specific $u,v$ coordinates.
		Note, that using \eqref{sg lower 1} the first integral does not converge as $v\to\infty$, therefore, for $v$ sufficiently large
		\begin{equation*}
			\begin{gathered}
				r\partial_t\phi_1|_{t-r\in B_\epsilon(u_1)}\gtrsim t^{1-q_1}.
			\end{gathered}
		\end{equation*}
		This estimate is strong, because $q_1<2$. 
		
		Using this improved bound on $\partial_t\phi_1$, we get
		\begin{equation*}
			\begin{gathered}
				\phi_2|_{t-r\geq u_2}\gtrsim t^{-1} \jpns{t-r}^{-s_2}
			\end{gathered}
		\end{equation*}
		where $s_2=q_2-2+q_2(q_1-2)$ for $u_2$ sufficiently large.
		
		\begin{remark}\label{creation of tails}
			It doesn't matter if we provide a seed for $\partial_t\phi_1$ or $\phi_2$, the above generation will produce the lower bounds stated above.
		\end{remark}


		It suffices to show, that a solutions with these tails cannot exist. Let's define
		\begin{equation*}
			\begin{gathered}
				H_j(t)=\int\d x \bar{\phi}_j(x,t)
			\end{gathered}
		\end{equation*}
		for $j\in\{1,2\}$. Therefore, using Holder inequality
		\begin{equation}\label{H1 bound}
			\begin{gathered}
				\partial_t^2H_1(t)=\int\d x \abs{\bar{\phi}_2}^{q_1}(x,t)\gtrsim (t+1)^{-3(q_1-1)}\Bigg(\int\d x \abs{\bar{\phi}_2(x,t)}\Bigg)^{q_1}\gtrsim (t+1)^{-3(q_1-1)}\abs{H_2(t)}^{q_1}
			\end{gathered}
		\end{equation}
		similarly
		\begin{equation}\label{H2 bound}
			\begin{gathered}
				\partial_t^2 H_2(t)=\int\d x \abs{\partial_t\bar{\phi}_1}^{q_2}(x,t)\gtrsim(1+t)^{-3(q_2-1)}\abs{\partial_tH_1(t)}^{q_2}.
			\end{gathered}
		\end{equation}
		
		The lower bound we established for $\phi_2$ gives
		\begin{equation*}
			\begin{gathered}
				H_2(t)=\int \phi_2(x,t)+\int\tilde{\phi}_2(x,t)\gtrsim \int_{t-r>u_2}\d x t^{-1}\jpns{t-r}^{-s_2}-c\gtrsim t^{2-s_2}-c
			\end{gathered}
		\end{equation*}
		where $c$ depends on initial data and $t$ sufficiently large. Since $q_1q_2s_2<1$, we see that for sufficiently large $t$, the non-linear part will overwhelm any linear contribution and we can drop $c$.
		
		Since $q_1<2$ and $q_1q_2s_2<1$ (the second being the critical curve) we have $(2-s_2)q_1-3(q_1-1)>0$, thus \eqref{H1 bound} yields
		\begin{equation*}
			\begin{gathered}
				\partial_t H_1(t)\gtrsim t^{(2-s_2)q_1-3(q_1-1)+1}
			\end{gathered}
		\end{equation*}
		for $t$ sufficiently large. In particular, there exists a time $t_0$ such that $\partial_tH_1(t_0),H_2(t_0),\partial_tH_2(t_0)>0$. Consider the system
		\begin{equation*}
			\begin{gathered}
				\partial_ty_1=c_1 t^{-3(q_1-1)+(q_1-1-\epsilon)(2-s_2)}y_2^{1+\epsilon}\\
				\partial_ty_2=y_3
				\partial_ty_3=c_2 t^{-3(q_2-1)+(q_2-1-\epsilon)(4-s_2q_1-q_1)}y_1^{1+\epsilon}\\
				y_1(t_0)=H_1(t_0),\,y_2(t_0)=H_2(t_0),\,y_3(t_0)=\partial_tH_2(t_0)
			\end{gathered}
		\end{equation*}
		with constants $c_i$ taken from \eqref{H1 bound}, \eqref{H2 bound}. Note, that the monotonicity and inequalities imply $y_1\leq H_1,\,y_2\leq H_2,\,y_3\leq \partial_tH_2$. Using generalised Gronwall lemma and 
		\begin{equation*}
			\begin{gathered}
				-3(q_1-1)+(q_1-1)(2-s_2)-3(q_2-1)+(q_2-1)(4-s_2q_1-q_1)>-3\iff q_1q_2s_2<1
			\end{gathered}
		\end{equation*}
		we know that the $y$ system blows up in finite time, therefore, the same holds for $H$.
		
	\end{proof}
	
	\subsection{A critical problem}\label{critical section}
	
	In this section, we are going to show the global ill posedness of \cref{critical problem}:
	\begin{equation*}
		\begin{gathered}
			\Box\phi_1=\abs{\phi_3}^3\\
			\Box\phi_2=(\partial_t\phi_1)^2\\
			\Box\phi_3=(\partial_t\phi_2)^2+\abs{\phi_1}^3\\
			\partial_t^j\phi_i(0)=\phi_i^{(j)}:\R^3\to\R.
		\end{gathered}
	\end{equation*}	
	\begin{theorem}\label{critical eq}
		Let $\phi_i^{(j)}$ be spherically symmetric smooth initial data with support in $\{r\leq1\}$ such that there exists $r_0\leq1$ with $(r\phi^{(1)}_1-\partial_r r\phi^{(0)}_1)|_{r=r_0}>0$.
		Than \cref{critical problem} has no global $\mathcal{C}^2$ solution.
	\end{theorem}
	\begin{remark}
		We believe that the constraint on initial data can be substantially weakened in the small data regime\footnote{Note, that small data is not actually a restriction, as one may restrict attention to the domain of dependence of $\mathcal{A}_{r_0}=\{r>r_0\}$ and using compact support and continuity, we get $\norm{\cdot}_{\mathcal{A}_{r_0}}\to0$ as $r_0\to1$. }. Indeed, by a similar analysis as shown in the appendix of \cite{keir_weak_2018}, the above system has semi-global well posedness, ie. the solution exists for some finite retarded time $u\leq u_0$. Than, the conclusion of the theorem will follow if one has some $u_1\leq u_0$ such that $\lim_{r\to\infty}\partial_t\psi_3(r,r+u_1)\neq 0$.
	\end{remark}
	\begin{corollary}\label{failure of weak null condition}
		Let $\phi_j^{(i)}$ be spherically symmetric smooth initial data with support in $\{r\leq1\}$ such that there exists $r_0\leq1$ with $(r\phi^{(1)}_1-\partial_r r\phi^{(0)}_1)|_{r=r_0}>0$ and $\phi^{(j)}_i\geq0$ for $i\in\{1,2,3\},\,j\in\{0,1\}$. Than
		\begin{equation}\label{critical problem without abs}
			\begin{gathered}
				\Box\phi_1=\phi_3^3\\
				\Box\phi_2=(\partial_t\phi_1)^2\\
				\Box\phi_3=(\partial_t\phi_2)^2+\phi_1^3\\
				\partial_t^j\phi_i(0)=\phi_j^{(i)}
			\end{gathered}
		\end{equation}
		has no global $C^2$ solution.
	\end{corollary}
	\begin{proof}
		Since the kernel of $\Box_{\R^{3+1}}$ is positive, and the initial data is non-negative, any solution of \cref{critical problem without abs} also solves \cref{critical problem}. Using the previous theorem, the result follows.
	\end{proof}
	\begin{proof}[Proof of \cref{critical eq}]
		We work by contradiction. Let's assume that there is such a solution.
		
		As always, we first want to prove lower bounds on the solution using positivity properties and then use Gronwall to exclude the possibility of such solutions. Rescaling $\phi_i=\psi_i/r$, we get
		\begin{equation*}
			\begin{gathered}
				\Box_{\R^{1+1}}\psi_1=\frac{1}{r^2}\abs{\psi_3}^3\\
				\Box_{\R^{1+1}}\psi_2=\frac{1}{r}(\partial_t\psi_1)^2\\
				\Box_{\R^{1+1}}\psi_3=\frac{1}{r}(\partial_t\psi_2)^2+\frac{1}{r^2}\abs{\psi_3}^3.
			\end{gathered}
		\end{equation*}		
		Since there exist $r_0<1$ such that $\partial_u\psi_1|_{\{t=0\}}(r_0)=(r\phi^{(1)}_1-\partial_r r\phi^{(0)}_1)|_{r_0}>0$, we get that $\partial_u\psi_1|_{\{t=0,r\in B_\epsilon(r_0)\}}>\epsilon$ for some $\epsilon>0$. Using the solution kernel for $\psi_1$ (\cref{l infinite theory}), we get
		\begin{equation*}
			\begin{gathered}
				\partial_t\psi_1(s,r+s)=\int_{0}^{s}\d x \frac{\abs{\psi_3(x,x+r)}^3}{(x+r)^2}+\int_{0}^{(1-r)/2}\frac{\abs{\psi_3(s-x,s+x+r)}^3}{(s+x+r)^2}+\partial_u\psi_1(0,r)\geq\epsilon
			\end{gathered}
		\end{equation*}
		for $r\in B_{\epsilon}(r_0)$. We may substitute this lower bound for $\psi_2$ to get
		\begin{equation*}
			\begin{gathered}
				\partial_t\psi_2(s,r+s)\geq\int_0^s\d x \frac{\epsilon^2}{x+r}+\partial_u\psi_2(0,r).
			\end{gathered}
		\end{equation*}
		Note, that the first integral does not converge, so there exists $R$ sufficiently large, such that for all $s>R,r\in B_\epsilon(r_0)$ $\partial_t\psi_2(s,r+s)\geq \epsilon^2\log((s+r)/R)$. Without loss of generality assume also that $R>10$. Using this lower bound for the nonlinear part of $\psi_3$ equation, we get
		\begin{equation*}
			\begin{gathered}
				\psi_3^{nl}(r,t)\geq\epsilon^4 \int_{r_0-\epsilon}^{r_0+\epsilon}\d u'\int_t^{t+r}\d v' \frac{\log(v'/R)^2}{v'-u'}\gtrsim \epsilon^5 \int_t^{t+r} \d v'\frac{\log^2(v'/R)}{v'}\\
				\geq \epsilon^5\bigg(\log^3(\frac{t+r}{R})-\log^3(\frac{t}{R})\bigg)\geq\epsilon^5\log^2(t/R)\big(\log(1+\frac{r}{t})-\log(1-\frac{r}{t})\big)
			\end{gathered}
		\end{equation*}
		for $t>R$ and $t-r>1$. Let's define
		\begin{equation*}
			\begin{gathered}
				H_i(t)=\int\d r \psi_i(r,t).
			\end{gathered}
		\end{equation*}
	The just derived lower bound for $\psi_3^{nl}$ tells us that for $t>R$
	\begin{equation*}
		\begin{gathered}
			\int_0^{t-1} \d r \psi_3^{nl}(r,t)\geq \int_0^{t-1}\epsilon^2\log(t/R)^2\frac{\log(1+r/t)-\log(1-r/t)}{r}\gtrsim \epsilon^2\log^2(t/R).
		\end{gathered}
	\end{equation*}
	Using the splitting of $\psi_3$ into linear and nonlinear part $\psi_3=\psi_3^{(\text{lin})}+\psi_3^{(\text{nl})}$, we see that the contribution of $\psi_3^{(\text{lin})}$ is bounded - as $\psi_3^{lin}$ is bounded with $\supp\psi_3^{lin}\in\{-1<t-r<1\}$ -  and $\psi_3^{(\text{nl})}$ has positive sign. Therefore, we conclude 
	\begin{equation*}
		\begin{gathered}
			H_3(t)\gtrsim \epsilon^2\log^2(t/R).
		\end{gathered}
	\end{equation*}
	
	Using the equation of motion \cref{critical problem} and Holder inequality, as in the previous sections, we derive the following relations
	\begin{equation}\label{critical H quantities}
		\begin{gathered}
			H_1''(t)\geq \int \d r \frac{\abs{\psi_3}^3}{r^2}\gtrsim H_3^3t^{-4}\\
			H_3''(t)\geq\int\d r \frac{\abs{\psi_1}^3}{r^2}\gtrsim H_1^3t^{-4}.
		\end{gathered}
	\end{equation}
	By positivity of the second derivative $H_3''$ and unboundedness of $H_3$, it follows by the mean value theorem that there exists $T_1>0,\epsilon_1>0$  such that $H_3'|_{t>T_1}>\epsilon_1$. This in turn implies $H_3|_{t>T_2}>(t-T_2)\epsilon_1$, for eg. $T_2=\max(T_1,R)$. Using \cref{critical H quantities}, we find that $H_1''$ is not integrable , thus $H_1'$ will grow at least logarithmically and $H_1$ linearly. In particular, there exists $T>0$ such that
	\begin{equation*}
		\begin{gathered}
			H_1'(T),H_1(T),H_3'(T),H_3(T)>0.
		\end{gathered}
	\end{equation*}
	By \cref{gronwall}, we conclude that any solution to \cref{critical H quantities} must blow up in finite time.
	\end{proof}

	\subsection{Quadratic weak null problem}

	Remember, we want to prove global ill-posedness for
	\begin{equation*}
		\begin{gathered}
			\Box\phi_1=0\\
			\Box\phi_2=\phi_1^2\\
			\Box\phi_3=\phi_2^2\\
			\Box\phi_4=\phi_3^2\\
			(1\pm\partial_v\phi_4)\Box\eta=(\partial_v\phi_4)^2
		\end{gathered}
	\end{equation*}
	with the sign choice in the last equation to be determined and initial data supported in $B=\{r\leq1\}$. Lower bounds for $\phi_i$ will be easy to establish, as the kernel is positive. In particular, we will prove that $\abs{\partial_v\phi_4}>1$ at some spacetime point, thus \cref{counter example to WNC} cannot have global smooth $\eta$. Note, that $\phi_i$ satisfy inhomogeneous wave equations, $\phi_i$ have unique global solutions.
	
	\begin{lemma}\label{phi growth}
		Given initial data for \cref{counter example to WNC} with $\norm{\phi_1|_{t=0}}_{L^\infty}\geq 1$ and $\partial_t\phi_1|_{t=0}=0$, we have 
		\begin{equation*}
			\begin{gathered}
				\frac{u^{3+\epsilon}}{v}\gtrsim_\epsilon\phi_4\gtrsim \frac{u^3}{v},
			\end{gathered}
		\end{equation*}
		for $u\geq2$.
	\end{lemma}
	\begin{proof}
		The lower bound is established using the same argument as in \cref{iterative tail generation}. The upper bounds follow from \cref{john l infty} iterated application.
	\end{proof}

	\begin{lemma}\label{lemma asymptotic of phi4}
		Fix initial data for the $\phi$ part of \cref{counter example to WNC} with non-vanishing data for $\phi_1$. The corresponding $\phi_i$ solutions are global and moreover we have the asymptotic expansion in the region $\frac{r}{t}\leq1/2$
		\begin{equation}\label{asymptotics of phi4}
			\begin{gathered}
				\phi_4(t)|_{r\leq t/2}=t^3\log^j(t)f(r/t)+\mathcal{O}(t^3 \log^{j-1}(t))
			\end{gathered}
		\end{equation}
		in $W^{1,\infty}$, where $f$ is a non-vanishing smooth function. Moreover 
		\begin{equation}\label{asymptotics of dv phi4}
			\begin{gathered}
				\partial_v\phi_4(t)|_{r\leq t/2}=t^2\log^j(t)g(r/t)+\mathcal{O}(t^2 \log^{j-1}(t))
			\end{gathered}
		\end{equation}
		for a non-vanishing smooth function $g$.
	\end{lemma}

	\begin{proof}
		The expansion \cref{asymptotics of phi4} is a consequence of polyhomogeneity of the solution, see Lemma 7.6-7.8 of \cite{hintz_stability_2020}. Using these lemmas, one infers in a similar manner as proof of Theorem 7.1 in \cite{hintz_stability_2020} the  polyhomogeneity statement $\phi_4\in\mathcal{A}_{phg}^{\mathcal{E}_+,\mathcal{E}_\scri}$ \footnote{for definition of such spaces see \cite{hintz_stability_2020}. Note that such a statement does not follow from \cite{baskin_asymptotics_2018}, as in this work the authors only consider compactly supported force term}. This in turn means that in the region $r<t/2$ one has an expansion
		\begin{equation*}
			\begin{gathered}
				\phi_4(t,r)=\sum_{i=1}^N\sum_ja_{i,j}\big(\frac{r}{t}\big) t^{\alpha_i}\log^j(t)+\mathcal{O}(t^{\alpha_0-1})
			\end{gathered}
		\end{equation*}
		with $a_{i,j}$ smooth functions. The upper bound from \cref{phi growth} implies that $\alpha_1\leq3$, while the lower bound implies $\alpha_1\geq3$ and that there is $a_{1,j}\neq0$. Keeping only the leading order term implies the \cref{asymptotics of phi4}.
		
		The leading term of $\partial_v\phi_4$ is
		\begin{equation*}
			\begin{gathered}
				(\partial_t+\partial_r)t^3\log^j(t)f(r/t)=t^2\log^{j}(t)\Big(3f+(1-\frac{r}{t})f'\Big)+\mathcal{O}(t^2\log^{j-1}(t)).
			\end{gathered}
		\end{equation*}
		Therefore, unless $3f+(1-\frac{r}{t})f'=0\implies f(\rho)=(1-\rho)^{-3}$, there exists $\rho\in[0,1/2]$ such that 
		\begin{equation*}
			\begin{gathered}
				 \partial_v\phi_4|_{r=t\rho}=ct^2\log^j(t)+\mathcal{O}(t^2\log^{j-1}(t))
			\end{gathered}
		\end{equation*}
		for $c\neq0$. We will prove that this is indeed the case.
		
		Writing $\Box$ with respect to coordinates $\tau=t,\,\rho=r/t$, we get
		\begin{equation*}
			\begin{gathered}
				\Box\phi_4=\bigg(\partial_\tau^2-\frac{2\rho}{\tau}\partial_\tau\partial_\rho-\frac{1-\rho^2}{\tau^2}\partial_\rho^2+\frac{2}{\tau^2}(\rho-\frac{1}{\rho})\partial_\rho\bigg)\phi_4=\phi_3^2.
			\end{gathered}
		\end{equation*}
		Matching the leading order terms, we get 
		\begin{equation*}
			\begin{gathered}
				6f-2(2\rho+\frac{1}{\rho})f'-(1-\rho^2)f''=F\geq0,
			\end{gathered}
		\end{equation*}
		where $F$ is the leading order behaviour of $\phi_3$, at the particular order. Substituting in $(1-\rho)^{-3}$ in place of $f$, we get
		\begin{equation*}
			\begin{gathered}
				-\frac{6(1+\rho+6\rho^2)}{\rho(1-\rho)^4}<0.
			\end{gathered}
		\end{equation*}
		This is a contradiction, thus $f(\rho)\neq(1-\rho)^{-3}$.
	\end{proof}

	As $\abs{\partial_v\phi_4}$ grows polynomially in time, provided that $\phi_1$ has non-trivial initial data, there must be a point where it reaches size 1, given not-trivial data for $\phi_1$ of any size. At this point, $\Box\eta$ cannot be a smooth function at least for one of the signs. Therefore, we conclude
	
	\begin{cor}\label{WNC cor}
		Fix smooth initial data for \cref{counter example to WNC}, such that $\phi_1$ is non-trivial. Then, at least one of the signs in \cref{counter example to WNC} admits no smooth global solutions.
	\end{cor}

	\bibliographystyle{alpha}
	\bibliography{allBib}
\end{document}